\documentclass[11pt]{article}

\usepackage[margin=1.05in]{geometry}
\usepackage[T1]{fontenc}
\usepackage{lmodern}
\usepackage[nopatch=footnote]{microtype}
\usepackage{amsmath,amsfonts,amssymb,amsthm,mathtools}
\usepackage{aliascnt,booktabs,placeins}
\usepackage{graphicx}
\usepackage{array}
\usepackage{enumitem}
\usepackage{natbib}
\usepackage[colorlinks=true,allcolors=blue]{hyperref}
\usepackage[capitalise,noabbrev]{cleveref}

\theoremstyle{plain}
\newtheorem{theorem}{Theorem}[section]
\newaliascnt{proposition}{theorem}
\newtheorem{proposition}[proposition]{Proposition}
\aliascntresetthe{proposition}
\newaliascnt{corollary}{theorem}
\newtheorem{corollary}[corollary]{Corollary}
\aliascntresetthe{corollary}
\newaliascnt{lemma}{theorem}
\newtheorem{lemma}[lemma]{Lemma}
\aliascntresetthe{lemma}
\theoremstyle{definition}
\newaliascnt{definition}{theorem}
\newtheorem{definition}[definition]{Definition}
\aliascntresetthe{definition}
\newaliascnt{remark}{theorem}
\newtheorem{remark}[remark]{Remark}
\aliascntresetthe{remark}

\newcommand{\R}{\mathbb{R}}
\newcommand{\Pp}{\mathbb{P}}
\newcommand{\E}{\mathbb{E}}
\newcommand{\Dir}{\operatorname{Dir}}
\newcommand{\Beta}{\operatorname{Beta}}
\newcommand{\Unif}{\operatorname{Unif}}
\newcommand{\one}{\mathbf{1}}
\newcommand{\dd}{\mathrm{d}}
\newcommand{\calC}{\mathcal{C}}

\newcommand{\pK}{p_{\mathrm K}}
\newcommand{\pSP}{p_{\mathrm{SP}}}
\newcommand{\Kad}{K_2^{\mathrm{ad}}}

\newcommand{\capn}{\operatorname{cap}}

\crefname{proposition}{proposition}{propositions}
\crefname{corollary}{corollary}{corollaries}
\crefname{lemma}{lemma}{lemmas}
\crefname{definition}{definition}{definitions}
\crefname{remark}{remark}{remarks}

\usepackage[onehalfspacing]{setspace}

\makeatletter
\let\old@fnsymbol\@fnsymbol
\renewcommand{\@fnsymbol}[1]{\ifnum#1=5 \ensuremath{\diamond}\else
    \old@fnsymbol{#1}\fi
}
\makeatother

\title{A triply~efficient finite-sample\\confidence interval for bounded means}
\author{
Jiahao Ming\thanks{Department of Statistics and Actuarial Science, University of Waterloo, Canada. \texttt{j5ming@uwaterloo.ca}}
\and
Aaditya Ramdas\thanks{Department of Statistics, Stanford University, USA. \texttt{aramdas@stanford.edu}}
\and
Yi Shen\thanks{Department of Statistics and Actuarial Science, University of Waterloo, Canada. \texttt{yi.shen@uwaterloo.ca}}
\and
Ruodu Wang\thanks{Department of Statistics and Actuarial Science, University of Waterloo, Canada. \texttt{wang@uwaterloo.ca}}
\and
Ian Waudby-Smith\thanks{Department of Statistics, University of California, Berkeley, USA. \texttt{ianws@berkeley.edu}}
}
\date{\today}

\begin{document}
\maketitle

\begin{abstract}
Given observations $\mathbf x \in \R^n_+$, Gaffke defined a statistic $K_n(\mathbf x)$
and conjectured that it is a p-value whenever the inputs are independent e-values.  
Recently, Vlassis and Thomas proved this conjecture. Given
iid data in $[0,1]$, this can be used to derive Gaffke's confidence interval for their mean. 
We give a nuanced finite- and large-sample account of Gaffke's p-value and interval.  We first find that  
Gaffke's p-value is inadmissible and can be improved.  For $n=2$, we construct the unique admissible  rule that dominates $K_2$. For $n\ge3$, we construct a symmetric nonincreasing strict dominator,  
and find an admissible completion mechanism for any merger.
Despite its inadmissibility, for iid data with unknown variance $\sigma^2>0$, the equal-tail Gaffke confidence interval $I_n$ is first-order
asymptotically efficient (including sharp constants) in three different senses. First, at every fixed confidence level $1-\alpha$, its asymptotic width exactly matches that of the central limit theorem (``Gaussian-efficient'').
Second, it is efficient in the regime where $\alpha \to 0$, $n \to \infty$ such that $\alpha = e^{-o(n)}$ (moderate deviations), and third, when $\log(1/\alpha)/n$ converges to a constant (large deviations).
It is the only one we know of with this property.
Our simulations also find that, among a variety of bounded-mean intervals and settings,
 Gaffke's is always the shortest.

\end{abstract}

\section{Introduction}

Testing the mean of a nonnegative distribution without imposing a parametric
model is a classical statistical problem.  Given independent nonnegative
random variables $X_1,\ldots,X_n$, consider the hypothesis 
\[
H_0:\quad \E[X_i]\le1\quad\text{for every }i.
\]
Based on the data $\mathbf x=(x_1,\dots,x_n)$,
\cite{Gaffke2005} proposed a statistic $K_n(\mathbf x)$, obtained by adjoining a zero to the observations,
averaging the resulting $n+1$ values with uniform Dirichlet weights, and
evaluating the conditional probability that the average does not exceed one; see~\eqref{eq:Kdef}. He conjectured that this is a
valid p-value; \citet{VlassisThomas2026} recently proved the conjecture in
full.

This result has important implications for another classical statistics problem: estimating the mean of a bounded random variable on $[0,1]$.
For this ``bounded mean'' problem,
one-sided Gaffke tests can be inverted into a confidence interval.
\citet{LearnedMillerThomas2020} studied this interval before the validity
conjecture had been resolved.  Its endpoints are quantiles of Dirichlet
averages, and for Bernoulli observations they coincide with the classical
Clopper--Pearson endpoints \citep{ClopperPearson1934}.  Modern betting and concentration constructions
provide many alternative fixed-time and sequential confidence bounds, such as
those of \citet{WaudbySmithRamdas2024,VO25} and
\citet{austern2022efficient}.  This makes both finite-sample performance and
large-sample efficiency natural criteria for evaluating Gaffke's confidence interval.

One set of our findings is strongly favorable. We establish three senses of asymptotic
efficiency for the Gaffke confidence interval  under iid sampling on $[0,1]$.
Write $\mu$ and $\sigma^2$ for the common mean and variance.
First, at fixed tail levels and $\sigma^2>0$, the Gaffke endpoints
are first-order equivalent to those of the usual Wald interval and the scaled width converges almost surely to the oracle
Gaussian constant.
Second, in the moderate-deviation regime, equal tail levels
$q_n=e^{-h_n}$ satisfy $h_n\to\infty$ and $h_n/n\to0$.
For $\sigma^2>0$, the sharp endpoint benchmarks are $\mu\pm r_n$, where
$r_n=\sigma\sqrt{2h_n/n}$, and the sharp width benchmark is $2r_n$. These bounds apply to every interval with coverage
at least $1-2q_n$ under every iid law on $[0,1]$, without a prescribed
tail allocation. The Gaffke confidence interval attains the normalized endpoint bounds in
probability and the normalized width bound almost surely.
Third, in the large-deviation regime, equal tail levels satisfy
$-\log q_n/n\to c\in(0,\infty)$. For $0<\mu<1$, the Gaffke endpoints
converge in probability to those of an interval of positive width,
determined by the minimum relative entropy to laws with a prescribed mean. The interval attains both endpoint and width bounds
under the same form of joint coverage requirement as above. The information
benchmark is due to \citet{DeepBassambooJuneja2026}; we give a self-contained
proof of the corresponding bounds in our setting.
Therefore, the Gaffke confidence interval   is first-order
asymptotically efficient   in {three different senses}: Gaussian efficiency, moderate-deviation efficiency, and large-deviation efficiency.
Our fixed-level numerical comparisons examine finite-sample performance.

A second set of our findings adds a  qualification in finite sample. As an equivalent statement of the result of \cite{VlassisThomas2026}, we can
view Gaffke's statistic as a rule that converts independent e-values \citep{VW21, grunwald2024safe,ramdas2025hypothesis} into one
p-value and ask whether that rule is admissible in the pointwise Wald sense:
can another valid rule be everywhere no larger and somewhere strictly
smaller (the smaller the better for p-values)?   For one input, Gaffke's p-value is
exactly Markov's conversion and is admissible.  For two inputs, it is
inadmissible.  We derive an explicit valid rule $\Kad$, and prove that $\Kad$ is the unique
admissible dominator of $K_2$. 
For every $n\ge3$, we construct the explicit symmetric coordinatewise
nonincreasing merger $M_n$, which strictly
improves $K_n$ on a nonempty open set. More generally, we give a countable
construction of an admissible completion below any symmetric monotone
merger, and apply it to $M_n$.

We also study what changes when external randomization is permitted.  Using a
single independent $U\sim\Unif(0,1)$, we construct the rule  
 $
K_n^{\Pi}(\mathbf x;U)$ that is equal to 
 $
U/\prod_{i=1}^n x_i$ when $\min_i x_i\ge1$ and equal to $
K_n(\mathbf x) $ otherwise. 
We prove that it is a valid randomized p-value, it is never larger than Gaffke's rule, and
it is strictly smaller throughout the upper orthant whenever $U<1$.  A
point-mass least-favorable argument shows that the uniform factor is locally
optimal there.  This construction is related to the uniformly randomized
Markov inequality of \citet{ramdas2026randomized}, but its patching with Gaffke's
outside the upper orthant requires a separate product-mixture argument.

Another finite-sample comparison concerns the recent SymPol method of
\citet{MingShenWang2026}.  
We prove that the Gaffke p-value is pointwise smaller than the SymPol p-value, making the latter inadmissible whenever the former is valid, although the SymPol p-value retains the advantage of
validity under the broader co-valid dependence condition.

The remainder of the paper is organized as follows. \cref{sec:setup}
introduces the two mean tests and proves
pointwise domination. \cref{sec:admissibility} presents  definitions  of mergers and 
their admissibility, 
and obtains an exact
bivariate admissible improvement, admissible completions in higher dimensions, and deterministic
and randomized improvements of Gaffke's p-value. \cref{sec:ci} derives the confidence
interval and proves its inadmissibility.
\cref{sec:asymptotics} establishes fixed-level Gaussian efficiency and
sharp endpoint and width bounds in the moderate- and large-deviation
regimes, with  complementary fixed-sample extreme-confidence
results in the appendix.
\cref{sec:simulations} reports the numerical comparisons.
\cref{sec:discussion} concludes the paper with some further discussions. 

\section{Two distribution-free mean tests}\label{sec:setup}

\subsection{The Gaffke p-value}

For $n\ge 1$ and $\mathbf x=(x_1,\ldots,x_n)\in[0,\infty)^n$, let
\begin{equation}\label{eq:Kdef}
K_n(\mathbf x)
:=
\Pp\!\left\{\sum_{i=1}^n x_iD_i\le 1\right\},
\quad
(D_0,D_1,\ldots,D_n)\sim\Dir(1,\ldots,1).
\end{equation}
Note that $D_0$ is not explicitly used; one may think about it as being multiplied by $x_0=0$ on the left hand side.
The probability in \eqref{eq:Kdef} is a function of  $\mathbf x$.
We use the convention $K_0=1$.  The  representation
\[
D_i=\frac{E_i}{\sum_{j=0}^nE_j},
\quad E_0,\ldots,E_n\stackrel{\mathrm{iid}}{\sim}\operatorname{Exp}(1),
\]
gives the equivalent form
\begin{equation}\label{eq:Kexp}
K_n(\mathbf x)=
\Pp\!\left\{\sum_{i=1}^n(x_i-1)E_i\le E_0\right\}.
\end{equation}
In particular, $K_n$ is nonincreasing in each coordinate and equals one if all
coordinates are at most one.
The main result of \citet{VlassisThomas2026} is the following.

\begin{theorem}[\cite{VlassisThomas2026}]\label{thm:VT}
If $X_1,\ldots,X_n$ are independent, nonnegative, and satisfy
$\E [X_i]\le1$ for every $i$, then
\(
\Pp\{K_n(\mathbf X)\le\alpha\}\le\alpha\) for \(0\le\alpha\le1.
\)
Thus $\pK(\mathbf X):=K_n(\mathbf X)$ is a finite-sample distribution-free p-value.
\end{theorem}

A useful way to restate  \cref{thm:VT}  puts it in the recent literature on e-values \citep[Definition 1.2]{ramdas2025hypothesis}, which we briefly explain here, adapted to our setting.
\begin{enumerate}
    \item  An \emph{e-variable} (also called an e-value in this paper) for the given probability $\Pp$  is a nonnegative random variable $E$ such that $\E[E]\le 1$, where $\E$ is the expectation under $\Pp$.  
    \item A \emph{p-variable} (also called a p-value) for $\Pp$ is a nonnegative random variable $P$ such that $\Pp(P\le \alpha)\le \alpha$ for all $\alpha \in (0,1)$. 
\end{enumerate}
\cref{thm:VT} then states that the function $K_n$ 
combines independent e-values into a p-value. We will formally define such functions in \cref{def:merger}.

Below we present a simple property of $K_n$, which will be useful for later results. 
\begin{proposition}\label{prop:append-coordinate}
For every $m\ge1$, $\mathbf{x}\in[0,\infty)^m$, and $z\ge1$,
\(
K_{m+1}(\mathbf{x},z)\ge  K_m(\mathbf{x})/ {z}.
\)
Equality holds if and only if $z=1$ or $\mathbf{x}\in[1,\infty)^m$.
\end{proposition}

\begin{proof}
Let $E_0,\ldots,E_{m+1}$ be independent unit exponential
random variables, and write
\[
S=\sum_{i=1}^m(x_i-1)E_i,
\qquad R=E_0-(z-1)E_{m+1}.
\]
The event $A=\{R\ge0\}$ has probability $1/z$ by direct computation.
By the memoryless property, conditionally on $A$, the
random variable $R$ is unit exponential and independent of $S$.
Consequently,
\[
K_{m+1}(\mathbf{x},z)
=\Pp(S\le R)
\ge \Pp(S\le R,A)
=\frac{K_m(\mathbf{x})}{z}.
\]
If $z=1$ or all $x_i\ge1$, the event $\{S\le R\}$ implies $A$,
so equality holds.
Otherwise, $\Pp(S<0)>0$, while $R$ is independent of $S$ and
has a strictly positive density on $(-\infty,0)$.
Thus $\Pp(S\le R<0)>0$, proving strict inequality.
\end{proof}

Applying \cref{prop:append-coordinate} repeatedly, 
we obtain, for $x_1,\dots,x_n\ge 1$,
\begin{equation}
\label{eq:product-rule}
    K_n(\mathbf x) = \frac{1}{\prod_{i=1}^n x_i}.
\end{equation}
Note that the product of independent e-values is an e-value, and so $1/\prod_{i=1}^n E_i$ is always a valid p-value by Markov's inequality for independent e-values $E_1,\dots,E_n$; Gaffke's formula matches this rule when the signals are large, but modifies it smartly when there are some weak signals.

\subsection{The symmetric-polynomial (SymPol) p-value}
Throughout, $[n]=\{1,\dots,n\}$. 
For $k=0,\ldots,n$, define the elementary symmetric polynomials
\[
e_k(\mathbf x)=\sum_{\substack{S\subseteq[n]\\|S|=k}}\prod_{i\in S}x_i,
\quad
A_k(\mathbf x)=\frac{e_k(\mathbf x)}{\binom nk},
\]
with $e_0=A_0=1$.  The SymPol statistic of \citet{MingShenWang2026} is
\[
M_{\mathrm{SP}}(\mathbf x)=\max_{0\le k\le n}A_k(\mathbf x),
\]
and the corresponding SymPol p-value is
\begin{equation}\label{eq:pSP}
\pSP(\mathbf x)=\frac{1}{M_{\mathrm{SP}}(\mathbf x)}.
\end{equation}
Because $A_0=1$, $p_{\rm SP}$ will be between $0$ and $1$. 

A vector $\mathbf E=(E_1,\ldots,E_n)$ of nonnegative random variables is called a
vector of \emph{co-valid e-variables} if
\[
\E[E_i\mid \mathbf E_{-i}]\le1\quad\text{almost surely for every }i.
\]
Independent nonnegative variables with means at most one are co-valid
e-variables.  The main result  of \citet{MingShenWang2026}
states that
\begin{equation*}
\Pp\!\left\{\max_{0\le k\le n}A_k(\mathbf E)\ge t\right\}\le\frac1t,
\quad t>0.
\end{equation*}
Hence \eqref{eq:pSP} is a valid p-value under co-validity, a broader
dependence model than the independence assumption in \cref{thm:VT}, but
we will be focusing on the independent setting in this paper.

\subsection{Pointwise domination}\label{sec:domination}

We next offer a comparison between the two methods above. 
The comparison below is deterministic and therefore does not depend on any
probabilistic assumptions on the observed vector.

\begin{theorem}
\label{thm:master}
For every $n\ge0$, every $\mathbf x\in[0,\infty)^n$, and every $k=0,\ldots,n$,
\begin{equation*}
K_n(\mathbf x)e_k(\mathbf x)\le\binom nk.
\end{equation*}
Equivalently, $K_n(\mathbf x)A_k(\mathbf x)\le1$ for every $k$.
\end{theorem}

\noindent The proof is given in Appendix~\ref{app:proof-master}.

The above result implies that 
for every $\mathbf x\in[0,\infty)^n$,
\begin{equation}\label{eq:pointwise}
\pK(\mathbf x)=K_n(\mathbf x)\le
\frac{1}{\max_{0\le k\le n}A_k(\mathbf x)}
=\pSP(\mathbf x).
\end{equation}
Consequently, for every $\alpha\in[0,1]$,
\[
\{\mathbf x:\pSP(\mathbf x)\le\alpha\}
\subseteq
\{\mathbf x:\pK(\mathbf x)\le\alpha\}.
\]
In particular, under independent nonnegative observations, where both
procedures are valid, the Gaffke test is uniformly at least as powerful as the SymPol test.

The qualification concerning independence is essential: \eqref{eq:pointwise}
is a pointwise fact, but \cref{thm:VT} does not assert validity of $\pK$ under
weaker dependence assumptions, such as co-validity.  
\citet{MingShenWang2026} also consider the optimized-betting or KL-inf statistic
\[
M_{\mathrm{bet}}(\mathbf x)=\sup_{0\le\lambda\le1}
\prod_{i=1}^n(1-\lambda+\lambda x_i).
\]
They show 
$M_{\mathrm{bet}}(\mathbf x)\le\max_kA_k(\mathbf x)$.  Thus, wherever all three procedures
are valid, their p-values satisfy
\begin{equation}
    \label{eq:chain-compare}
\pK(\mathbf x)\le\pSP(\mathbf x)\le p_\mathrm{bet}(\mathbf x),\quad \mbox{where~} p_\mathrm{bet}(\mathbf x) = \frac{1}{M_{\mathrm{bet}}(\mathbf x)}.
\end{equation} 
We note that the outer inequality $\pK(\mathbf x)\le p_\mathrm{bet} (\mathbf x)$ was already known to~\citet{Gaffke2005}, so our results yield a finer story.

Elementary calculations show that if $x_i\le1$ for every $i$, then $\pK(\mathbf x)=\pSP(\mathbf x)=p_\mathrm{bet}(\mathbf x)=1$, and if $x_i\ge1$ for every $i$, then
$\pK(\mathbf x)=\pSP(\mathbf x)=p_\mathrm{bet}(\mathbf x)=  (\prod_{i=1}^n x_i)^{-1}$; this also follows by using \eqref{eq:product-rule}. Thus, the three p-values agree in the two extreme cases of data.

\section{The Gaffke p-value as an e-to-p merger}
\label{sec:admissibility}

The pointwise domination in \eqref{eq:pointwise} compares the Gaffke p-value with a
particular competitor, but it does not answer whether the Gaffke p-value itself can be
improved.  In this section, we study admissibility among all deterministic rules that
combine independent e-values.  
External randomization is considered in \cref{sec:randomized-mergers}. 

\subsection{E-to-p mergers and admissibility}
We first define e-to-p mergers. 
The following concept is introduced in \citet[Appendix G]{VW21} under the term ``ie-to-p merging functions'', where ``i'' stands for ``independent''. 
Since our theory exclusively concerns independent e-values, we drop ``i''  here and simply call them e-to-p mergers.  
We also drop   ``e-to-p''  in many places.

\begin{definition}[E-to-p mergers]\label{def:merger}
A Borel function $F:[0,\infty)^n\to[0,1]$ is an (e-to-p)
merger  if
\begin{equation}
    \label{eq:merger-def}
    \Pp( F(E_1,\ldots,E_n)\le\alpha)\le\alpha,
\quad 0\le\alpha\le1,
\end{equation} 
for every collection of independent e-variables, meaning independent
nonnegative random variables with $\E[ E_i]\le1$.  A    merger $G$ dominates
$F$ if $G\le F$ pointwise, and strictly dominates it if strict inequality
holds somewhere.  A   merger is admissible if it has no strict dominator.
\end{definition}

Here and throughout, inequalities between functions are pointwise, and admissibility is in the usual sense of  Wald, used
for merging functions; see, for example, \citet{vovk2022admissible} in the case of merging p-values.  We sometimes call an e-to-p merger a \emph{valid rule}, when we discuss its validity.  

 By definition, for any e-to-p merger $F$ and independent e-variables $E_1,\dots,E_n$,  the  random variable 
$P=F(E_1,\dots,E_n)$ is a p-variable.
Therefore, characterizing mergers is equivalent to characterizing p-values for testing the hypothesis  
$$H_0:\quad \E[X_i]\le1\quad\text{for every }i $$ for independent nonnegative random variables.
\citet[Appendix G]{VW21} also studied \emph{e-to-p merging functions} that satisfy \eqref{eq:merger-def} for arbitrarily dependent  e-variables, which is different from this paper.

It follows from \cref{thm:VT} that $K_n=p_{\mathrm K}$ is a merger,
and so are  $\pSP  $ and $ p_\mathrm{bet} $ by \eqref{eq:chain-compare}. 
For admissibility issues, it suffices to analyze $K_n$ due to the domination in \eqref{eq:chain-compare}. 
For one input,
\begin{equation*}
K_1(x)=\min\{1,1/x\},
\end{equation*}
with $1/0=\infty$. This merger's validity is an application of Markov's inequality, and it is the only admissible merger that dominates every other merger (\citet[Proposition 2.2]{VW21}).

For two inputs,  something interesting starts to happen. For  $0\le x<1<y$, let 
\begin{equation*}
\tau(x,y)=\tau(y,x)=\frac{y(1+x)+\sqrt{y^2(1+x)^2-4xy}}2.
\end{equation*}
Equivalently, $\tau(x,y)$ is the larger root $t$ of
\begin{equation}\label{eq:ad-tau-poly}
t^2-(yx  + \max\{x,y\})t +xy=0.
\end{equation} 
Define $\Kad:[0,\infty)^2\to[0,1]$  by
\begin{equation*}
\Kad(x,y)=
\begin{cases}
1, & x,y\le1,\\[1.5mm]
\dfrac1{xy}, & x,y\ge 1,~xy\ne 1\\[2mm]
\dfrac{2}{\tau(x,y)}-\dfrac{1}{\tau(x,y)^2},
   &  \mbox{otherwise}.
\end{cases}
\end{equation*}
The function $\Kad$ is continuous. 
The following theorem characterizes its validity and optimality.

\begin{theorem} \label{thm:Kad-optimality}
The function $\Kad$ is a
  nonincreasing e-to-p merger, and
  it dominates  any e-to-p merger $F$ satisfying
$F\le K_2$. Also, $\Kad(x,y)< K_2(x,y)$ when $0<\min\{x,y\}<1<\max\{x,y\}$.
\end{theorem}

\noindent
Consequently, $\Kad$ is an admissible merger, and it is the unique admissible merger among
all valid mergers that dominate $K_2$. 
\cref{thm:Kad-optimality}   fully settles the  admissibility and improvement issues of $K_n$ for $n=2$. The situation for $n\ge 3$ is much more complicated, and is analyzed in \cref{sec:open-set-improvement}.

The proof is given in Appendix~\ref{app:proof-Kad-optimality}.
Sharp caps and the maximum pointwise improvement over $K_2$ are described
in Appendix~\ref{app:proof-ad-maxgap}.

\subsection{Admissible completion of mergers in higher dimension}
\label{sec:open-set-improvement}
\label{sec:sm025-improvement}

We next study admissibility  of $K_n$  in higher dimensions and its deterministic  improvements.
The conclusion is that $K_n$ is not admissible and can be  deterministically  improved.
Nevertheless, based on the asymptotic optimality results in \cref{sec:asymptotics}  and  simulation results in Appendix~\ref{app:sm025-simulation}, the  improvements 
do not lead to substantially better confidence intervals in typical statistical applications, and thus they mainly serve theoretical purposes for a better understanding of $K_n$ and the domination structure of mergers. 

Let $\mathcal P_n$ be the class of product laws on $[0,\infty)^n$ whose
marginals have means at most one. For a Borel set $A\subseteq[0,\infty)^n$,
define its product-null capacity by
\begin{equation*}
 \capn_n(A)=\sup_{\Pp\in\mathcal P_n}\Pp(A).
\end{equation*}
For a Borel rule $F:[0,\infty)^n\to[0,1]$, write
\begin{equation*}
 R_F(\alpha)=\{\mathbf{x}\in[0,\infty)^n:F(\mathbf{x})\le\alpha\},
 \qquad 0\le\alpha\le1.
\end{equation*}

\begin{proposition}\label{prop:ad-capacity}
A Borel function $F:[0,\infty)^n\to[0,1]$ is an e-to-p merger if and only if
\[
 \capn_n(R_F(\alpha))\le\alpha
 \quad\text{for every }\alpha\in[0,1].
\]
\end{proposition}

\begin{proof}
This is \eqref{eq:merger-def} with a supremum over all product null laws.
\end{proof}

The extreme-component decomposition in \cref{lem:ad-mixture}
 reduces this capacity to
finite-support product laws. Namely, let
\(
 \mathcal{C}_1=
 \{\mathbb{D}_a:0\le a\le1\}
 \cup\{\mathbb{D}_{a,b}:0\le a<1<b\},
\)
where $\mathbb{D}_a$ denotes the point mass at $a$ and
\[
 \mathbb{D}_{a,b}
 =\frac{b-1}{b-a}\mathbb{D}_a+\frac{1-a}{b-a}\mathbb{D}_b,
 \qquad 0\le a<1<b,
\]
is the mean-one two-point law from \eqref{eq:ad-Qab}.
Then, for every Borel set $A$,
\begin{equation}\label{eq:ad-capacity-extreme}
 \capn_n(A)=
 \sup_{\mathbb{C}_1,\ldots,\mathbb{C}_n\in\mathcal{C}_1}
 (\mathbb{C}_1\otimes\cdots\otimes \mathbb{C}_n)(A).
\end{equation}
Indeed, every product null is a mixture of the product laws on the right,
and every such product law belongs to $\mathcal P_n$.

We consider three classes of competitors: all Borel rules, Borel
coordinatewise nonincreasing rules, and Borel symmetric coordinatewise
nonincreasing rules. For $\mathbf{x}\in[0,\infty)^n$, let
\[
 \uparrow \mathbf{x}=\{\mathbf{y}\in[0,\infty)^n:\mathbf{y}\ge \mathbf{x}\},
\]
where vector inequalities are coordinatewise. For a fixed choice of class,
define
\begin{equation}\label{eq:ad-forced-set}
 D_{\mathbf{x}}=
 \begin{cases}
  \{\mathbf{x}\},&\text{for all Borel rules},\\
  \uparrow \mathbf{x},&\text{for nonincreasing rules},\\
  U_{\mathbf{x}}:=\displaystyle\bigcup_{\pi\in S_n}\uparrow(\pi \mathbf{x}),
    &\text{for symmetric nonincreasing rules},
 \end{cases}
\end{equation}
where $S_n$ is the permutation group. For a rule $F$ in the selected class
and $t\in[0,1]$, define
\begin{equation*}
 (T_{\mathbf{x},t}F)(\mathbf{y})=
 \begin{cases}
   \min\{F(\mathbf{y}),t\},&\mathbf{y}\in D_{\mathbf{x}},\\
   F(\mathbf{y}),&\mathbf{y}\notin D_{\mathbf{x}}.
 \end{cases}
\end{equation*}

\begin{proposition}
\label{thm:sm-cap-local}
Fix one of the three classes above, and let $F$ be an e-to-p merger in
that class. The rule $T_{\mathbf{x},t}F$ belongs to the same class and satisfies
$T_{\mathbf{x},t}F\le F$. If $t\ge F(\mathbf{x})$, then $T_{\mathbf{x},t}F=F$.
For $0\le t<F(\mathbf{x})$, it is an e-to-p merger if and only if
\begin{equation}\label{eq:sm-cap-validity}
 \capn_n\bigl(R_F(\alpha)\cup D_{\mathbf{x}}\bigr)\le\alpha
 \quad\text{for every }\alpha\in[t,F(\mathbf{x})).
\end{equation}
Moreover, $F$ is admissible among e-to-p mergers in the selected class
if and only if, for every $\mathbf{x}\in[0,\infty)^n$ and every $t\in[0,F(\mathbf{x}))$,
there exists $\alpha\in[t,F(\mathbf{x}))$ such that
\begin{equation}\label{eq:sm-local-admissibility}
 \capn_n\bigl(R_F(\alpha)\cup D_{\mathbf{x}}\bigr)>\alpha.
\end{equation}
\end{proposition}

\begin{proof}
Each $D_{\mathbf{x}}$ is closed; in the monotone classes it is upward closed, and
in the symmetric class it is permutation invariant. Thus the cap preserves
Borel measurability and the structural properties of the selected class.
Also, $F(\mathbf{y})\le F(\mathbf{x})$ on $D_{\mathbf{x}}$, so $T_{\mathbf{x},t}F=F$ when $t\ge F(\mathbf{x})$.
Otherwise its rejection regions agree with those of $F$ outside $[t,F(\mathbf{x}))$,
while
\[
 R_{T_{\mathbf{x},t}F}(\alpha)=R_F(\alpha)\cup D_{\mathbf{x}},
 \qquad t\le\alpha<F(\mathbf{x}).
\]
Thus \cref{prop:ad-capacity} gives \eqref{eq:sm-cap-validity}.

A valid nontrivial cap strictly dominates $F$ in the selected class.
Conversely, if an e-to-p merger $G\le F$ in that class satisfies
$G(\mathbf{x})=t<F(\mathbf{x})$, then $G\le t$ on $D_{\mathbf{x}}$, so $G\le T_{\mathbf{x},t}F\le F$.
Increasing a valid p-value preserves validity, and therefore this cap is
valid. Consequently, admissibility is equivalent to invalidity of every
nontrivial cap, giving \eqref{eq:sm-local-admissibility}.
\end{proof}

For the rest of this subsection, we work in the symmetric monotone class,
so $D_{\mathbf{x}}=U_{\mathbf{x}}$. These caps provide a countable construction of an admissible
merger below any initial merger in this class.

\begin{theorem}
\label{thm:sm-admissible-completion}
Let $F_0$ be a symmetric coordinatewise nonincreasing e-to-p merger.
Enumerate
$\mathbb Q_{>0}^n\times(\mathbb Q\cap(0,1))$
as $\{(\mathbf{u}_j,t_j):j\ge1\}$, with each pair occurring once, and define
\begin{equation}\label{eq:sm-admissible-recursion}
 F_j=
 \begin{cases}
 T_{\mathbf{u}_j,t_j}F_{j-1},&\text{if this is an e-to-p merger},\\
 F_{j-1},&\text{otherwise}.
 \end{cases}
\end{equation}
Set
\begin{equation}\label{eq:sm-admissible-limit}
 L(\mathbf{x})=\inf_{j\ge0}F_j(\mathbf{x}),\qquad
 F_\star(\mathbf{x})=\lim_{h\downarrow0}L(\mathbf{x}+h\one).
\end{equation}
Then $F_\star$ is a lower semicontinuous symmetric coordinatewise
nonincreasing e-to-p merger, is admissible in the symmetric monotone class,
and satisfies $F_\star\le F_0$.
\end{theorem}

\noindent The proof is given in Appendix~\ref{app:proof-sm-admissible-completion}.

For every nontrivial proposed cap, \eqref{eq:sm-cap-validity} gives the exact
acceptance criterion in \eqref{eq:sm-admissible-recursion}.
We next construct an explicit improvement of $K_n$, and applying the above construction to it will lead to an admissible dominator of $K_n$. 
Its cap level $1/40$ is chosen so that the improvement also appears in a
$95\%$ equal-tailed confidence interval.

For $n\ge3$, put
\[
 m=n-2,\quad q=\frac1{40},\quad
 a_n=\frac{163}{273}-\frac1{80m},\quad \mathbf{z}_n=(a_n,\underbrace{1,\ldots,1}_{n-2},56).
\]

\begin{proposition}
\label{thm:sm025-seed}
For every $n\ge3$, the rule
\(
 M_n:=T_{\mathbf{z}_n,q}K_n
\)
is a symmetric coordinatewise nonincreasing e-to-p merger, satisfies
$M_n\le K_n$, and is strictly smaller than $K_n$ on a nonempty open subset
of $(0,\infty)^n$. Consequently, $K_n$ is inadmissible even within the
symmetric monotone class.
\end{proposition}

\noindent The proof is given in Appendix~\ref{app:proof-sm025-seed}.

\begin{remark}
The construction in \cref{thm:sm-admissible-completion} requires a choice
of enumeration, and the resulting admissible rule may depend on its order.
We can construct two different admissible rules below the same initial merger.
An explicit example is given in Appendix~\ref{app:distinct-admissible-completions}.
\end{remark}

Every e-to-p merger $F\le K_n$ satisfies
$\capn_n(R_F(\alpha))=\alpha$ for all $\alpha\in[0,1]$.
Thus exact calibration does not imply admissibility.
Appendix~\ref{app:self-calibration-witnesses} develops self-calibration
and least-favorable witness criteria, including a criterion for the least
valid rule below a given merger.

\subsection{Randomized mergers}
\label{sec:randomized-mergers}

We now enlarge the class of competitors by allowing one auxiliary random
variable $U\sim\Unif(0,1)$, independent of all inputs.  Randomized
improvements of Markov's inequality are studied systematically by
\citet{ramdas2026randomized}.  In contrast to the deterministic notion in
\cref{sec:open-set-improvement}, a competing rule may now depend on the external
randomness $U$.

\begin{definition}[Randomized e-to-p mergers and randomized admissibility]
\label{def:randomized-merger}
A Borel function
$
Q:[0,\infty)^n\times[0,1]\to[0,1]
$
is a randomized e-to-p merger if
 $
\Pp\{Q(E_1,\ldots,E_n;U)\le\alpha\}\le\alpha
 $
for every product null and every $\alpha\in[0,1]$.

Let $\mathbb{L}$ denote Lebesgue measure on $[0,1]$.  A randomized merger $Q$
dominates a deterministic merger $F$ if
$
Q(\mathbf{x};u)\le F(\mathbf{x})
$
for every $\mathbf{x}\in[0,\infty)^n$ and for $\mathbb{L}$-almost every $u\in[0,1]$.
It strictly dominates $F$ if, in addition, there exists
$\mathbf{x}\in[0,\infty)^n$ such that
\[
\mathbb{L}\bigl\{u\in[0,1]:Q(\mathbf{x};u)<F(\mathbf{x})\bigr\}>0.
\]
A deterministic merger $F$ is randomized-admissible if no randomized merger
strictly dominates it; otherwise, it is randomized-inadmissible.
\end{definition}

Thus randomized admissibility is stronger than deterministic admissibility:
the class of potential dominators is enlarged by allowing external
randomization.  To express validity and domination in terms of ordinary
functions of the input vector, define, for any Borel function
$Q:[0,\infty)^n\times[0,1]\to[0,1]$, its conditional level-$\alpha$
rejection probability by
\begin{equation*}
\phi^Q_\alpha(\mathbf{x})
=
\int_0^1\one\{Q(\mathbf{x};u)\le\alpha\}\,\dd u.
\end{equation*}

\begin{proposition}
\label{prop:ad-randomized-validity}
A Borel function $Q:[0,\infty)^n\times[0,1]\to[0,1]$ is a randomized
e-to-p merger if and only if
\begin{equation}\label{eq:ad-fractional-capacity}
\sup_{\Pp\in\mathcal P_n}
\int\phi^Q_\alpha(\mathbf{x})\,\Pp(\dd \mathbf{x})
\le\alpha
\quad\text{for every }\alpha\in[0,1],
\end{equation}
where   $\mathcal P_n$ is the class of product laws on $[0,\infty)^n$ whose
marginals have means at most one. 
If $Q$ dominates a deterministic merger $F$, then
 $
\phi^Q_\alpha(\mathbf{x})\ge\one\{F(\mathbf{x})\le\alpha\}
$
for every $\mathbf{x}$ and every $\alpha$.  Consequently, a sufficient condition for
a deterministic merger $F$ to be randomized-admissible is that, for every
$\alpha$, the indicator of $R_F(\alpha)$ is pointwise maximal among all
 tests $\phi:[0,\infty)^n\to[0,1]$ satisfying
\eqref{eq:ad-fractional-capacity}.
\end{proposition}

\noindent The proof is given in Appendix~\ref{app:proof-ad-randomized-validity}.
The following elementary obstruction is useful both deterministically and
randomly.  With the convention $1/0=\infty$, define
\begin{equation}\label{eq:ad-atom-capacity}
m_n(\mathbf{x})=\prod_{i=1}^n\min\{1,1/x_i\}.
\end{equation}

\begin{lemma} 
\label{lem:ad-point-mass}
For every $\mathbf{x}$, the largest mass that a product null can assign to the
singleton $\{\mathbf{x}\}$ is $m_n(\mathbf{x})$.  Hence every deterministic     merger
satisfies
$
F(\mathbf{x})\ge m_n(\mathbf{x}),
$
and every valid randomized merger satisfies
\begin{equation}\label{eq:ad-randomized-atom-bound}
m_n(\mathbf{x})\,\phi^Q_\alpha(\mathbf{x})\le\alpha
\quad\text{for every }\alpha\in[0,1].
\end{equation}
\end{lemma}

\begin{proof}
For a marginal null, an atom at $x_i>1$ has probability at most $1/x_i$ by
the mean constraint, while an atom at $x_i\le1$ can have probability one.
These bounds are attained by the deterministic law at $x_i$ when $x_i\le1$
and by the two-point law on $\{0,x_i\}$ when $x_i>1$.  Independence gives
\eqref{eq:ad-atom-capacity}.  Applying validity to this product law proves
 $F(\mathbf{x})\ge m_n(\mathbf{x}) $and
\eqref{eq:ad-randomized-atom-bound}.
\end{proof}

On the upper orthant, the Gaffke p-value attains this point-mass lower bound. Indeed, if
$x_i\ge1$ for all $i$, the exponential representation \eqref{eq:Kexp} gives
\begin{equation}\label{eq:ad-K-product-orthant}
K_n(\mathbf{x})=
\prod_{i=1}^n\E \left[e^{-(x_i-1)E_i}\right]
=
\frac1{\prod_{i=1}^n x_i}
=
m_n(\mathbf{x}).
\end{equation}
This makes a uniformly randomized Markov improvement possible on the entire
upper orthant.

 The randomized product-orthant rule  is defined as: 
\begin{equation}\label{eq:ad-KPi}
K_n^{\Pi}(\mathbf{x};u)=
\begin{cases}
\displaystyle \frac{u}{\prod_{i=1}^n x_i},
&\min_i x_i\ge1,\\[3mm]
K_n(\mathbf{x}),&\min_i x_i<1.
\end{cases}
\end{equation} 

\begin{theorem}
\label{thm:ad-KPi-validity}
For every $n\ge1$, $K_n^{\Pi}$ is a valid randomized independent e-to-p
merger.  Moreover,
 $
K_n^{\Pi}(\mathbf{x};u)\le K_n(\mathbf{x})
 $
for every $\mathbf{x}$ and $u$, with strict inequality whenever
$\min_i x_i\ge1$ and $u<1$. Thus $K_n$ is inadmissible once external
randomization is allowed, including when $n=1$.
\end{theorem}

\noindent The proof is given in Appendix~\ref{app:proof-ad-KPi-validity}.

The randomized factor in \eqref{eq:ad-KPi} is pointwise sharp on the upper
orthant.

\begin{proposition}  
\label{prop:ad-KPi-optimality}
Fix $\mathbf{x}\in[1,\infty)^n$ and put
 $
k=K_n(\mathbf{x})=1/\prod_{i=1}^n x_i.
$
For every valid randomized merger $Q$,
\begin{equation}\label{eq:ad-upper-cdf-bound}
\phi^Q_\alpha(\mathbf{x})\le\min\{1,\alpha/k\},
\quad 0\le\alpha\le1.
\end{equation}
The random variable $Uk$ attains equality in
\eqref{eq:ad-upper-cdf-bound}. Hence it is stochastically the smallest
possible randomized p-value at the fixed point $\mathbf{x}$.  If $u\mapsto Q(\mathbf{x};u)$ is
nondecreasing, then necessarily
\[
Q(\mathbf{x};u)\ge uk
\quad\text{for Lebesgue almost every }u.
\]
\end{proposition}

\begin{proof}
On the upper orthant, $m_n(\mathbf{x})=k$ by
\eqref{eq:ad-K-product-orthant}, so
\eqref{eq:ad-upper-cdf-bound} is exactly
\eqref{eq:ad-randomized-atom-bound}.  Since $Uk$ is uniform on $[0,k]$, its
CDF attains the bound.  The final assertion is the corresponding quantile
inequality.
\end{proof}

\section{Confidence intervals for a bounded mean}\label{sec:ci}

We now move from tests to confidence intervals, which are of course highly related dual problems. Still, some subtleties arise in such a transition.  

The preceding results concern a fixed vector of e-values.  Inverting a test to
obtain a confidence interval is a more structured operation: one evaluates a
parameter-indexed family of scaled e-value vectors and must preserve nestedness
of the retained parameter set.  Pointwise test domination alone does not say
how far, or even whether, either numerical endpoint moves; a setwise
refinement may leave the closure unchanged. 
For $n=2$, the continuous coordinatewise nonincreasing merger $\Kad$
from \cref{sec:admissibility} can be inverted into a valid confidence
interval and gives a strict refinement for some samples; see
\cref{sec:CI-inadmissibility}.  

For $n\ge3$, the symmetric coordinatewise nonincreasing merger $M_n$ from \cref{thm:sm025-seed} admits an explicit inversion. At the $95\%$ equal-tailed level, this inversion is everywhere nonempty and has strictly smaller length than the Gaffke interval on an open set of samples; see \cref{sec:CI-inadmissibility}. The admissible completion in \cref{thm:sm-admissible-completion} is an existence result and is not needed for this explicit interval improvement.

Let $X_1,\ldots,X_n$ be independent random variables supported on $[0,1]$ and
having a common mean $\mu$.  Identical distributions are not needed for the
finite-sample validity in this section.  Let
$\mathbf x=(x_1,\ldots,x_n)$ denote the observed sample and write
$\overline x_n=n^{-1}\sum_{i=1}^n x_i$.

\subsection{Two one-sided inversions}

For a candidate $\theta\in(0,1)$, define
\begin{align*}
 p_{\le}(\theta;\mathbf x)
 &:=K_n\!\left(\frac{x_1}{\theta},\ldots,\frac{x_n}{\theta}\right),
 \\
 p_{\ge}(\theta;\mathbf x)
 &:=K_n\!\left(\frac{1-x_1}{1-\theta},\ldots,
                 \frac{1-x_n}{1-\theta}\right).
\end{align*}
The first is valid for testing $H_{0,\le}(\theta):\mu\le\theta$, and the second
is valid for testing $H_{0,\ge}(\theta):\mu\ge\theta$.

For the same Dirichlet vector as in \eqref{eq:Kdef}, define
\begin{equation*}
S_{\mathbf x}:=\sum_{i=1}^n x_iD_i,
\quad
T_{\mathbf x}:=D_0+\sum_{i=1}^n x_iD_i=D_0+S_{\mathbf x}.
\end{equation*}
Then
\begin{equation}\label{eq:pCDF}
p_{\le}(\theta;\mathbf x)=\Pp_{\mathbf D}(S_{\mathbf x}\le\theta),
\quad
p_{\ge}(\theta;\mathbf x)=\Pp_{\mathbf D}(T_{\mathbf x}\ge\theta).
\end{equation}
The second identity follows from
$\sum_i(1-x_i)D_i=1-D_0-S_{\mathbf x}=1-T_{\mathbf x}$.

For a real random variable $Y$, write
\[
Q_q(Y):=\inf\{t:\Pp(Y\le t)\ge q\},\quad 0\le q\le1,
\]
with the usual endpoint conventions: $Q_0(Y)$ is the essential infimum and
$Q_1(Y)$ is the essential supremum.  Choose tail levels
$\alpha_{\mathrm L},\alpha_{\mathrm U}\ge0$ such that
 $
\alpha_{\mathrm L}+\alpha_{\mathrm U}\le\alpha<1.
 $
A standard choice is $\alpha_{\mathrm L}=\alpha_{\mathrm U}=\alpha/2$.
Because each one-sided test rejects when its corresponding p-value is at
most the allocated tail level, the inversion retains an interior candidate
only when both p-values are strictly larger than their tail levels:
\begin{equation}\label{eq:strict-inversion}
\calC_{1-\alpha}(\mathbf x)\cap(0,1)
=
\left\{\theta\in(0,1):
 p_{\le}(\theta;\mathbf x)>\alpha_{\mathrm L},\ 
 p_{\ge}(\theta;\mathbf x)>\alpha_{\mathrm U}
\right\}.
\end{equation}
Following the endpoint convention in
\citet[Corollary~4.1 and Proposition~4.2]{MingShenWang2026}, we use open
endpoints whenever possible.  The boundary point $0$ is included only for the
all-zero sample, and $1$ only for the all-one sample.

\begin{theorem}
\label{thm:CI}
Let $\alpha_{\mathrm L}+\alpha_{\mathrm U}\le\alpha<1$.  Define
\begin{equation*}
L_{\alpha_{\mathrm L}}(\mathbf x)
=Q_{\alpha_{\mathrm L}}(S_{\mathbf x}),
\quad
U_{\alpha_{\mathrm U}}(\mathbf x)
=Q_{1-\alpha_{\mathrm U}}(T_{\mathbf x}),
\end{equation*}
and let
\begin{equation}\label{eq:CIset}
\calC_{1-\alpha}(\mathbf x)
=
\begin{cases}
[0,U_{\alpha_{\mathrm U}}(\mathbf x)),
    & \mathbf x=\mathbf 0,\\
(L_{\alpha_{\mathrm L}}(\mathbf x),1],
    & \mathbf x=\mathbf 1,\\
(L_{\alpha_{\mathrm L}}(\mathbf x),
 U_{\alpha_{\mathrm U}}(\mathbf x)),
    & \text{otherwise}.
\end{cases}
\end{equation}
Equivalently, the interval is open when
$\overline x_n\notin\{0,1\}$ and is half-open at the corresponding parameter-space
boundary when $\overline x_n=0$ or $1$.  Then
\begin{equation}\label{eq:coverage}
\Pp\!\left\{\mu\in\calC_{1-\alpha}(\mathbf X)\right\}
\ge1-\alpha_{\mathrm L}-\alpha_{\mathrm U}
\ge1-\alpha.
\end{equation}
Moreover,
\(
L_{\alpha_{\mathrm L}}(\mathbf x)
<U_{\alpha_{\mathrm U}}(\mathbf x)
\text{ for every }\mathbf x\in[0,1]^n,
\)
so the confidence interval is always nonempty.
\end{theorem}

\noindent The proof is given in Appendix~\ref{app:proof-CI}.
The usual equal-tail interval takes
$\alpha_{\mathrm L}=\alpha_{\mathrm U}=\alpha/2$.
A one-sided upper confidence interval uses $\alpha_{\mathrm L}=0$ and
$\alpha_{\mathrm U}=\alpha$; a one-sided lower interval is analogous.  In
both cases the endpoint inclusion is still governed by \eqref{eq:CIset}.

\subsection{The Learned-Miller--Thomas bound}

Let $z_1\le\cdots\le z_n$ be the order statistics of the sample, and let
$0=V_0\le V_1\le\cdots\le V_n\le V_{n+1}=1$, where
$V_1,\ldots,V_n$ are uniform order statistics.  Their spacings
 $
(V_1-V_0,\ldots,V_{n+1}-V_n)
$ 
have the $\Dir(1,\ldots,1)$ distribution.  The random induced mean in
\citet{LearnedMillerThomas2020} is
\[
m(\mathbf z,\mathbf V)=\sum_{i=1}^{n+1}z_i(V_i-V_{i-1}),
\quad z_{n+1}=1.
\]
By exchangeability of the Dirichlet spacings,
\begin{equation}\label{eq:LMTidentity}
m(\mathbf z,\mathbf V)\ \stackrel{d}{=}\ D_0+\sum_{i=1}^n x_iD_i=T_{\mathbf x}.
\end{equation}

\begin{corollary}
\label{cor:LMT}
The upper endpoint of the one-sided interval of
\citet{LearnedMillerThomas2020} is exactly
$ 
U_{\alpha}(\mathbf x)=Q_{1-\alpha}(T_{\mathbf x}).
$ 
The one-sided confidence interval obtained from \eqref{eq:CIset} by taking
$\alpha_{\mathrm L}=0$ and $\alpha_{\mathrm U}=\alpha$ has finite-sample
coverage at least $1-\alpha$ under iid sampling from every distribution on
$[0,1]$.  The corresponding lower endpoint is
 $
L_\alpha(\mathbf x)=1-U_\alpha(\mathbf 1-\mathbf x).
$
\end{corollary}

As already noted by \citet{VlassisThomas2026}, their validity theorem
resolves the general coverage conjecture in \citet{LearnedMillerThomas2020}.
Identity \eqref{eq:LMTidentity} makes explicit that the original
order-statistic construction and the Gaffke inversion are two representations
of the same bound.

\subsection{Bernoulli data and Clopper--Pearson endpoints}

Suppose the sample contains $k$ ones and $n-k$ zeros.  By the aggregation
property of Dirichlets,
\[
S_{\mathbf x}\sim\Beta(k,n-k+1),
\quad
T_{\mathbf x}\sim\Beta(k+1,n-k),
\]
with the natural degenerate conventions at $k=0$ and $k=n$.  Therefore
\begin{align*}
L_{\alpha_{\mathrm L}}(\mathbf x)
&=Q_{\alpha_{\mathrm L}}\{\Beta(k,n-k+1)\},\\
U_{\alpha_{\mathrm U}}(\mathbf x)
&=Q_{1-\alpha_{\mathrm U}}\{\Beta(k+1,n-k)\}.
\end{align*}

\begin{corollary}\label{cor:CP}
For Bernoulli observations, the numerical endpoints of the Gaffke interval
are exactly the usual Clopper--Pearson endpoints with tail allocation
$(\alpha_{\mathrm L},\alpha_{\mathrm U})$.  Under the open/half-open
convention \eqref{eq:CIset}, the confidence interval is
\[
\begin{cases}
(L_{\alpha_{\mathrm L}},U_{\alpha_{\mathrm U}}),&0<k<n,\\
[0,U_{\alpha_{\mathrm U}}),&k=0,\\
(L_{\alpha_{\mathrm L}},1],&k=n.
\end{cases}
\]
Thus the endpoint values agree with \citet{ClopperPearson1934}, but the two procedures
need not agree as sets if the classical interval is written with both
endpoints included.
\end{corollary}

The endpoint formulas are known~\citep{phan2021towards} and are included only
for completeness.

\subsection{Comparison with the SymPol interval}

The SymPol test can be inverted explicitly.  Homogeneity gives
$A_k(\mathbf x/\theta)=A_k(\mathbf x)/\theta^k$.  Its lower and upper
endpoints are
\begin{align}
L_{\mathrm{SP}}(\mathbf x)
&=\max_{1\le k\le n}
\left\{\alpha_{\mathrm L}A_k(\mathbf x)\right\}^{1/k},
\label{eq:LSP}\\
U_{\mathrm{SP}}(\mathbf x)
&=1-\max_{1\le k\le n}
\left\{\alpha_{\mathrm U}A_k(\mathbf 1-\mathbf x)\right\}^{1/k}.
\label{eq:USP}
\end{align}
Using the same open/half-open endpoint convention as in \eqref{eq:CIset}, define
\begin{equation*}
\calC^{\mathrm{SP}}_{1-\alpha}(\mathbf x)
=
\begin{cases}
[0,U_{\mathrm{SP}}(\mathbf x)),
    & \mathbf x=\mathbf0,\\
(L_{\mathrm{SP}}(\mathbf x),1],
    & \mathbf x=\mathbf1,\\
(L_{\mathrm{SP}}(\mathbf x),U_{\mathrm{SP}}(\mathbf x)),
    & \text{otherwise}.
\end{cases}
\end{equation*}
This is a finite-sample confidence interval under independence, and more
generally whenever the relevant scaled observations are co-valid e-variables;
this is the tail-split analogue of
\citet[Proposition~4.2]{MingShenWang2026}.  The following corollary is an
immediate consequence of \cref{thm:master}.

\begin{corollary}
\label{cor:CIcontain}
For every $\mathbf x\in[0,1]^n$,
\[
L_{\mathrm{SP}}(\mathbf x)
\le L_{\alpha_{\mathrm L}}(\mathbf x)
< U_{\alpha_{\mathrm U}}(\mathbf x)
\le U_{\mathrm{SP}}(\mathbf x).
\]
Consequently,
 $
\calC_{1-\alpha}(\mathbf x)
\subseteq
\calC^{\mathrm{SP}}_{1-\alpha}(\mathbf x).
$ 
Thus the Gaffke interval is pointwise contained in the SymPol interval under
the same endpoint convention.
\end{corollary}

\subsection{Inadmissibility of the Gaffke confidence interval}
\label{sec:CI-inadmissibility}

We now compare confidence intervals by pointwise set inclusion,
using the same tail levels and the endpoint convention.
An interval procedure is inadmissible if another
valid interval procedure is contained in it for every sample and
strictly contained for at least one sample at some tail level.

For $n=2$, \cref{thm:Kad-optimality} shows that $\Kad$ is a
coordinatewise nonincreasing merger satisfying $\Kad\le K_2$.
Together with its continuity, which follows from the formula in
\cref{sec:admissibility}, these properties imply that its inversion is a
valid confidence interval contained in the Gaffke interval for every sample.
The containment is strict for some samples.  For example, choose
$0<a<1<b$, $0<\theta\le1/b$, and
$\Kad(a,b)<\alpha_{\mathrm L}<K_2(a,b)$, and set
$\mathbf x=(a\theta,b\theta)$ and $\alpha_{\mathrm U}=0$.  Then $\theta$ is
retained by the Gaffke inversion but rejected by the $\Kad$ inversion.
Thus the two-sample Gaffke interval is inadmissible even within the class of
interval-valued procedures.

For $n\ge3$, the merger $M_n$ gives an explicit refinement at the
$95\%$ equal-tailed level, with a strict reduction in length on an open
set of samples.
Fix $n\ge3$ and $\alpha_{\mathrm L}=\alpha_{\mathrm U}=q=1/40$.
Use the endpoints $L_q,U_q$ from \cref{thm:CI}, and write
$I_G(\mathbf{x})=\calC_{0.95}(\mathbf{x})$ for the corresponding Gaffke interval.
Define
\begin{equation*}
 d_n(\mathbf{x})=\min\left\{\frac{x_{(1)}}{a_n},\ x_{(2)},\
                                  \frac{x_{(n)}}{56}\right\}.
\end{equation*}

\begin{proposition}
\label{prop:sm025-interval}
The equal-tailed inversion of $M_n$ has endpoints
\begin{equation*}
 L_M(\mathbf{x})=\max\{L_q(\mathbf{x}),d_n(\mathbf{x})\},\qquad
 U_M(\mathbf{x})=\min\{U_q(\mathbf{x}),1-d_n(\one-\mathbf{x})\}.
\end{equation*}
With \cref{eq:CIset}, these endpoints define an everywhere nonempty
$95\%$ confidence interval $I_M(\mathbf{x})\subseteq I_G(\mathbf{x})$.
In particular, $L_M(\mathbf{x})<U_M(\mathbf{x})$ for every $\mathbf{x}\in[0,1]^n$.
\end{proposition}

\noindent The proof is given in Appendix~\ref{app:proof-sm025-interval}.

\begin{theorem}
\label{thm:sm025-length}
For every $n\ge3$, there are a nonempty open set
$V_n\subset(0,1)^n$ and $\varepsilon_n>0$ such that
\[
 \operatorname{Leb}_1(I_G(\mathbf{x}))-
 \operatorname{Leb}_1(I_M(\mathbf{x}))\ge\varepsilon_n,
 \qquad \mathbf{x}\in V_n.
\]
Consequently, every sampling law assigning positive probability to $V_n$
has strictly positive expected length reduction. This applies, in
particular, to iid uniform observations on $[0,1]$.
\end{theorem}

\noindent The proof is given in Appendix~\ref{app:proof-sm025-length}.

Since $I_M$ is an everywhere nonempty valid interval with a strictly
smaller length on $V_n$, the Gaffke interval is inadmissible under
pointwise set inclusion even within
the class of interval-valued procedures, for every $n\ge3$. By \cref{cor:CIcontain}, the same conclusion
holds for the corresponding SymPol interval.

\section{Asymptotic efficiency of the Gaffke interval}\label{sec:asymptotics}

We next study the original Gaffke interval under iid sampling and identify
its asymptotic benchmarks at fixed and vanishing tail levels.
Throughout this section, $X_1,X_2,\ldots$ are iid with common law $\Pp$
on $[0,1]$, mean $\mu$, and variance $\sigma^2$. For each $n$, write
$\mathbf X=(X_1,\ldots,X_n)$ and
\[
\bar X_n=\frac1n\sum_{i=1}^nX_i,
\qquad
\widehat\sigma_n^2=\frac1n\sum_{i=1}^n(X_i-\bar X_n)^2.
\]
Let $z_q=\Phi^{-1}(q)$, where $\Phi$ is the standard normal CDF.
We establish Gaussian efficiency at fixed tail levels and first-order
asymptotic optimality in two regimes of vanishing tail levels:
moderate deviations and large deviations. In both vanishing-level regimes,
the Gaffke interval attains sharp endpoint and width bounds under the
stated coverage guarantees. The differences between the two regimes are: moderate deviations concern
a shrinking neighborhood of the mean, whereas large deviations yield a
limiting interval of positive width.

\subsection{Fixed-level Gaussian efficiency}
\label{sec:fixed-level-efficiency}

At fixed tail levels, the Gaffke endpoints are first-order
equivalent to those of the usual Wald interval. This follows from a
conditional central limit theorem for their Dirichlet averages.

\begin{theorem}\label{thm:width}
Suppose $\sigma^2>0$ and fix $\alpha_{\mathrm L},\alpha_{\mathrm U}>0$
with $\alpha:=\alpha_{\mathrm L}+\alpha_{\mathrm U}<1$.
Then, almost surely,
\begin{align}
L_{\alpha_{\mathrm L}}(\mathbf X)
&=\bar X_n+\frac{\widehat\sigma_n}{\sqrt n}z_{\alpha_{\mathrm L}}
+o(n^{-1/2}),\label{eq:Lasymp}\\
U_{\alpha_{\mathrm U}}(\mathbf X)
&=\bar X_n+\frac{\widehat\sigma_n}{\sqrt n}z_{1-\alpha_{\mathrm U}}
+o(n^{-1/2}).\label{eq:Uasymp}
\end{align}
Consequently,
\begin{equation}\label{eq:widthlimit}
\sqrt n\,\bigl(U_{\alpha_{\mathrm U}}-L_{\alpha_{\mathrm L}}\bigr)
\longrightarrow
\sigma\bigl(z_{1-\alpha_{\mathrm U}}-z_{\alpha_{\mathrm L}}\bigr)
\quad\text{almost surely},
\end{equation}
and
\(
\Pp\!\left\{\mu\in\calC_{1-\alpha}(\mathbf X)\right\}
=\Pp\!\left\{L_{\alpha_{\mathrm L}}<\mu<U_{\alpha_{\mathrm U}}\right\}
\longrightarrow 1-\alpha.
\)
\end{theorem}

\noindent The proof is given in Appendix~\ref{app:proof-width}.
For equal tails, \eqref{eq:widthlimit} becomes
\(
\sqrt n\,(U_{\alpha/2}-L_{\alpha/2})
\longrightarrow 2\sigma z_{1-\alpha/2}
\text{ almost surely}.
\)
Thus the interval attains the oracle Gaussian width while retaining
finite-sample distribution-free coverage.

\begin{corollary}\label{cor:equal-tail}
Suppose $\sigma^2>0$ and fix $\alpha\in(0,1)$.
Among fixed positive allocations
$\alpha_{\mathrm L}+\alpha_{\mathrm U}=\alpha$, the equal-tail allocation
uniquely minimizes the limiting width constant in \eqref{eq:widthlimit}.
\end{corollary}

\begin{proof}
Write $t=\alpha_{\mathrm L}$ and
$g(t)=z_{1-\alpha+t}-z_t$ for $0<t<\alpha$.  The function is symmetric about
$\alpha/2$.  If $t<\alpha/2$, then
$z_{1-\alpha+t}$ lies strictly between $z_t$ and $-z_t$, so
$\varphi(z_{1-\alpha+t})>\varphi(z_t)$ and
\[
g'(t)=\frac{1}{\varphi(z_{1-\alpha+t})}-\frac{1}{\varphi(z_t)}<0,
\]
where $\varphi$ is the standard normal density function. Symmetry gives the result.
\end{proof}

The nondegeneracy assumption is essential for the $n^{-1/2}$ scale.
If $\sigma^2=0$, then $X_i=\mu$ almost surely, and
$D_0\sim\Beta(1,n)$ gives the exact endpoints
\begin{equation}\label{eq:degenerate}
L_{\alpha_{\mathrm L}}=\mu\alpha_{\mathrm L}^{1/n},
\qquad
U_{\alpha_{\mathrm U}}=1-(1-\mu)\alpha_{\mathrm U}^{1/n}.
\end{equation}
For fixed equal tails $q=\alpha/2\in(0,1/2)$, their width is
$1-q^{1/n}=-\frac{\log q}{n}+o(n^{-1})$.

\subsection{Moderate-deviation optimality}
\label{sec:moderate-deviations}

In the moderate-deviation regime, the   levels
$q_n\in(0,1/2)$ vanish subexponentially with $n$:
\[
 h_n=\log(1/q_n)\longrightarrow\infty,
 \qquad \frac{h_n}{n}\longrightarrow0.
\]
The Gaffke interval still contracts to the mean. Moderate-deviation optimality asks
whether another interval with the same coverage guarantee can bring its
endpoints closer to the mean or have a smaller width. We formulate this
comparison under a common total error budget $2q_n$, without prescribing
how it is split between the two tails.

Specifically, consider Borel endpoints
$\widetilde L_n,\widetilde U_n:[0,1]^n\to\R$ with
$\widetilde L_n\le\widetilde U_n$ and
$\widetilde I_n=[\widetilde L_n,\widetilde U_n]$ satisfying, for every
probability law $\Pp'$ on $[0,1]$ and every $n$,
\begin{equation}\label{eq:ext-md-coverage}
 (\Pp')^n\{\E_{\Pp'}X\in\widetilde I_n\}\ge1-2q_n.
\end{equation}
At a fixed law $\Pp$ with $\sigma^2>0$, set
$r_n=\sigma\sqrt{2h_n/n}$.
An interval sequence in this class, with endpoints denoted by $L_n,U_n$,
is \emph{optimal in the moderate-deviation sense}
at $\Pp$ if every competing interval satisfies, for each $\varepsilon\in(0,1)$,
\begin{equation}\label{eq:ext-md-endpoint-bound}
 \begin{split}
 \Pp^n\{\widetilde L_n>\mu-(1-\varepsilon)r_n\}&\longrightarrow0,\\
 \Pp^n\{\widetilde U_n<\mu+(1-\varepsilon)r_n\}&\longrightarrow0,
 \end{split}
\end{equation}
and its own endpoints attain these bounds:
\begin{equation}\label{eq:ext-md-attainment}
 \frac{\mu-L_n}{r_n}\xrightarrow{\Pp}1,
 \qquad
 \frac{U_n-\mu}{r_n}\xrightarrow{\Pp}1.
\end{equation}
For every $\varepsilon\in(0,1)$, these conditions imply the width comparison
\[
 \Pp^n\{\widetilde U_n-\widetilde L_n<2(1-\varepsilon)r_n\}
 \longrightarrow0,
 \qquad
 \frac{U_n-L_n}{2r_n}\xrightarrow{\Pp}1.
\]
Moderate-deviation optimality concerns the leading constants on the
shrinking scale $r_n$: it rules out an inward shift of either endpoint
by a fixed fraction of $r_n$ with nonvanishing probability. The theorem below shows
that the equal-tailed Gaffke interval attains these bounds among all
intervals satisfying the joint coverage requirement
\eqref{eq:ext-md-coverage}.

\begin{theorem}\label{thm:ext-md}
Suppose $\sigma^2>0$. Let $(q_n)$ be a deterministic sequence in $(0,1/2)$, set
$h_n=\log(1/q_n)$, and assume $h_n\to\infty$ and $h_n/n\to0$.
The equal-tailed Gaffke interval is optimal in the moderate-deviation sense
at $\Pp$ among intervals satisfying \eqref{eq:ext-md-coverage}.
Moreover, almost surely,
\begin{align*}
 L_{q_n}(\mathbf X)
 &=\bar X_n-\sqrt{\frac{2\widehat\sigma_n^2h_n}{n}}\{1+o(1)\},\\
 U_{q_n}(\mathbf X)
 &=\bar X_n+\sqrt{\frac{2\widehat\sigma_n^2h_n}{n}}\{1+o(1)\}.
\end{align*}
In particular, for $r_n=\sigma\sqrt{2h_n/n}$,
\[
 \frac{U_{q_n}(\mathbf X)-L_{q_n}(\mathbf X)}{2r_n}
 \longrightarrow1\qquad\text{almost surely}.
\]
\end{theorem}

\noindent The proof is given in Appendix~\ref{app:proof-ext-md}.

The leading constant depends on the sampling law only through its variance.
When the tail levels instead vanish exponentially, the benchmark is a
limiting interval of positive width.

\subsection{Large-deviation optimality}
\label{sec:large-deviations}

In the large-deviation regime, the  tail levels
$\alpha_{\mathrm L,n},\alpha_{\mathrm U,n}\in(0,1)$ satisfy
$\alpha_{\mathrm L,n}+\alpha_{\mathrm U,n}<1$ for every $n$ and
\[
 -\frac1n\log\alpha_{j,n}\longrightarrow c_j\in(0,\infty),
 \qquad j\in\{\mathrm L,\mathrm U\}.
\]
Thus both levels decrease exponentially, with possibly different rates.

Let $\mathcal P([0,1])$ denote the probability laws on $[0,1]$.
For $\Pp\in\mathcal P([0,1])$ with mean $\mu\in(0,1)$, define
\[
 d(\Pp,m)=\inf_{\substack{\Pp'\in\mathcal P([0,1])\\ \E_{\Pp'}X=m}}
          \operatorname{KL}(\Pp\|\Pp'),\qquad 0<m<1.
\]
This is the smallest relative entropy from $\Pp$ to a law with mean $m$.
For $c>0$, its sublevel set is an interval
\[
 J_c(\Pp)=\{m\in(0,1):d(\Pp,m)\le c\}
          =[l_c(\Pp),u_c(\Pp)],
\]
where $l_c(\Pp)<\mu<u_c(\Pp)$ are the unique solutions of
$d(\Pp,m)=c$ on the two sides of $\mu$.
The information bounds of \citet{DeepBassambooJuneja2026} provide the
benchmark in this regime. We give a self-contained proof of the
 bounds in our setting and show that the Gaffke interval
attains them.

The rates $c_{\mathrm L}$ and $c_{\mathrm U}$ determine the respective
benchmarks $l_{c_{\mathrm L}}(\Pp)$ and $u_{c_{\mathrm U}}(\Pp)$.
We assess optimality among Borel one-sided bounds
$\widetilde L_n,\widetilde U_n:[0,1]^n\to\R$ satisfying the prescribed
error guarantees: for every $\Pp'\in\mathcal P([0,1])$ and every $n$,
\begin{equation}\label{eq:ext-ld-coverage}
 \begin{split}
 (\Pp')^n\{\widetilde L_n>\E_{\Pp'}X\}&\le\alpha_{\mathrm L,n},\\
 (\Pp')^n\{\widetilde U_n<\E_{\Pp'}X\}&\le\alpha_{\mathrm U,n}.
 \end{split}
\end{equation}
These constraints fix each one-sided error budget across procedures.
A pair $L_n,U_n$ satisfying
\eqref{eq:ext-ld-coverage} is \emph{optimal in the large-deviation sense}
at $\Pp$, relative to the benchmarks $l_{c_{\mathrm L}}(\Pp)$ and
$u_{c_{\mathrm U}}(\Pp)$, if every competing pair satisfies, for each $\varepsilon>0$,
\begin{equation}\label{eq:ext-ld-endpoint-bound}
 \begin{split}
 \Pp^n\{\widetilde L_n>l_{c_{\mathrm L}}(\Pp)+\varepsilon\}&\longrightarrow0,\\
 \Pp^n\{\widetilde U_n<u_{c_{\mathrm U}}(\Pp)-\varepsilon\}&\longrightarrow0,
 \end{split}
\end{equation}
and $L_n\xrightarrow{\Pp}l_{c_{\mathrm L}}(\Pp)$,
$U_n\xrightarrow{\Pp}u_{c_{\mathrm U}}(\Pp)$.
Writing $w_\star=u_{c_{\mathrm U}}(\Pp)-l_{c_{\mathrm L}}(\Pp)>0$,
these conditions imply, for every $\varepsilon>0$,
\[
 \Pp^n\{\widetilde U_n-\widetilde L_n<w_\star-\varepsilon\}
 \longrightarrow0,
 \qquad U_n-L_n\xrightarrow{\Pp}w_\star.
\]
Large-deviation optimality thus gives sharp limiting endpoints and width, and rules out an inward shift of either endpoint by a
fixed positive amount with nonvanishing probability.

\begin{theorem}\label{thm:ext-ld}
Suppose $0<\mu<1$. Let $(\alpha_{\mathrm L,n})$ and
$(\alpha_{\mathrm U,n})$ be deterministic sequences in $(0,1)$ with
$\alpha_{\mathrm L,n}+\alpha_{\mathrm U,n}<1$ for every $n$, and suppose
\[
 -\frac1n\log\alpha_{\mathrm L,n}\longrightarrow c_{\mathrm L}\in(0,\infty),
 \qquad
 -\frac1n\log\alpha_{\mathrm U,n}\longrightarrow c_{\mathrm U}\in(0,\infty).
\]
Then
\[
 L_{\alpha_{\mathrm L,n}}(\mathbf X)\xrightarrow{\Pp}l_{c_{\mathrm L}}(\Pp),
 \qquad
 U_{\alpha_{\mathrm U,n}}(\mathbf X)\xrightarrow{\Pp}u_{c_{\mathrm U}}(\Pp).
\]
The Gaffke interval is optimal in the large-deviation sense at $\Pp$
among bounds satisfying \eqref{eq:ext-ld-coverage}.
At the same tail levels, the SymPol endpoints converge in probability to
the same limits, and the optimized-betting endpoints converge almost surely.
\end{theorem}

\noindent The proof is given in Appendix~\ref{app:proof-ext-ld}.
It combines the optimized-betting endpoint limits, the pointwise ordering
\eqref{eq:chain-compare}, and a change-of-measure bound.

For deterministic equal tail levels $q_n\in(0,1/2)$ with
$-\log q_n/n\to c\in(0,\infty)$,
the Gaffke endpoints converge in probability to $l_c(\Pp)$ and
$u_c(\Pp)$. Consider any competing interval with coverage at least
$1-2q_n$ under every iid law on $[0,1]$, without a prescribed tail
allocation. Each of its one-sided error probabilities is at most
$2q_n$. The change-of-measure argument proving
\eqref{eq:ext-ld-endpoint-bound} applies with $2q_n$ as the one-sided
error bound, since $-\log(2q_n)/n\to c$.
Thus, as in the moderate-deviation regime, the equal-tailed Gaffke interval
attains the optimal bounds under the joint coverage requirement alone;
competing intervals need not use equal tails.

The Gaffke, SymPol and optimized-betting statistics share these large-deviation limits, despite
their finite-sample ordering. Unlike the preceding two theorems,
\cref{thm:ext-ld} also allows zero variance: for $\Pp=\mathbb{D}_\mu$,
 $
 J_c(\Pp)=[\mu e^{-c},1-(1-\mu)e^{-c}],
 $
in agreement with \eqref{eq:degenerate}. For Bernoulli $\Pp$,
$d(\Pp,m)$ is the binary relative entropy, consistently with the
Clopper--Pearson endpoints in \cref{cor:CP}.

These three regimes let the sample size grow. Complementary results
for fixed sample size and vanishing tail levels are given in
Appendix~\ref{app:proof-ext-fixed}.

\section{Simulations}\label{sec:simulations}

The simulations address practical performance rather than admissibility of the
underlying e-to-p merger.  We first quantify the pointwise advantage over
the SymPol test, then compare the inverted Gaffke interval with modern bounded-mean
procedures.  The admissible rule $\Kad$ is a two-input construction, and the
higher-dimensional refinement $M_n$ is examined separately in
Appendix~\ref{app:sm025-simulation}; neither is included as a general
confidence-interval comparator here.

\subsection{Computation of the Gaffke interval}\label{sec:computation}

No grid search over candidate means is needed.  Both confidence endpoints are
ordinary one-dimensional quantiles of a uniform Dirichlet average.

Let $\mathbf a=(a_0,\ldots,a_n)$ and
\(
W_{\mathbf a}=\sum_{i=0}^n a_iD_i,
\quad \mathbf D\sim\Dir(1,\ldots,1).
\)
When the knots $a_i$ are distinct, the CDF has the truncated-power form
\begin{equation}\label{eq:dirichletCDF}
F_{\mathbf a}(t)
=
\sum_{i=0}^n
\frac{(t-a_i)_+^n}{\prod_{j\ne i}(a_j-a_i)}.
\end{equation}
The density is the derivative of \eqref{eq:dirichletCDF},
\begin{equation*}
f_{\mathbf a}(t)
=
n\sum_{i=0}^n
\frac{(t-a_i)_+^{n-1}}{\prod_{j\ne i}(a_j-a_i)},
\end{equation*}
and is a normalized B-spline with knots $a_0,\ldots,a_n$; see
\citet{Carlson1991} and \citet{ZuCastell2002}.  Repeated knots are handled by
continuous extension, equivalently by confluent divided differences.

The direct formula \eqref{eq:dirichletCDF} can suffer severe cancellation when
knots are equal or nearly equal.  A stable implementation proceeds as follows:

\begin{enumerate}[leftmargin=2em]
\item sort the knots;
\item construct the normalized B-spline density associated with them;
\item integrate the spline to obtain $F_{\mathbf a}$;
\item solve $F_{\mathbf a}(t)=q$ on $[\min_i a_i,\max_i a_i]$ by bisection.
\end{enumerate}

For the Gaffke interval, use the knot vectors
 $
(0,x_1,\ldots,x_n)
 $ {and} $ 
(x_1,\ldots,x_n,1)
 $
and quantile levels $\alpha_{\mathrm L}$ and $1-\alpha_{\mathrm U}$,
respectively.  The Gaffke test itself is simply
\[
K_n(\mathbf x)=F_{(0,x_1,\ldots,x_n)}(1).
\]
For Bernoulli data, beta quantiles from \cref{cor:CP} are simpler.
The numerical routine returns the two endpoint values; confidence-set
membership is then evaluated according to \eqref{eq:CIset}. Endpoint
inclusion does not affect the width $U-L$.

We implemented the B-spline calculation, the normalized recursion
\eqref{eq:Arecursion}, and the confidence interval construction, and conducted
all numerical experiments in \cref{sec:simulations}.  Monte Carlo Dirichlet weights
provide an alternative for very large $n$, but deterministic spline inversion
avoids simulation error in the reported confidence endpoints.

We used Python for the B-spline evaluation of
$K_n$, the confidence interval, the SymPol recursion, all
simulations, exact two-point calculations, and the figures. The continuous
experiments use seeds $260708415$ and $260310329$; the sample-size experiment
uses seed $881+n$ for sample size $n$.

\subsection{Test power: Gaffke versus SymPol}\label{sec:empirical}

Pointwise domination settles the direction of the power comparison.  Numerical
experiments remain useful for quantifying its magnitude and for distinguishing
the effect of pointwise ordering from the effect of finite-sample
conservativeness.

We used $n=30$,  $\alpha=0.05$, and $50{,}000$ Monte Carlo
replications for each distribution.  Under the boundary null, each distribution
was scaled to have mean one.  Under the alternative it was scaled to have mean
$1.2$.  Common random numbers were used by multiplying each mean-one sample by
$1.2$.  The maximum Monte Carlo standard error of a reported rejection
probability is $\sqrt{0.25/50000}=0.00224$.

The Gaffke p-value was evaluated deterministically by the B-spline method in
\cref{sec:computation}.  The symmetric means were computed through the normalized
recursion
\begin{equation}\label{eq:Arecursion}
A_j^{(m)}
=\frac{m-j}{m}A_j^{(m-1)}
+\frac{j}{m}x_mA_{j-1}^{(m-1)},
\quad j=1,\ldots,m,
\end{equation}
where $A_0^{(m)}=1$.
For comparison, we also report a diagnostic ``size-matched'' power.  For each
parametric null family and each method, the critical value was set to its
empirical lower $5\%$ null quantile.  These calibrated critical values are not
distribution-free and are used only to separate ordering from calibration.

\begin{table}[!htbp]
\centering
\caption{Null rejection and power in continuous models.  Each entry gives
Gaffke / SymPol, in percent.  Size-matched power uses a
family-specific empirical $5\%$ critical value.}
\label{tab:continuous}
\small
\begin{tabular}{lccc}
\toprule
Distribution
& Null rejection
& Nominal power
& Size-matched power\\
\midrule
Gamma, shape $0.25$     & $2.84/0.85$ & $8.91/3.46$  & $13.62/13.79$\\
Exponential             & $2.69/0.79$ & $19.17/8.98$ & $27.47/27.38$\\
Gamma, shape $4$        & $1.55/0.80$ & $44.74/33.68$& $64.39/63.59$\\
Lognormal, $\sigma=0.5$ & $1.60/0.79$ & $42.97/33.10$& $62.50/63.28$\\
Lognormal, $\sigma=1$   & $2.51/0.72$ & $14.45/6.63$ & $22.48/22.58$\\
Lognormal, $\sigma=1.5$ & $2.36/0.73$ & $7.44/2.85$  & $12.63/12.36$\\
Pareto, shape $1.5$     & $1.64/0.63$ & $6.31/2.87$  & $14.95/14.65$\\
\bottomrule
\end{tabular}
\end{table}

At the advertised $5\%$ thresholds, the Gaffke power is between $1.30$ and
$2.61$ times the SymPol power in these examples.  The
SymPol test rejects only about $0.6\%$--$0.9\%$ of boundary-null
samples, whereas the Gaffke test rejects about $1.6\%$--$2.8\%$.  After
family-specific size matching, the maximum absolute power difference in
\cref{tab:continuous} is $0.80$ percentage points.  Thus most of the nominal
power advantage in these models is attributable to tighter finite-sample
calibration.

\begin{table}[!htbp]
\centering
\caption{Power against an exponential alternative with mean $1.2$ as the
$n$ varies.  Each entry gives Gaffke / SymPol.}
\label{tab:sample-size}
\small
\begin{tabular}{rcc}
\toprule
$n$ & Nominal power & Size-matched power\\
\midrule
5   & $0.0279/0.0245$ & $0.1181/0.1152$\\
10  & $0.0675/0.0414$ & $0.1511/0.1495$\\
20  & $0.1312/0.0631$ & $0.2184/0.2164$\\
30  & $0.1889/0.0883$ & $0.2730/0.2728$\\
50  & $0.2914/0.1402$ & $0.3782/0.3789$\\
100 & $0.5135/0.2918$ & $0.5760/0.5760$\\
\bottomrule
\end{tabular}
\end{table}

\FloatBarrier
\subsection{Confidence interval widths}

The two tests compared in the previous subsection were one-sided tests for unbounded data. For the problem of estimating means of bounded data, more methods exist. 
Here, we compare several leading confidence intervals: W-S\&R24 from
\citet[Remark 3]{WaudbySmithRamdas2024}, STAR from
\citet[Algorithm 4]{VO25}, both SymPol and KL-inf
from~\cite{MingShenWang2026}, and Gaffke's.  \hyperlink{ci-widths-start}{\Cref*{fig:ci-widths}} shows the
numerical interval endpoints (left) and their widths (right); endpoint
inclusion is understood according to each procedure's convention and does
not affect the displayed widths.  The following five distributions are shown
in order:
\[
\mathrm{Bernoulli}(1/10),\quad \mathrm{Bernoulli}(0.5),\quad
\mathrm{Unif}(0,1),\quad \mathrm{Beta}(10,30),\quad
\mathrm{Beta}(100,100).
\]
All confidence intervals are constructed at the $95\%$ level, and are averaged over 5 replications.
Among the methods, distributions, and sample sizes examined here, the Gaffke interval is the shortest throughout, while the identity of the best non-Gaffke comparator changes across settings.  This finite-sample pattern is consistent with the asymptotic analysis: Gaffke's adapts to the population variance at first order, whereas the other procedures in this particular comparison retain larger limiting constants.  The statement is empirical and restricted to the designs considered; it is not an admissibility claim.

\begin{figure}[!htp]
\centering
\setkeys{Gin}{width=\linewidth,height=97.58pt,keepaspectratio}
\noindent\vadjust pre{\hypertarget{ci-widths-start}{}}%
\includegraphics{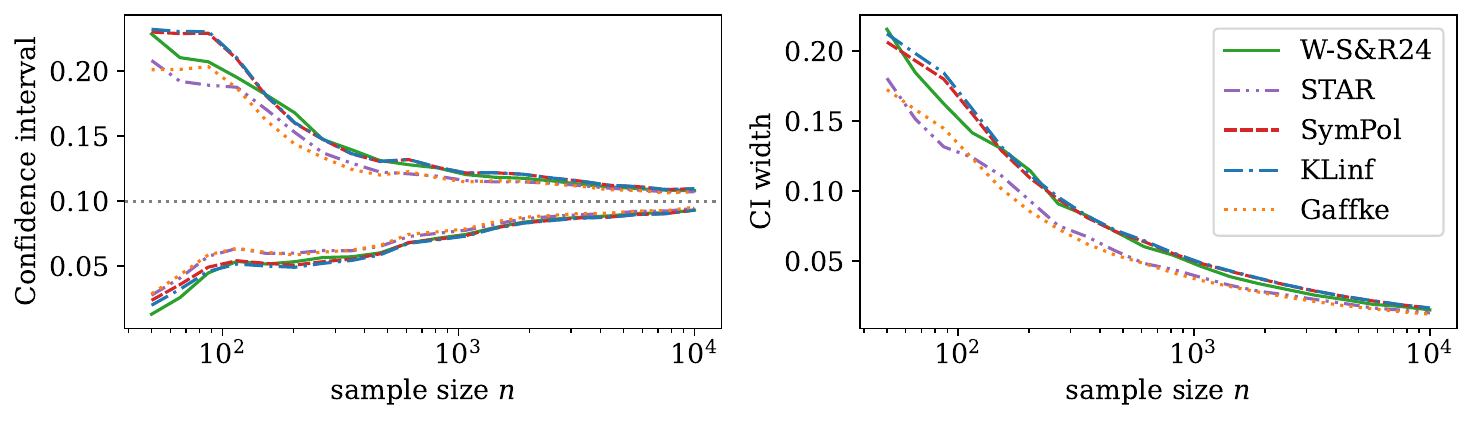}\par
\includegraphics{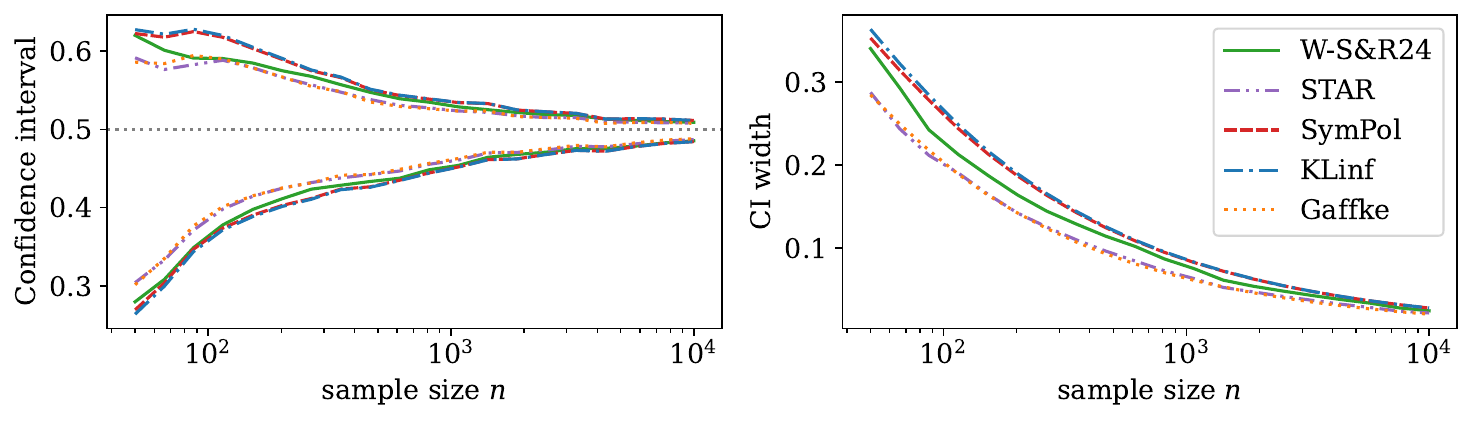}\par
\includegraphics{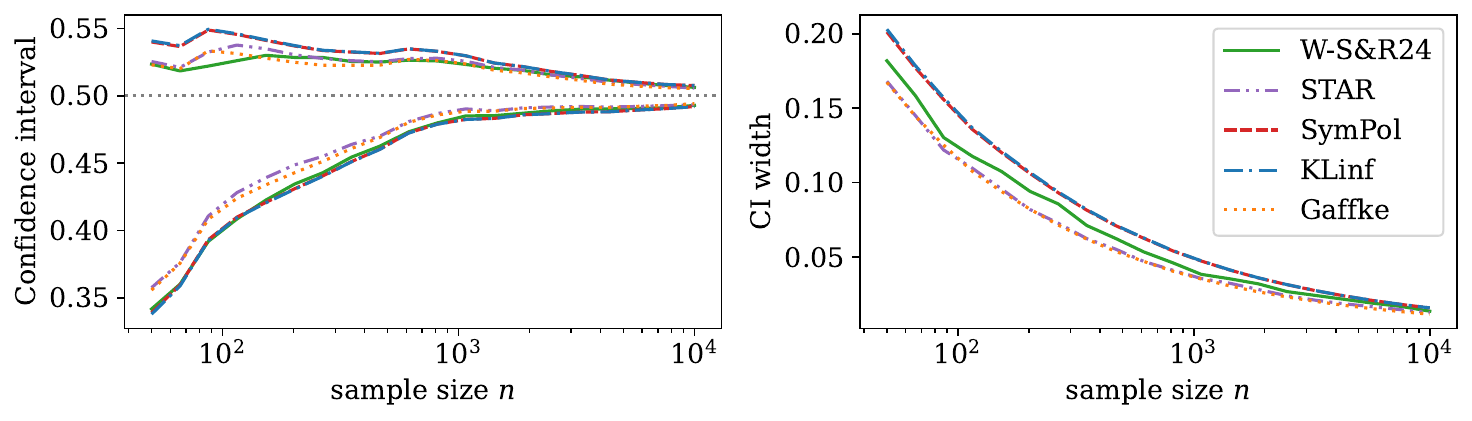}\par
\includegraphics{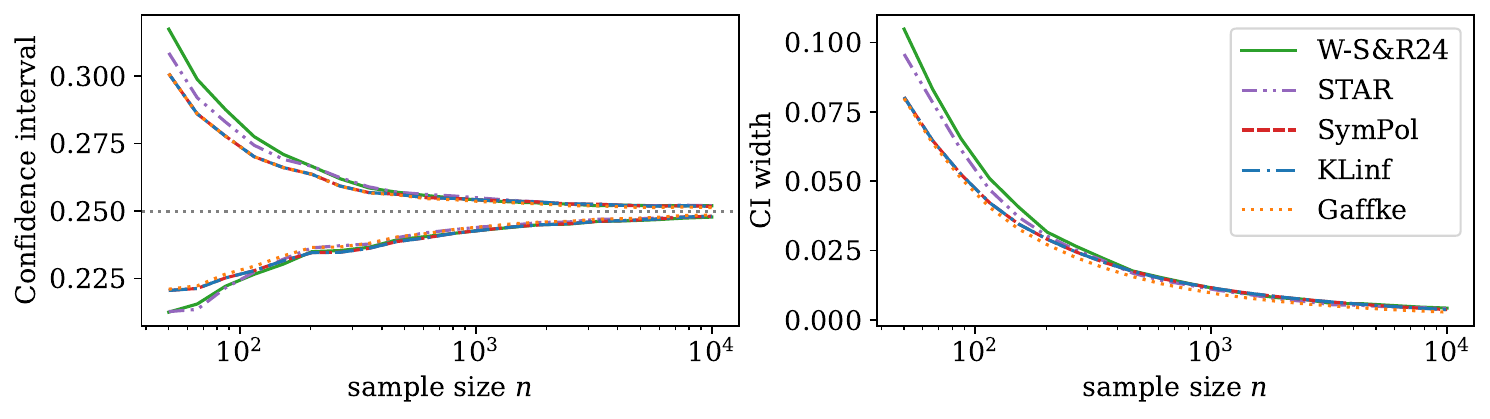}\par
\includegraphics{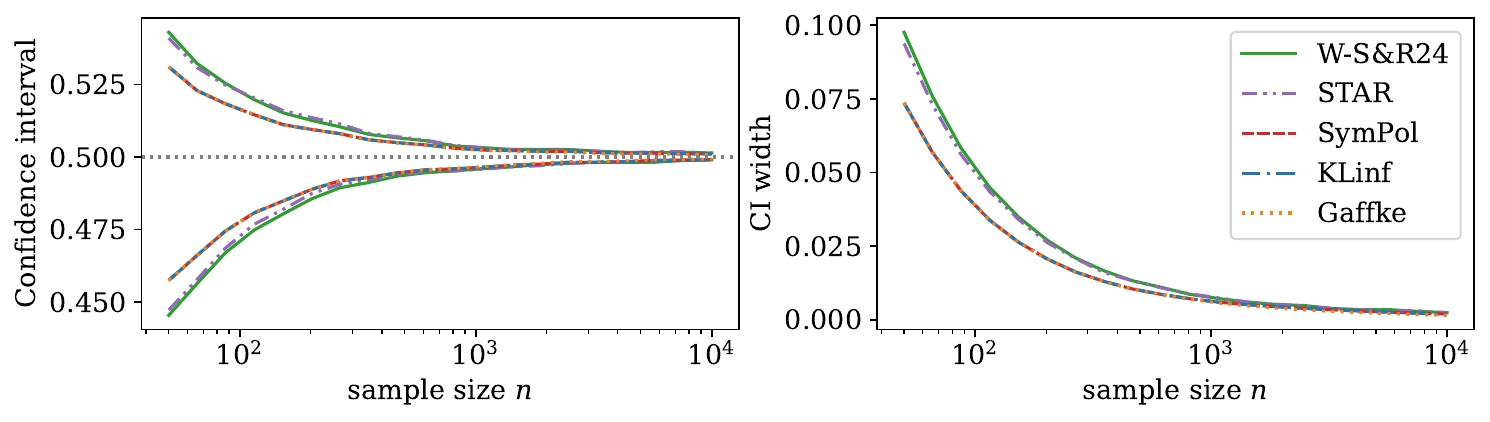}\par
\caption{Confidence interval endpoints (left) and widths (right) at the $95\%$ level.
From top to bottom, the distributions are $\mathrm{Bernoulli}(1/10)$,
$\mathrm{Bernoulli}(0.5)$, $\mathrm{Unif}(0,1)$, $\mathrm{Beta}(10,30)$,
and $\mathrm{Beta}(100,100)$.}
\label{fig:ci-widths}
\end{figure}

\FloatBarrier
\subsection{Gaffke versus Austern--Mackey}
\label{sec:gaffke-vs-ebe}

We compare the equal-tail Gaffke interval with the finite-sample empirical Berry--Esseen (EBE) interval of \citet{austern2022efficient} for estimating the mean of an iid random variable supported on $[0,1]$, since this method also has the same asymptotic CLT efficiency property for its limiting width, unlike all the comparators in the previous subsection.    

\paragraph*{Comparator}
The EBE construction first forms a high-probability interval for the unknown population standard deviation and then maximizes a known-variance quantile bound over that interval.  We used the authors' current implementations of the efficient quantile bound and empirical variance bounds.  To make the Monte Carlo experiment computationally feasible, the maximization over $\sigma\in[0,1/2]$ was evaluated on an 11-point grid, after expanding each sample-dependent standard-deviation interval by one grid point on both sides.  This is a numerical approximation to the repository's continuous optimizer, not a bit-for-bit invocation of \texttt{ebe\_ci} for every replication.  Three direct checks found grid-minus-optimizer half-width differences of $-0.0013$, $0.0105$, and $0.0091$ at $n=25,100,$ and $500$, respectively.  The grid therefore tends to be conservative in the moderate- and large-$n$ checks, and the reported EBE widths should be interpreted as approximate upper estimates rather than exact repository outputs.

\paragraph*{Simulation design}
We used $\alpha=0.05$,  $n\in\{25,50,100,250,500,1000\}$, and six distributions:
\[
\mathrm{Bernoulli}(0.5),\quad \mathrm{Bernoulli}(0.1),\quad
\mathrm{Unif}(0,1),\quad \mathrm{Beta}(2,5),\quad
\mathrm{Beta}(1/2,1/2),\quad \mathrm{Beta}(20,20).
\]
We used 1,000 replications in the first 23 cells and 250 replications in the remaining clustered-beta cells, for 26,250 intervals in total.  We recorded the replication counts and binomial
Monte Carlo standard errors.  The distributions span the maximal-variance Bernoulli law, a sparse Bernoulli law, a continuous symmetric law, a skewed law, a U-shaped law, and a highly concentrated low-variance law.

\begin{table}[t]
\centering
\caption{Empirical coverage and mean width of equal-tail $95\%$ Gaffke intervals and grid-approximated empirical Berry--Esseen (EBE) intervals. Values in the $N$ column are Monte Carlo replications.}
\label{tab:gaffke-ebe-main}
\small
\begin{tabular}{lrrrrrr}
\toprule
Distribution & $n$ & $N$ & Cov. G & Cov. EBE & Width G & Width EBE \\
\midrule
$\mathrm{Bernoulli}(0.5)$ & 100 & 1000 & 0.967 & 0.987 & 0.2024 & 0.2505 \\
$\mathrm{Bernoulli}(0.5)$ & 1000 & 1000 & 0.975 & 0.989 & 0.0629 & 0.0741 \\
$\mathrm{Bernoulli}(0.1)$ & 100 & 1000 & 0.955 & 1.000 & 0.1265 & 0.2230 \\
$\mathrm{Bernoulli}(0.1)$ & 1000 & 1000 & 0.957 & 0.999 & 0.0382 & 0.0704 \\
$\mathrm{Unif}(0,1)$ & 100 & 1000 & 0.968 & 1.000 & 0.1224 & 0.2505 \\
$\mathrm{Unif}(0,1)$ & 1000 & 1000 & 0.955 & 1.000 & 0.0367 & 0.0664 \\
$\mathrm{Beta}(2,5)$ & 100 & 1000 & 0.986 & 1.000 & 0.0755 & 0.2334 \\
$\mathrm{Beta}(2,5)$ & 1000 & 250 & 0.960 & 1.000 & 0.0209 & 0.0477 \\
$\mathrm{Beta}(1/2,1/2)$ & 100 & 250 & 0.940 & 0.996 & 0.1465 & 0.2505 \\
$\mathrm{Beta}(1/2,1/2)$ & 1000 & 250 & 0.964 & 1.000 & 0.0447 & 0.0727 \\
$\mathrm{Beta}(20,20)$ & 100 & 250 & 1.000 & 1.000 & 0.0480 & 0.2112 \\
$\mathrm{Beta}(20,20)$ & 1000 & 250 & 0.972 & 1.000 & 0.0109 & 0.0371 \\
\bottomrule
\end{tabular}
\end{table}

\paragraph*{Coverage}
All simulated means lie in $(0,1)$, so Gaffke's interval coverage is evaluated using the
strict event
$L_{\alpha/2}(\mathbf X)<\mu<U_{\alpha/2}(\mathbf X)$, in accordance with
\eqref{eq:CIset}.  The empirical coverage of the Gaffke interval ranged from
0.940 to 1.000.  The lowest value, 0.940, occurred in a cell with only 250 replications and is within roughly one to two Monte Carlo standard errors of $0.95$.  The grid-approximated EBE interval covered between 0.979 and 1.000, and was almost always quite conservative. \cref{fig:coverage-gaffke} displays the Gaffke and EBE coverages.

\begin{figure}[t]
\centering
\includegraphics[width=.49\linewidth]{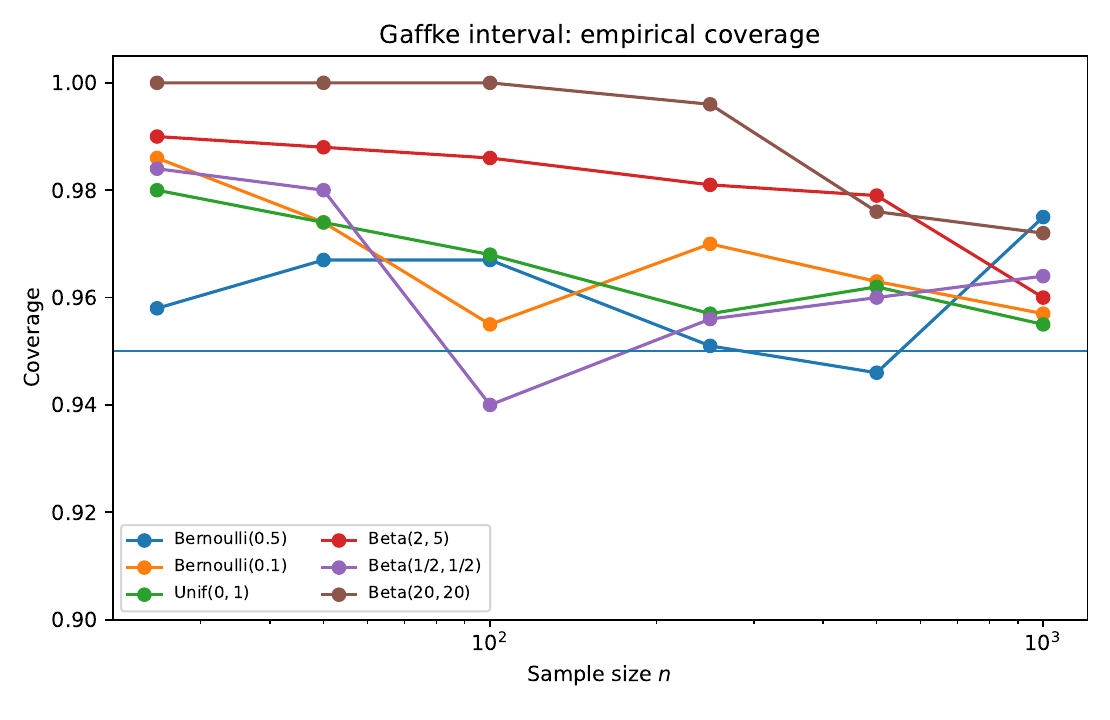}
\includegraphics[width=.49\linewidth]{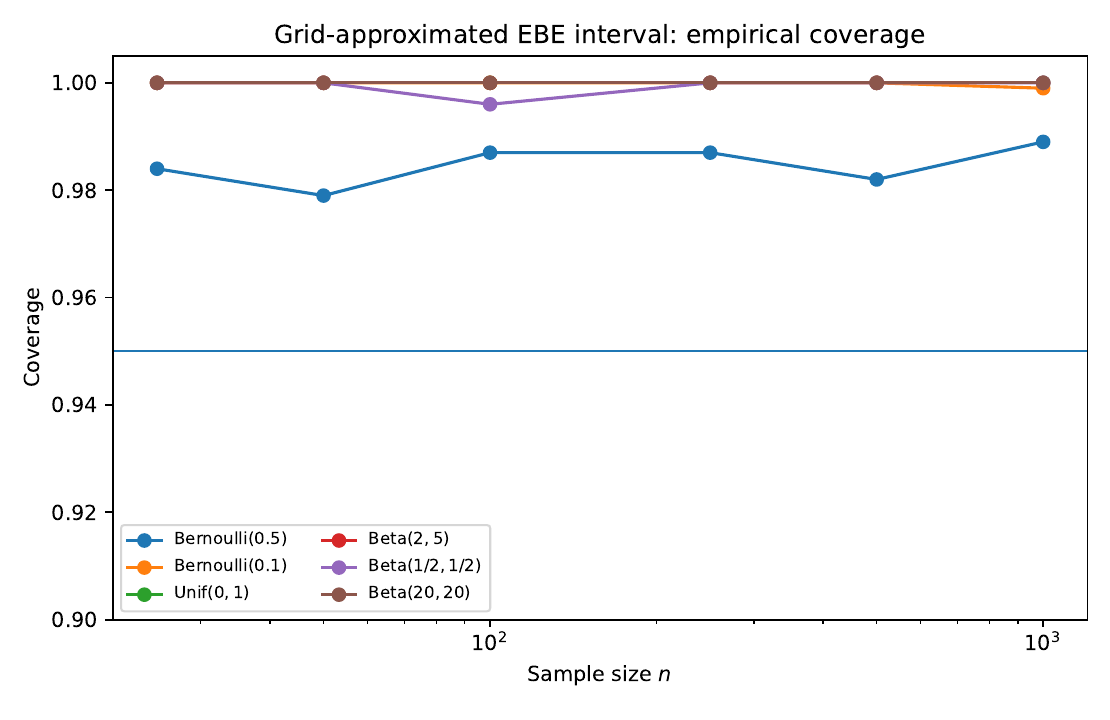}
\caption{Empirical coverage of the equal-tail Gaffke interval (left) and the grid-approximated EBE interval (right). The horizontal line is the nominal level $0.95$.}
\label{fig:coverage-gaffke}
\end{figure}

\paragraph*{Width}
The EBE-to-Gaffke mean-width ratio ranged from 1.18 (Bernoulli(0.5), $n=1000$) to 4.40 (Beta(20,20), $n=100$).  The gap was smallest for $\mathrm{Bernoulli}(0.5)$, whose variance is the largest possible on $[0,1]$, and largest for the concentrated $\mathrm{Beta}(20,20)$ distribution.  This pattern is consistent with a finite-sample price for constructing and protecting uniformly over a confidence set for the unknown variance.  In contrast, the Gaffke interval adapts directly to the empirical distribution through the conditional Dirichlet quantiles.  \cref{fig:width-ratio} reports the full width ratios.

\begin{figure}[t]
\centering
\includegraphics[width=.6\linewidth]{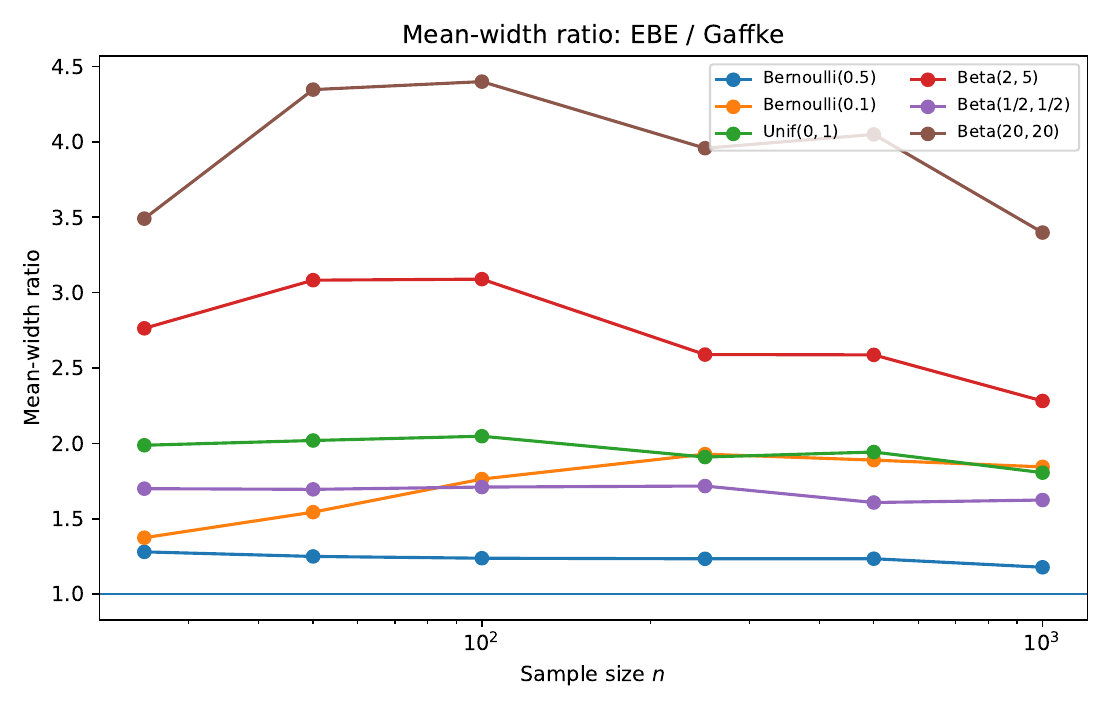}
\caption{Ratio of mean EBE width to mean Gaffke width. Values above one favor Gaffke.}
\label{fig:width-ratio}
\end{figure}

Both procedures are finite-sample valid under their assumptions and both have the oracle Gaussian width $2\sigma z_{1-\alpha/2}/\sqrt n$ as their first-order target.  The experiment indicates that the Gaffke interval can have a substantial finite-sample advantage, especially when the population variance is small.  The magnitude of the EBE disadvantage in the present tables is somewhat overstated by the coarse variance grid; the direct optimizer checks suggest that this numerical effect is material for low-variance distributions.  A definitive high-precision comparison should use a finer grid or perform the continuous optimization separately for every replication.


\FloatBarrier
\section{Discussion}
\label{sec:discussion}

Gaffke's statistic is a finite-sample valid, distribution-free method
for combining independent e-values and pointwise dominates the SymPol statistic. 
Its bounded-mean inversion is always nonempty,
is determined by two Dirichlet quantiles, and recovers the
Clopper--Pearson endpoints for Bernoulli data.
Under iid sampling, this interval has three distinct asymptotic properties established in Section~\ref{sec:asymptotics}.  
The fixed-level simulations complement these asymptotic conclusions.
In \cref{sec:simulations}, Gaffke's has the smallest reported mean widths among the implemented
procedures in the designs considered.

Admissibility is a more delicate issue, studied  in Section~\ref{sec:admissibility}.
We can find deterministic improvements of the Gaffke statistic for every dimension $n\geq 2$ and randomized improvements for every dimension $n\geq 1$, although the deterministic improvement only leads to small improvements in practice.
When comparing with the SymPol statistics, 
we note that the Gaffke p-value dominates the SymPol p-value, but the latter is applicable to a broader class of dependence.  

In his original work~\citep{Gaffke2005} as well as the follow-up papers by \cite{LearnedMillerThomas2020}  and \cite{VlassisThomas2026}, Gaffke's method was cast as a test for the mean of a nonnegative random variable. We viewed the test from a slightly different lens:
Gaffke's method combines $n$ independent e-values $X_1,\dots,X_n$ into a p-value. 
The most naive way of performing this combination is to return $p_\mathrm{prod}(\mathbf x) = \min\{1/\prod_{i=1}^n x_i,1\}$, which is indeed always a valid p-value by Markov's inequality.
If all the e-values are larger than one, then in fact $K_n(\mathbf x) = p_\mathrm{prod}(\mathbf x)$. But for mixed configurations containing coordinates on both sides of one, it can happen that $\prod_{i=1}^n x_i \leq 1$, so $p_\mathrm{prod}=1$, while $K_n$ evaluates to a different expression less than 1. For instance, $K_n(M,0,0,\dots,0)=1-(1-1/M)^n$ if $M\ge1$, which resembles a \v{S}id\'ak correction for the maximum e-value. In this sense, $K_n$ interpolates between the product and the maximum of the e-values (appropriately converted to p-values), and thus adapts to both sparse and dense alternatives, providing evidence against the null both under dense weak evidence (many e-values just larger than 1) and under sparse strong evidence (a few very large e-values, the rest smaller than one). 

This also means that Gaffke's test has implications well beyond the bounded mean problem for which it has been extensively studied, a fact that has not been previously appreciated. It is a tool for hypothesis testing in any problem setting for which one can convert the raw data into e-values. In fact, it is known that \emph{every} null hypothesis can be equally specified by instead describing the set of e-values for it, and that one can always construct e-values for any testing problem~\citep{ramdas2025hypothesis}. Thus, in some sense, Gaffke's p-value is a very broadly applicable tool in hypothesis testing, especially in nonparametric composite null hypotheses, where nonasymptotically valid tests may be hard to construct.

\section*{Acknowledgments}
\paragraph*{Statement on AI use}
The authors have been working on the topics of this paper (bounded mean inference, merging e-values/p-values, Gaffke's conjectures) for years without AI help. In the current version, based on the authors' prior results and new ideas, GPT helped with proof development and numerical experiments. All results and texts that involved GPT help are rewritten by the authors. The authors take full responsibility for its contents.

\bibliographystyle{plainnat}
\bibliography{references}

\clearpage
\appendix

\section{Appendix}
\label{app:proofs}

The appendix contains omitted proofs from the main paper as well as some further results.
\subsection{Proof of \texorpdfstring{\cref{thm:master}}{Theorem \getrefnumber{thm:master}}}
\label{app:proof-master}

\begin{proof}
We use induction on $n$.  The claim is immediate for $n=0$.  Assume it holds
in dimension $n-1$, and write
\[
\mathbf x=(\mathbf y,a),\quad \mathbf y=(x_1,\ldots,x_{n-1}),\quad a=x_n.
\]
We use the conventions $e_j(\mathbf y)=0$ and $\binom{n-1}{j}=0$ when
$j\notin\{0,\ldots,n-1\}$.

Let $T=D_n$.  Under the uniform Dirichlet law,
$T\sim\Beta(1,n)$ with density $n(1-t)^{n-1}$, and, conditionally on $T=t$,
\[
\left(\frac{D_0}{1-t},\ldots,\frac{D_{n-1}}{1-t}\right)
\sim\Dir(1,\ldots,1).
\]
Set
\[
u=\min\{1,1/a\},
\]
where $1/a=\infty$ if $a=0$, and, for $0\le t<u$, set
\[
c(t)=\frac{1-at}{1-t}>0.
\]
Conditioning on $T$ gives the slicing identity
\begin{equation}\label{eq:slicing}
K_n(\mathbf y,a)
=
 n\int_0^u(1-t)^{n-1}
 K_{n-1}\!\left(\frac{\mathbf y}{c(t)}\right)\dd t.
\end{equation}
Indeed, when $at\le1$, the defining event is equivalent to
$\sum_{i=1}^{n-1}y_iD_i'\le c(t)$; when $a>1$ and $t>1/a$, it is impossible.

The Pascal identity for elementary symmetric polynomials is
\begin{equation}\label{eq:pascal}
e_k(\mathbf y,a)=e_k(\mathbf y)+a e_{k-1}(\mathbf y).
\end{equation}
For any $j$, the induction hypothesis applied to $\mathbf y/c$ yields
\[
K_{n-1}(\mathbf y/c)e_j(\mathbf y)
=
c^jK_{n-1}(\mathbf y/c)e_j(\mathbf y/c)
\le\binom{n-1}{j}c^j.
\]
Using this inequality for $j=k$ and $j=k-1$ in \eqref{eq:pascal}, and then
substituting into \eqref{eq:slicing}, gives
\begin{align*}
e_k(\mathbf y,a)K_n(\mathbf y,a)
&\le n\int_0^u(1-t)^{n-1}
\left\{
\binom{n-1}{k}c(t)^k
+a\binom{n-1}{k-1}c(t)^{k-1}
\right\}\dd t\\
&=\binom nk\int_0^u
\left\{
(n-k)(1-t)^{n-k-1}(1-at)^k
+ka(1-t)^{n-k}(1-at)^{k-1}
\right\}\dd t.
\end{align*}
Terms with zero binomial coefficient are omitted at $k=0$ and $k=n$.
The integrand is
\[
-\frac{\dd}{\dd t}
\left[(1-t)^{n-k}(1-at)^k\right].
\]
Therefore
\begin{align*}
e_k(\mathbf y,a)K_n(\mathbf y,a)
&\le
\binom nk
\left\{1-(1-u)^{n-k}(1-au)^k\right\}\le\binom nk,
\end{align*}
where a factor with exponent zero is interpreted as one, and the last
inequality follows because the resulting product lies in $[0,1]$.  This
completes the induction.
\end{proof}

\subsection{Proof of \texorpdfstring{\cref{thm:Kad-optimality}}{Theorem \getrefnumber{thm:Kad-optimality}}}
\label{app:proof-Kad-optimality}

We first establish a mixture representation and the level-set and
domination properties of $\Kad$. We then prove its validity and
optimality.

\begin{lemma}\label{lem:ad-mixture}
Every probability law $\Pp$ on $[0,\infty)$ with mean at most one is a mixture
of laws of the following forms:
\begin{enumerate}[label=(\roman*)]
\item point masses $\mathbb{D}_a$ at $a$, with $0\le a\le1$;
\item mean-one two-point laws
\begin{equation}\label{eq:ad-Qab}
\mathbb{D}_{a,b}=\frac{b-1}{b-a}\,\mathbb{D}_a+
        \frac{1-a}{b-a}\,\mathbb{D}_b,
\quad 0\le a<1<b.
\end{equation}
\end{enumerate}
If $\Pp$ has mean exactly one, the only point mass required is $\mathbb{D}_1$.
\end{lemma}

\begin{proof}
Let $\Pp_-$ and $\Pp_+$ be the restrictions of $\Pp$ to $[0,1)$ and
$(1,\infty)$, and put
\[
d=\int(1-a)\,\Pp_-(\dd a),\quad
\tilde d=\int(b-1)\,\Pp_+(\dd b).
\]
The mean constraint gives $\tilde d\le d$.  If $\tilde d=0$, the law is supported on
$[0,1]$ and is already a mixture of point masses.  Otherwise put
$\theta=\tilde d/d$, $\widetilde{\Pp}_-=\theta\Pp_-$, and define a finite mixing
measure on pairs by
\[
\mathbb{M}(\dd a,\dd b)=\frac{b-a}{\tilde d}\,
\widetilde{\Pp}_-(\dd a)\Pp_+(\dd b).
\]
Integrating the low-point mass in \eqref{eq:ad-Qab} against $\mathbb{M}$ gives
$\widetilde{\Pp}_-$, and integrating the high-point mass gives $\Pp_+$.
The remaining measure $(1-\theta)\Pp_-$, together with any atom at one, is a
mixture of point masses in $[0,1]$.  If the mean is one, then $d=\tilde d$ and no
point mass below one remains.
\end{proof}

Write $a=\min\{x,y\}$ and $b=\max\{x,y\}$. Direct calculation gives
\begin{equation*}
K_2(a,b)=
\begin{cases}
1, & b\le1,\\[1.5mm]
1-\dfrac{(b-1)^2}{b(b-a)}, & 0\le a<1<b,\\[3mm]
\dfrac1{ab}, & 1\le a.
\end{cases}
\end{equation*}

For $0\le a<1<b$, put $f(t)=t^2-b(1+a)t+ab$.
Since $f(0)=ab\ge0$ and $f(1)=1-b<0$, its smaller root lies in $[0,1)$
and its larger root $\tau(a,b)$ exceeds one. Moreover,
$f(b)=ab(1-b)\le0$, so $\tau(a,b)\ge b$, with equality exactly when $a=0$.

On the mixed quadrant, $\Kad(a,b)=K_2(0,\tau(a,b))$. Its level sets are
particularly simple.  For $0<\alpha<1$, put
\begin{equation*}
s_\alpha=\sqrt{1-\alpha},\quad
B_\alpha(a)=\frac1{(1-s_\alpha)(1+a s_\alpha)}.
\end{equation*}

\begin{lemma}\label{lem:ad-levelset}
Fix $\alpha\in(0,1)$. For $0\le a<1<b$,
\[
\Kad(a,b)\le\alpha\quad\Longleftrightarrow\quad b\ge B_\alpha(a).
\]
Moreover, $B_\alpha$ is decreasing,
$B_\alpha(0)=1/(1-s_\alpha)$, and $B_\alpha(1)=1/\alpha$.
The function $\Kad$ is symmetric, continuous and coordinatewise nonincreasing.
\end{lemma}

\begin{proof}
Put $T=1/(1-s_\alpha)>1$.  Since the smaller root of
\eqref{eq:ad-tau-poly} is below one,
\[
\tau(a,b)\ge T
\quad\Longleftrightarrow\quad
T^2-b(1+a)T+ab\le0,
\]
which is equivalent to $b\ge B_\alpha(a)$.  On the other hand,
\[
\Kad(a,b)=1-\left(1-\frac1{\tau(a,b)}\right)^2,
\]
so $\Kad(a,b)\le1-s_\alpha^2$ is equivalent to $\tau(a,b)\ge T$.
The threshold $B_\alpha(a)$ decreases with $a$. At $a=1$, it agrees
with the threshold $b=1/(\alpha a)$ in the upper quadrant. Together with
$\Kad=1$ on $[0,1]^2$, these level sets show that increasing either
coordinate cannot increase $\Kad$. Symmetry follows from its definition.

Each branch is continuous on its domain. As $b\downarrow1$ with $a<1$
fixed, $\tau(a,b)\to1$, so the mixed branch tends to one. As $a\uparrow1$
with $b>1$ fixed, $\tau(a,b)\to r=b+\sqrt{b(b-1)}$. The identity
$r^2-2br+b=0$ gives $2/r-1/r^2=1/b$, matching the upper branch.
At $(1,1)$ all branches tend to one, and at $a=0$ the mixed formula is
continuous with $\tau(0,b)=b$. This proves continuity on $[0,\infty)^2$
and, in particular, Borel measurability.
\end{proof}

\begin{lemma} \label{lem:ad-dominance}
For all $x,y\ge0$, $\Kad(x,y)\le K_2(x,y)$. The inequality is strict
if and only if $0<\min\{x,y\}<1<\max\{x,y\}$.
\end{lemma}

\begin{proof}
Only the mixed quadrant requires calculation.  Write $t=\tau(a,b)$.  From
\eqref{eq:ad-tau-poly},
\[
b=\frac{t^2}{(1+a)t-a}.
\]
A direct simplification yields
\begin{equation}\label{eq:ad-gap-factor}
K_2(a,b)-\Kad(a,b)
=\frac{a(1-a)(t-1)^2}{t(b-a)((1+a)t-a)}.
\end{equation}
This is nonnegative and is positive precisely for $0<a<1<b$.
\end{proof}

\begin{figure}[t]
\centering
\includegraphics[width=.48\linewidth]{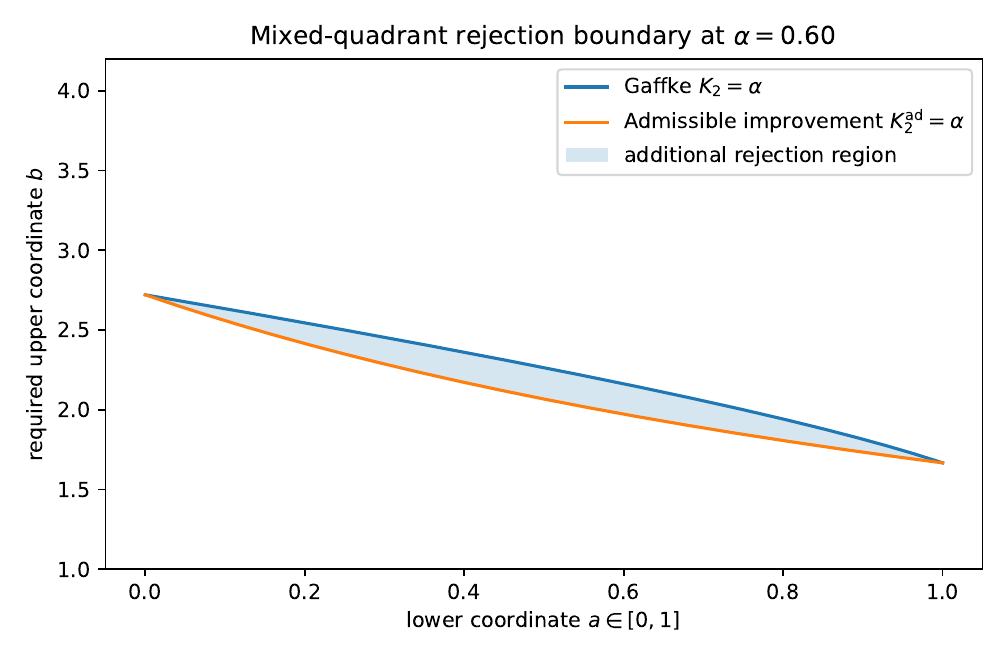}
\includegraphics[width=.48\linewidth]{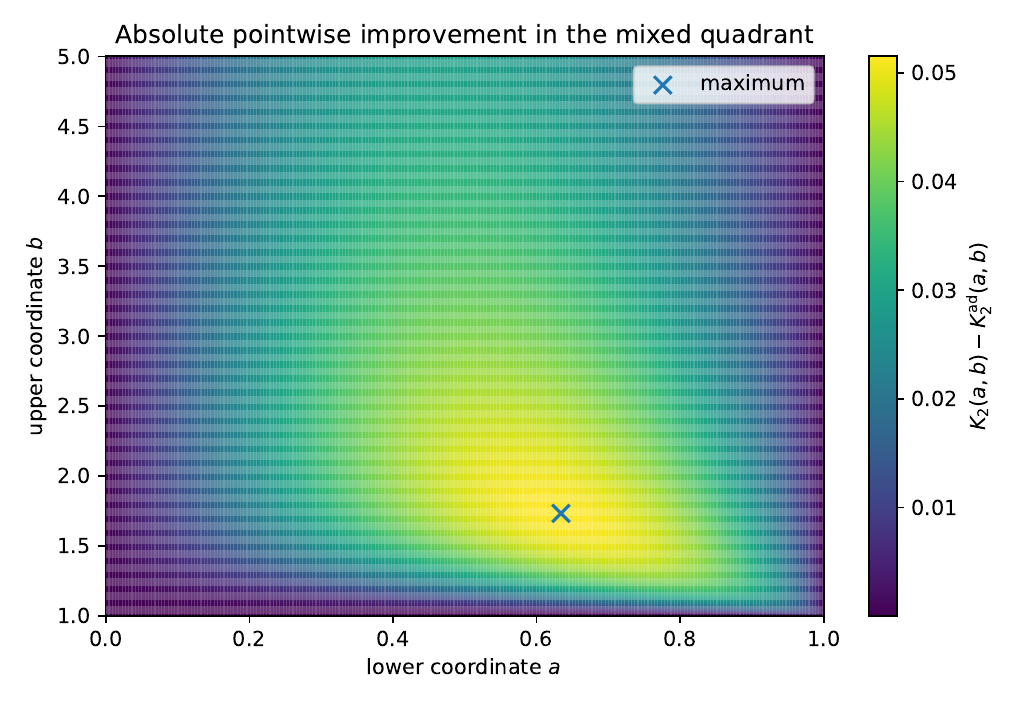}
\caption{Left: at a fixed level, $\Kad$ rejects on a strictly larger part of
the mixed quadrant, with the two boundaries agreeing at $a=0$ and $a=1$.
Right: the pointwise reduction $K_2-\Kad$ in the mixed quadrant.}
\label{fig:admissibility-geometry}
\end{figure}

\Cref{fig:admissibility-geometry} illustrates the level-set expansion and
pointwise gap established above. We now prove the theorem.

\begin{proof}[Proof of \cref{thm:Kad-optimality}]
\emph{Validity.}
By \cref{lem:ad-levelset,lem:ad-dominance}, $\Kad$ is continuous and
coordinatewise nonincreasing, satisfies $\Kad\le K_2$, and is strictly
smaller on the mixed region stated in the theorem. To establish validity,
we verify \eqref{eq:merger-def} for independent e-variables $E_1,E_2$.
The case $\alpha=0$ follows from $\Kad(x,y)>0$ for finite $x,y$,
and $\alpha=1$ is immediate. Fix $0<\alpha<1$.
By \cref{lem:ad-mixture}, each marginal law is a mixture of point masses
$\mathbb{D}_a$, $0\le a\le1$, and the mean-one two-point laws
$\mathbb{D}_{a,b}$ in \eqref{eq:ad-Qab}. Independence makes the joint law
a mixture of products of these component laws. Since rejection
probabilities average over this mixture, it suffices to check each product.

If one component is $\mathbb{D}_a$ with $a\le1$, write $Z$ for the other
input, whose mean is at most one. By \cref{lem:ad-levelset} and the
identity $\Kad(1,z)=\min\{1,1/z\}$,
\[
 \Pp\{\Kad(a,Z)\le\alpha\}
 \le\Pp\{\Kad(1,Z)\le\alpha\}
 =\Pp\{Z\ge1/\alpha\}
 \le\alpha\E[Z]\le\alpha.
\]
This includes zero inputs and pairs of point masses; symmetry handles a
point mass in the second coordinate.

It remains to consider
\[
\mathbb{C}_i=\mathbb{D}_{a_i,b_i},
\quad
0\le a_i<1<b_i,
\quad i=1,2.
\]
Let $X_1,X_2$ be independent with these laws, and write
\[
p_i=\Pp(X_i=b_i)
=\frac{1-a_i}{b_i-a_i},
\quad
r_i=\Pp(X_i=a_i)
=1-p_i
=\frac{b_i-1}{b_i-a_i}.
\]
Thus the four possible states are
\[
(a_1,a_2),\quad
(b_1,a_2),\quad
(a_1,b_2),\quad
(b_1,b_2).
\]
Set
\(
s=\sqrt{1-\alpha}\in(0,1).
\)
The low--low state is not rejected, since
\(
\Kad(a_1,a_2)=1>\alpha.
\)
Moreover, because $\Kad$ is coordinatewise nonincreasing, the rejection set
among the four states is upward closed.
We distinguish the possible numbers of rejected side states
$(b_1,a_2)$ and $(a_1,b_2)$.

If neither side state is rejected, the rejection probability is zero
unless the high-high state is rejected. In the latter case it is
\[
p_1p_2\le\frac1{b_1b_2}=\Kad(b_1,b_2)\le\alpha.
\]
Suppose exactly one side state, say $(b_1,a_2)$, is rejected.  Then
$b_1\ge B_\alpha(a_2)$.  Since
\[
\alpha B_\alpha(a)=\frac{1+s}{1+as}\ge1,
\]
we have $p_1\le1/b_1\le\alpha$.  The high-high state is also rejected, so the
total rejection probability is $p_1$.

It remains to consider the case in which both side states are rejected.  The
side conditions are
\begin{equation*}
b_1(1-s)(1+a_2s)\ge1,\quad
b_2(1-s)(1+a_1s)\ge1.
\end{equation*}
The denominators below are positive: for example,
$1-a_1(1-s)(1+a_2s)>s^2>0$.
Using $b_i=(1-a_ir_i)/(1-r_i)$, the first inequality implies
\begin{equation*}
r_1\ge s\frac{1-a_2(1-s)}{1-a_1(1-s)(1+a_2s)},
\end{equation*}
and symmetrically for $r_2$.  Multiplying the two bounds, the factor
multiplying $s^2$ is at least one because
\begin{align*}
&[1-a_2(1-s)][1-a_1(1-s)] -[1-a_1(1-s)(1+a_2s)]
       [1-a_2(1-s)(1+a_1s)]\\
&=a_1a_2s(1-s)
  [2-(1-s)(a_1+a_2+a_1a_2s)]\ge0.
\end{align*}
Indeed, $a_1+a_2+a_1a_2s\le2+s$ and
$(1-s)(2+s)=2-s-s^2<2$.  Hence $r_1r_2\ge s^2$.  All states except low-low
are rejected, and their total probability is
$1-r_1r_2\le1-s^2=\alpha$.
Averaging over the component laws proves validity.

\emph{Optimality.}
Write $R_q=\{(x,y):K_2(x,y)\le q\}$. We show that, for every
$\mathbf{z}\in[0,\infty)^2$ and $0\le q<\Kad(\mathbf{z})$, there is a product null law
$\Pp$ such that
\[
 \Pp\{(E_1,E_2)\in R_q\cup\{\mathbf{z}\}\}>q.
\]
If $\mathbf{z}=(x,y)$ with $x,y\le1$, the point mass at $\mathbf{z}$ is such a law.
If $x,y\ge1$, take independent variables equal to $x$ and $y$ with
probabilities $1/x$ and $1/y$, respectively, and zero otherwise.
A coordinate equal to one is deterministically one. The atom at $\mathbf{z}$
has probability $1/(xy)=\Kad(\mathbf{z})>q$.

Now let $\mathbf{z}=(a,b)$ with $0\le a<1<b$, and write $\tau=\tau(a,b)$.
For $q=0$, the product of $\mathbb{D}_{a,t}$ and $\mathbb{D}_{0,b}$,
with any $t>1$, assigns positive probability to $\mathbf{z}$ and gives the required
witness. For $0<q<\Kad(a,b)$, the function
\[
 K_2(0,t)=\frac2t-\frac1{t^2},\qquad t>1,
\]
strictly decreases from one to zero. Since $\Kad(a,b)=K_2(0,\tau)$,
there is a unique $t_q>\tau$ with $K_2(0,t_q)=q$.
Take independent mean-one variables
\[
 X\sim\mathbb{D}_{a,t_q},\qquad Y\sim\mathbb{D}_{0,b}.
\]
The set $R_q\cup\{(a,b)\}$ contains the three support points
$(a,b)$, $(t_q,0)$ and $(t_q,b)$: the first is added explicitly,
$K_2(t_q,0)=q$, and $K_2(t_q,b)=1/(bt_q)<1/t_q<q$.
The remaining point $(a,0)$ is excluded because $K_2(a,0)=1>q$.
Consequently,
\[
 \Pp\{(X,Y)\in R_q\cup\{(a,b)\}\}
 =\pi(t_q),\qquad
 \pi(t)=1-\frac{(b-1)(t-1)}{b(t-a)}.
\]
For $t>1$,
\[
 \pi(t)-K_2(0,t)
 =\frac{(t-1)[t^2-b(1+a)t+ab]}{bt^2(t-a)}.
\]
This is positive at $t=t_q$, since $t_q$ lies above the larger root
$\tau$. Thus $\pi(t_q)>q$, as required. Symmetry of $K_2$ and $\Kad$
allows the two marginal laws to be interchanged, giving the witness
for $\mathbf{z}=(b,a)$ as well.

Finally, let $F$ be any e-to-p merger with $F\le K_2$. If
$F(\mathbf{z})=q<\Kad(\mathbf{z})$ at some point, then
\[
 R_q\cup\{\mathbf{z}\}\subseteq\{(x,y):F(x,y)\le q\}.
\]
The witness just constructed contradicts validity of $F$. Hence
$F\ge\Kad$, proving optimality without imposing symmetry or monotonicity
on $F$. We have already shown that $\Kad$ is valid and lies below $K_2$.
Any valid merger below $\Kad$ must therefore equal it, and every other
valid merger below $K_2$ is dominated by $\Kad$. This proves the
admissibility and uniqueness stated after the theorem.
\end{proof}

\subsection{Optimal caps and maximum improvement}
\label{app:proof-ad-maxgap}

The leastness result identifies the smallest valid cap on a symmetric
upper set. We then quantify the largest pointwise reduction in the
Gaffke p-value.

\begin{corollary}\label{cor:golden-cap}
For $0\le a<1<b$, let
\[
\mathcal U_{a,b}=\{(x,y):\min\{x,y\}\ge a,\ \max\{x,y\}\ge b\}.
\]
The rule obtained by replacing $K_2$ with
$\min\{K_2,\Kad(a,b)\}$ on $\mathcal U_{a,b}$ is valid, and no smaller cap can
lower the value at $(a,b)$ while remaining valid and pointwise below $K_2$.
In particular,
\[
\Kad(1/2,2)=\frac{\sqrt5-1}{2}.
\]
\end{corollary}

\begin{proof}
By symmetry and coordinatewise monotonicity,
$\Kad(x,y)\le\Kad(a,b)$ on $\mathcal U_{a,b}$, while
$\Kad\le K_2$ everywhere.  The capped rule is therefore pointwise at least
$\Kad$ and is valid.  Sharpness follows from \cref{thm:Kad-optimality}.
At $(a,b)=(1/2,2)$, the polynomial in \eqref{eq:ad-tau-poly} is
$t^2-3t+1$, whose larger root is $(3+\sqrt5)/2$.
\end{proof}

\begin{proposition}\label{prop:ad-maxgap}
We have the following maximum improvement:
\[
\max_{(x,y)\in[0,\infty)^2}\{K_2(x,y)-\Kad(x,y)\}=\frac{2\sqrt3-3}{9}\approx 0.0516.
\]
The maximum occurs, up to permutation, at
\[
a_*=(3-\sqrt3)/2,\quad b_*=\sqrt3.
\]
\end{proposition}

\begin{proof}
By \cref{lem:ad-dominance}, the gap is zero outside the mixed region.
Put $s=1-1/\tau(a,b)\in(0,1)$. Then $\Kad=1-s^2$ and
$b=1/[(1-s)(1+as)]$. This parameterizes the mixed region by
$0\le a<1$ and $0<s<1$. Substitution into \eqref{eq:ad-gap-factor} yields
\[
\Delta(a,s)=\frac{a(1-a)s^2(1-s)}{1-a(1-s)(1+as)}.
\]
For fixed $s$,
\[
\frac{\partial\Delta}{\partial a}
=\frac{s^2(1-s)(1-a+as)[1-a(1+s)]}
       {[1-a(1-s)(1+as)]^2},
\]
so the unique maximizer is $a=1/(1+s)$.  There
\[
\Delta_{\max}(s)=\frac{s^2(1-s)}{1+3s},\quad
\Delta_{\max}'(s)=-\frac{2s(3s^2-1)}{(1+3s)^2}.
\]
The maximum occurs at $s=1/\sqrt3$, which gives the stated value and point.
\end{proof}

\subsection{Proof of \texorpdfstring{\cref{thm:sm-admissible-completion}}{Theorem \getrefnumber{thm:sm-admissible-completion}}}
\label{app:proof-sm-admissible-completion}

\begin{proof}
Each $F_j$ is a symmetric nonincreasing merger and $F_j\le F_{j-1}$.
Both limits in \eqref{eq:sm-admissible-limit} are Borel: the second can
be written as $\sup_{k\ge1}L(\mathbf{x}+k^{-1}\one)$. They preserve symmetry
and monotonicity, with
\begin{equation}\label{eq:sm-admissible-order}
 F_\star\le L\le F_j\le F_0.
\end{equation}
For independent e-variables $\mathbf{E}=(E_1,\ldots,E_n)$, $0\le\alpha<1$,
and $0<\eta<1-\alpha$, continuity from below gives
\[
 \Pp\{L(\mathbf{E})\le\alpha\}
 \le\lim_{j\to\infty}\Pp\{F_j(\mathbf{E})<\alpha+\eta\}
 \le\alpha+\eta.
\]
Letting $\eta\downarrow0$ proves validity of $L$.

Validity after the right regularization needs an additional argument.
For $h>0$, take mutually independent
$B_i\sim\operatorname{Bernoulli}((1+h)^{-1})$, independent of $\mathbf{E}$.
The variables $Z_i=B_i(E_i+h)$ are independent e-variables. Validity
of $L$, restricted to the event $B_1=\cdots=B_n=1$, gives
\[
 (1+h)^{-n}\Pp\{L(\mathbf{E}+h\one)\le\alpha\}
 \le\Pp\{L(\mathbf{Z})\le\alpha\}\le\alpha.
\]
Since $F_\star(\mathbf{E})\ge L(\mathbf{E}+h\one)$, its rejection probability
is at most $\alpha(1+h)^n$. Letting $h\downarrow0$ proves validity.
These auxiliary variables are used only in the proof; the rule in
\eqref{eq:sm-admissible-limit} is deterministic.

By monotonicity and its definition,
\[
 \lim_{h\downarrow0}F_\star(\mathbf{x}+h\one)
 =\sup_{h,r>0}L(\mathbf{x}+(h+r)\one)=F_\star(\mathbf{x}).
\]
For any sequence $\mathbf{y}_k\to \mathbf{x}$ in $[0,\infty)^n$, eventually
$\mathbf{y}_k\le \mathbf{x}+h\one$ for each fixed $h>0$. Hence
$\liminf_k F_\star(\mathbf{y}_k)\ge F_\star(\mathbf{x}+h\one)$;
letting $h\downarrow0$ proves lower semicontinuity, including on the
boundary of the orthant.

Finally, suppose that a symmetric nonincreasing merger
$G\le F_\star$ is strict at $\mathbf{x}$. Choose a rational
$t$ with $G(\mathbf{x})<t<F_\star(\mathbf{x})$. Right continuity supplies
$h>0$ such that $F_\star(\mathbf{x}+h\one)>t$. Choose
$\mathbf{u}\in\mathbb Q_{>0}^n$ with $x_i<u_i<x_i+h$ for every $i$.
Then $F_\star(\mathbf{u})>t$ whereas $G<t$ throughout $U_{\mathbf{u}}$.
At the stage that processes $(\mathbf{u},t)$, the proposed cap is pointwise
at least $G$, by \eqref{eq:sm-admissible-order}. It is therefore
valid and must be accepted. This forces
$F_\star(\mathbf{u})\le F_j(\mathbf{u})\le t$, a contradiction.
\end{proof}

\subsection{Proof of \texorpdfstring{\cref{thm:sm025-seed}}{Theorem \getrefnumber{thm:sm025-seed}}}
\label{app:proof-sm025-seed}

\begin{proof}
Writing $x_{(1)}\le\cdots\le x_{(n)}$ for the ordered coordinates gives
\begin{equation*}
 B_n:=U_{\mathbf{z}_n}
 =\{\mathbf{x}:x_{(1)}\ge a_n,\ x_{(2)}\ge1,\ x_{(n)}\ge56\}.
\end{equation*}
The set $B_n$ is closed, symmetric, and upward closed. Consequently,
\(M_n\) is Borel, symmetric, and coordinatewise
nonincreasing. It remains to verify validity and strict improvement.

First we localize every actual modification. Let $c_n=\left(\frac{17}{16}\right)^{1/m}$. We claim that
$M_n(\mathbf{x})<K_n(\mathbf{x})$ implies
\begin{equation}\label{eq:sm025-localization}
 \frac12\le x_{(1)}<1,\qquad
 1\le x_{(2)}\le\cdots\le x_{(n-1)}\le c_n,\qquad
 56\le x_{(n)}\le57.
\end{equation}
Indeed, $a_n>1/2$. If $\mathbf{x}\in B_n$ and $x_{(1)}\ge1$, then
$K_n(\mathbf{x})\le1/56<q$ by \cref{eq:product-rule}.
By \eqref{eq:Kexp}, we have the following neutral-coordinate identity: for $1\le r\le n$,
\begin{equation}\label{eq:ad-neutral}
K_n(x_1,\ldots,x_r,1,\ldots,1)=K_r(x_1,\ldots,x_r).
\end{equation}
If $x_{(n)}>57$, monotonicity and \cref{eq:ad-neutral} give
$K_n(\mathbf{x})\le K_2(a_n,57)<q$, since
\[
 K_2(a,57)-q=\frac{1271-2223a}{2280(57-a)},\qquad
 a_n\ge\frac{12767}{21840}>\frac{1271}{2223}.
\]
Finally, if $x_{(n-1)}>c_n$, the same argument gives
$K_n(\mathbf{x})\le K_3(a_n,c_n,56)<q$. To verify the last inequality, put
$t=c_n-1$ and $\bar a=163/273-t/2$. The identity $(1+t)^m=17/16$
implies
\[
 \frac1{17m}\le t\le\frac1{16},\qquad a_n\ge\bar a.
\]
Here the lower bound follows from
$1/16\le mt(1+t)^{m-1}\le17mt/16$.
Integration in \cref{eq:Kexp} gives
\[
 K_3(a,c,b)=\frac1{acb}-\frac{(1-a)^3}{a(c-a)(b-a)},
 \qquad 0<a<1<c,b.
\]
Substitution yields the exact certificate
\begin{equation*}
 K_3(\bar a,1+t,56)-q
 =\frac{t\bigl[-3047/273-(41429/780)t+(259/20)t^2\bigr]}
 {56(1+t)(1+t-\bar a)(56-\bar a)}<0.
\end{equation*}
The denominator is positive, and the bracket is at most
$-3047/273+259/(20\cdot16^2)<0$. Monotonicity in $a$ proves the
claim for $a_n$, and hence \eqref{eq:sm025-localization}.

We prove validity directly from the elementary-mixture reduction.
Throughout the region in \eqref{eq:sm025-localization}, monotonicity and
the neutral-coordinate identity give
\[
 K_n(\mathbf{x})\le A:=K_2(1/2,56)=\frac{83}{3108},
 \qquad q<A<\left(\frac{16}{17}\right)^2.
\]
For $\alpha<q$ the rejection event is unchanged from that of $K_n$.
It is also unchanged for $\alpha\ge A$, since every actual modification
is localized to this region. Fix $q\le\alpha<A$.

By \cref{lem:ad-mixture}, it suffices to consider a product of point
masses $\mathbb{D}_a$, $0\le a\le1$, and mean-one two-point laws
$\mathbb{D}_{a,b}$, $0\le a<1<b$. If its support contains no new rejection,
validity of $K_n$ applies. Otherwise its support intersects the region
in \eqref{eq:sm025-localization}. Enlarge the rejection event by adding
all support points in that region. We show that even this enlarged
event has probability at most $q$.

Let $H$ be the set of indices of nondegenerate two-point marginals
whose high endpoint exceeds $c_n$. At a support point in the localized
region, one coordinate has value in $[56,57]$, hence its index belongs
to $H$. Every other index in $H$ must be in its low state, since all
remaining coordinates except the minimum are at most $c_n$. The
region has exactly one coordinate below one, so
\[
 1\le |H|\le2.
\]
All possible values of coordinates outside $H$ are at most $c_n$;
this includes deterministic coordinates equal to one or below one.

If $H=\{j\}$, its high endpoint $b_j$ lies in $[56,57]$. When this
coordinate takes its low value, monotonicity and the product formula
give
\[
 K_n(\mathbf{x})\ge K_n(1,c_n,\ldots,c_n)
 =c_n^{-(n-1)}\ge\left(\frac{16}{17}\right)^2>A>\alpha.
\]
Such a point is also outside the localized region, since no coordinate
can reach $56$. Thus every rejection requires the high state of
coordinate $j$, whose probability is at most $1/b_j\le1/56<q$.

Suppose $|H|=2$, and denote those two coordinates by $X,Y$.
Replacing the other $n-2$ coordinates by $c_n$ and repeatedly applying
\cref{prop:append-coordinate} yields, on the entire product support,
\begin{equation*}
 K_n(\mathbf{x})\ge c_n^{-(n-2)}K_2(X,Y)
          =\frac{16}{17}K_2(X,Y).
\end{equation*}
Put $t=17\alpha/16<1$. Every original rejection therefore projects
into $\{K_2(X,Y)\le t\}$. Every support point in the localized region
projects onto one of the two mixed side states, with low value in
$[1/2,1)$ and high value in $[56,57]$. Thus the enlarged rejection
event is contained in the inverse image of the original projected
event together with one or both of these mixed sides. The low--low
state is excluded, since $K_2=1>t$ there.

If both mixed sides are added, each of the two low endpoints is at
least $1/2$ and each high endpoint is at least $56$. Each high-state
probability is therefore at most
\[
 \frac{1-1/2}{56-1/2}=\frac1{111}.
\]
By independence, the probability of all states other than low--low is
at most $1-(110/111)^2<q$.

If only one mixed side is added, write it as $(a,b)$, where
\[
 X\in\{a,v\},\quad Y\in\{u,b\},\qquad
 \frac12\le a<1<v,\quad0\le u<1,\quad56\le b\le57.
\]
If the opposite side $(v,u)$ does not belong to
$\{K_2\le t\}$, every projected rejection requires $Y=b$ and has
probability at most $(1-u)/(b-u)\le1/b\le1/56<q$.
Otherwise put $s=\sqrt{1-t}$. By \cref{lem:ad-dominance} and
\cref{lem:ad-levelset},
\[
 v\ge\frac1{(1-s)(1+us)}.
\]
The function $v\mapsto(v-1)/(v-a)$ is nondecreasing, so the low--low
probability satisfies
\begin{align*}
 \Pp(X=a,Y=u)
 &=\frac{v-1}{v-a}\frac{b-1}{b-u}\\
 &\ge
 \frac{(b-1)s[1-u(1-s)]}
 {(b-u)[1-a(1-s)(1+us)]}.
\end{align*}
The denominator is positive. The numerical bounds needed below are
\[
 s^2>1-\frac{17A}{16}
 =\frac{48317}{49728}>
 \left(\frac{273}{277}\right)^2,
 \qquad\frac{273}{277}>\frac{56}{57}.
\]
In particular, $b(1-s)<1$. The factor
$[1-u(1-s)]/(b-u)$ is nondecreasing in $u$, since its derivative has
numerator $1-b(1-s)>0$. The reciprocal factor
$[1-a(1-s)(1+us)]^{-1}$ is also nondecreasing in $u$.
Evaluating at $u=0$, then using $a\ge1/2$ and $b\ge56$, gives
\[
 \Pp(X=a,Y=u)
 \ge\frac{(b-1)s}{b[1-a(1-s)]}
 \ge\frac{110s}{56(1+s)}\ge\frac{39}{40}.
\]
The last inequality is equivalent to $s\ge273/277$.
Since low--low cannot reject, the probability of the enlarged event
is at most $1/40=q\le\alpha$.

This proves the required bound for every elementary product law,
including deterministic components. Averaging over the independent
marginal mixtures proves validity of $M_n$ for all independent
nonnegative inputs with means at most one.

For strictness, let $\mathbf{z}_n=(a_n,1,\ldots,1,56)$. Then
\begin{equation}\label{eq:sm025-witness}
 K_n(\mathbf{z}_n)-q=K_2(a_n,56)-q
   =\frac{163-273a_n}{280(56-a_n)}>0.
\end{equation}
Continuity of $K_n$ supplies a small open box immediately above $\mathbf{z}_n$
on which $K_n>q$. This box is contained in $B_n$, so $M_n=q<K_n$
there.
\end{proof}

\subsection{Different admissible completions of the same initial merger}
\label{app:distinct-admissible-completions}

We give a three-dimensional example in which two enumerations in
\cref{thm:sm-admissible-completion}, starting from the same merger
$F_0\equiv1$, yield incomparable admissible completions.
All admissibility statements here are relative to the symmetric monotone
class.

Put
\[
 \mathbf{u}=(1/20,2,3),\qquad \mathbf{v}=(1/2,2,2),\qquad
 H_A=T_{\mathbf{u},1/3}F_0,\qquad H_B=T_{\mathbf{v},5/18}F_0.
\]
These are the mergers proposed at the first step under the two enumerations
specified below. Thus $H_A$ equals $1/3$ on $U_{\mathbf{u}}$ and one elsewhere, whereas $H_B$
equals $5/18$ on $U_{\mathbf{v}}$ and one elsewhere. Both are Borel, symmetric,
and coordinatewise nonincreasing. We first verify their validity.

For $A_0=U_{(0,2,3)}$, round each input down to zero, two, or three,
according as it lies in $[0,2)$, $[2,3)$, or $[3,\infty)$.
This preserves membership in $A_0$, preserves independence, and does not
increase any marginal mean. By \cref{lem:ad-mixture}, each rounded
marginal is a mixture of $\mathbb{D}_0$, $\mathbb{D}_{0,2}$, and $\mathbb{D}_{0,3}$.
Consider a product with $k$ marginals $\mathbb{D}_{0,2}$ and $\ell$ marginals
$\mathbb{D}_{0,3}$, the remaining marginals being $\mathbb{D}_0$, so $k+\ell\le3$.
Its probability of $A_0$ is
\[
 \begin{cases}
  0,&\ell=0,\\
  \dfrac13(1-2^{-k}),&\ell=1,\\
  \dfrac19+\dfrac49(1-2^{-k}),&\ell=2,\\
  \dfrac7{27},&\ell=3.
 \end{cases}
\]
Every value is at most $1/3$, and equality holds when $(k,\ell)=(1,2)$.
Averaging over the marginal mixtures gives $\capn_3(A_0)=1/3$.
Since $U_{\mathbf{u}}\subseteq A_0$, we have $\capn_3(U_{\mathbf{u}})\le1/3$.

For $U_{\mathbf{v}}$, instead round each input down to zero, $1/2$, or two,
according as it lies in $[0,1/2)$, $[1/2,2)$, or $[2,\infty)$.
The rounded event is that all three coordinates are at least $1/2$
and at least two equal two. The elementary marginals are now
$\mathbb{D}_0$, $\mathbb{D}_{1/2}$, $\mathbb{D}_{0,2}$, and $\mathbb{D}_{1/2,2}$.
A product containing $\mathbb{D}_0$ gives probability zero. For the other
products, let $k$ marginals be $\mathbb{D}_{0,2}$ and $\ell$ be $\mathbb{D}_{1/2,2}$,
with the rest equal to $\mathbb{D}_{1/2}$. If $k\ge2$, the probability is
$2^{-k}\le1/4$. If $k=1$, it is
\[
 \frac12\bigl[1-(2/3)^\ell\bigr]\le\frac5{18},
\]
because $\ell\le2$. If $k=0$, it is at most the probability of at
least two successes in three independent Bernoulli$(1/3)$ trials,
namely $7/27<5/18$. The bound $5/18$ is attained when $(k,\ell)=(1,2)$.
Hence $\capn_3(U_{\mathbf{v}})=5/18$.
These capacity bounds prove validity of $H_A$ and $H_B$ by
\cref{prop:ad-capacity}: below their respective cap levels their rejection
sets are empty, and between the cap level and one they are $U_{\mathbf{u}}$ and
$U_{\mathbf{v}}$, respectively.

Next, take independent mean-one variables with laws
\[
 X_1\sim \mathbb{D}_{1/2,3},\qquad
 X_2,X_3\sim \mathbb{D}_{1/20,2},
\]
and write $p=\Pp(X_2=2)=\Pp(X_3=2)=19/39$.
On this support, $U_{\mathbf{v}}$ occurs exactly when $X_2=X_3=2$, while
$U_{\mathbf{u}}\setminus U_{\mathbf{v}}$ occurs when $X_1=3$ and exactly one of $X_2,X_3$
equals two. Consequently,
\begin{equation}\label{eq:distinct-completion-witness}
 \Pp\{\mathbf{X}\in U_{\mathbf{u}}\cup U_{\mathbf{v}}\}
 =p^2+\frac25p(1-p)
 =\frac{57}{169}=\frac13+\frac2{507}>\frac13.
\end{equation}
Every rule below both $H_A$ and $H_B$ rejects this union at level $1/3$.
Thus these two cap candidates have no common valid pointwise lower bound.

Now choose one enumeration whose first pair is $(\mathbf{u},1/3)$ and another
whose first pair is $(\mathbf{v},5/18)$; in each case enumerate all remaining
pairs afterward. Both first pairs belong to
$\mathbb Q_{>0}^3\times(\mathbb Q\cap(0,1))$.
Starting from $F_0\equiv1$, the first proposed caps are valid and hence
accepted, giving $F_1=H_A$ and $F_1=H_B$, respectively.
By \cref{thm:sm-admissible-completion}, the resulting rules $F_A,F_B$
are admissible and satisfy
\[
 F_A\le H_A\le F_0,\qquad F_B\le H_B\le F_0.
\]
If they were equal or comparable, the smaller rule would be a common
valid lower bound of $H_A,H_B$, contradicting
\eqref{eq:distinct-completion-witness}.

Their separation can also be seen at the two generating points:
\[
 \begin{array}{c|cc}
  &\mathbf{u}=(1/20,2,3)&\mathbf{v}=(1/2,2,2)\\\hline
  F_A&\le1/3&\ge57/169\\
  F_B&\ge57/169&=5/18.
 \end{array}
\]
Indeed, symmetry and monotonicity force every valid rule $F$ to satisfy
$\max\{F(\mathbf{u}),F(\mathbf{v})\}\ge57/169$, by applying validity to the union in
\eqref{eq:distinct-completion-witness}. Also,
$U_{\mathbf{v}}\subseteq R_{F_B}(F_B(\mathbf{v}))$ forces
$F_B(\mathbf{v})\ge\capn_3(U_{\mathbf{v}})=5/18$; the reverse bound follows from $F_B\le H_B$.

Thus the same initial merger $F_0\equiv1$ has two incomparable admissible
completions obtained from different enumerations.

\subsection{Self-calibration and least-favorable witnesses}
\label{app:self-calibration-witnesses}

We use the product-null class $\mathcal P_n$, capacity $\capn_n$, and
rejection regions $R_F$ defined in \cref{sec:open-set-improvement}.
For any Borel function $F:[0,\infty)^n\to[0,1]$, define its worst-case
size function and its recalibration by
\begin{equation*}
 c_F(s)=\capn_n(R_F(s)),\qquad
 F^{\mathrm{cal}}(\mathbf{x})=c_F(F(\mathbf{x})),\qquad 0\le s\le1.
\end{equation*}
The function $c_F$ is nondecreasing. By \cref{prop:ad-capacity},
$F$ is an e-to-p merger if and only if $c_F(s)\le s$ for every $s$.

\begin{proposition}\label{prop:cal-optimal}
For every Borel $F:[0,\infty)^n\to[0,1]$, the function
$F^{\mathrm{cal}}$ is an e-to-p merger. Recalibration preserves symmetry
and coordinatewise monotonicity. For every nondecreasing
$h:[0,1]\to[0,1]$,
\begin{equation*}
 h\circ F\text{ is an e-to-p merger}
 \quad\Longleftrightarrow\quad
 h(s)\ge c_F(s)\text{ for every }s\in[0,1].
\end{equation*}
Consequently, $F^{\mathrm{cal}}$ is the pointwise smallest valid
nondecreasing recalibration of $F$, and
$(F^{\mathrm{cal}})^{\mathrm{cal}}=F^{\mathrm{cal}}$.
If $F$ is an e-to-p merger, then $F^{\mathrm{cal}}\le F$.
If $F$ is also admissible in any of the three classes in
\cref{thm:sm-cap-local}, then
\begin{equation}\label{eq:cal-self-calibration}
 c_F(F(\mathbf{x}))=F(\mathbf{x})\qquad\text{for every }\mathbf{x}.
\end{equation}
\end{proposition}

\begin{proof}
Fix $\Pp\in\mathcal P_n$ and $\beta\in[0,1]$. Since $c_F$ is nondecreasing,
$A_\beta=\{s\in[0,1]:c_F(s)\le\beta\}$ is a downward interval.
If it is empty, then $\{F^{\mathrm{cal}}\le\beta\}$ is empty.
Otherwise let $r=\sup A_\beta$. If $r\in A_\beta$, then
\[
 \Pp\{F^{\mathrm{cal}}\le\beta\}
 =\Pp\{F\le r\}\le c_F(r)\le\beta.
\]
If $r\notin A_\beta$, continuity from below gives
\[
 \Pp\{F^{\mathrm{cal}}\le\beta\}
 =\Pp\{F<r\}=\sup_{0\le s<r}\Pp\{F\le s\}
 \le\sup_{0\le s<r}c_F(s)\le\beta.
\]
Thus $F^{\mathrm{cal}}$ is valid without any continuity assumption on
$c_F$. Its Borel measurability and the preservation of the structural
properties follow from the monotonicity of $c_F$.

If $h\circ F$ is valid, the inclusion
$R_F(s)\subseteq R_{h\circ F}(h(s))$ implies $c_F(s)\le h(s)$.
Conversely, $h\ge c_F$ implies $h\circ F\ge F^{\mathrm{cal}}$, so
validity follows by increasing a valid p-value.
For idempotence, put $G=F^{\mathrm{cal}}$. Since
$R_F(F(\mathbf{x}))\subseteq R_G(G(\mathbf{x}))$,
\[
 G(\mathbf{x})=c_F(F(\mathbf{x}))\le c_G(G(\mathbf{x}))\le G(\mathbf{x}),
\]
where the last inequality is validity of $G$.
Finally, validity of $F$ gives $F^{\mathrm{cal}}\le F$.
This recalibration belongs to the same one of the three structural
classes as $F$, so admissibility forces equality.\qedhere
\end{proof}

The identity \eqref{eq:cal-self-calibration} asserts exact calibration at
the values attained by $F$. Below Gaffke's rule, a stronger conclusion
holds at every numerical level, without imposing symmetry or monotonicity.

\begin{proposition}\label{prop:cal-gaffke-dominators}
For every $n\ge1$ and every e-to-p merger $F\le K_n$,
\begin{equation*}
 c_F(\alpha)=\alpha,\qquad 0\le\alpha\le1.
\end{equation*}
For each $0<\alpha<1$, the same product null attains this value for all
such mergers: let $E_1$ equal $1/\alpha$ with probability $\alpha$ and
zero otherwise, and let $E_2=\cdots=E_n=1$ deterministically.
Moreover,
\begin{equation}\label{eq:cal-fixed-regions}
 F(\mathbf{x})=K_n(\mathbf{x})=
 \begin{cases}
  1,&\mathbf{x}\in[0,1]^n,\\
  (\prod_{i=1}^n x_i)^{-1},&\mathbf{x}\in[1,\infty)^n.
 \end{cases}
\end{equation}
\end{proposition}

\begin{proof}
The point-mass bound in \cref{lem:ad-point-mass} forces $F(\mathbf{x})=1$ on
$[0,1]^n$. On $[1,\infty)^n$, it gives
$F(\mathbf{x})\ge(\prod_{i=1}^n x_i)^{-1}=K_n(\mathbf{x})$ by \eqref{eq:product-rule},
so $F\le K_n$ proves \eqref{eq:cal-fixed-regions}.
Under the displayed product null, $F(\mathbf E)$ therefore equals
$\alpha$ with probability $\alpha$ and one otherwise. Its level-$\alpha$
rejection probability is exactly $\alpha$. Validity supplies the matching
upper bound for the capacity. The levels zero and one follow from
validity and the range $[0,1]$, respectively.\qedhere
\end{proof}

In particular, recalibration leaves $K_n$, $\Kad$, $M_n$, and every
admissible completion below $M_n$ unchanged. Since $K_n$ is inadmissible
for $n\ge2$, exact calibration is not sufficient for admissibility.
Its valid dominators can enlarge rejection regions while preserving
their capacity at every level.

We next formulate the obstruction provided by a least-favorable witness.
Fix one of the three classes in \cref{thm:sm-cap-local}, and let $D_{\mathbf{x}}$
be the corresponding forced set in \eqref{eq:ad-forced-set}.
Call $\Pp=\mathbb{C}_1\otimes\cdots\otimes \mathbb{C}_n$ an elementary product null when
$\mathbb{C}_i\in\mathcal{C}_1$, as in \eqref{eq:ad-capacity-extreme}; each marginal
then has at most two support points.
For an e-to-p merger $F$ in the selected class,
\cref{thm:sm-cap-local} and \eqref{eq:ad-capacity-extreme} give the
following equivalent form of admissibility: for every $\mathbf{x}$ and every
$0\le t<F(\mathbf{x})$, there exist $\alpha\in[t,F(\mathbf{x}))$ and an elementary product
null $\Pp$ such that
\begin{equation}\label{eq:cal-local-witness}
 \Pp\bigl(R_F(\alpha)\cup D_{\mathbf{x}}\bigr)>\alpha.
\end{equation}
The reduction to elementary products uses the strict inequality and
does not require the capacity supremum to be attained.
It is sufficient to produce a witness at $\alpha=t$ for every $\mathbf{x},t$,
but the exact criterion allows a higher level below $F(\mathbf{x})$.

Witnesses relative to a given initial merger yield the stronger property
of being its least valid dominator.

\begin{proposition}\label{prop:cal-least-rule}
Fix one of the three classes in \cref{thm:sm-cap-local}, and let $B$ and
$M$ be e-to-p mergers in that class with $M\le B$.
Then $M\le G$ for every e-to-p merger $G\le B$ in the class if and only
if, for every $\mathbf{x}$ and every $0\le t<M(\mathbf{x})$, there exist
$\alpha\in[t,M(\mathbf{x}))$ and an elementary product null $\Pp$ such that
\begin{equation}\label{eq:cal-baseline-witness}
 \Pp\bigl(R_B(\alpha)\cup D_{\mathbf{x}}\bigr)>\alpha.
\end{equation}
Under this condition, $M$ is admissible and is the unique admissible
merger below $B$ in the selected class. In particular, witnesses at
$\alpha=t$ for all $\mathbf{x},t$ suffice.
\end{proposition}

\begin{proof}
Suppose first that $M$ is the least valid merger below $B$ in the class.
For $t<M(\mathbf{x})$, the cap $T_{\mathbf{x},t}B$ cannot be valid, since it is below $B$
and has value $t<M(\mathbf{x})$ at $\mathbf{x}$.
By \cref{thm:sm-cap-local}, there is a violating level
$\alpha\in[t,B(\mathbf{x}))$. However, $M\le T_{\mathbf{x},M(\mathbf{x})}B$: on $D_{\mathbf{x}}$ this follows
from $M(\mathbf{y})\le M(\mathbf{x})$, and elsewhere from $M\le B$.
Thus $T_{\mathbf{x},M(\mathbf{x})}B$ is valid, which excludes violating levels
$\alpha\ge M(\mathbf{x})$ below $B(\mathbf{x})$. Hence $\alpha\in[t,M(\mathbf{x}))$,
and \eqref{eq:ad-capacity-extreme} supplies an elementary product
witness for \eqref{eq:cal-baseline-witness}.

Conversely, if a valid $G\le B$ satisfies $G(\mathbf{x})=t<M(\mathbf{x})$, choose
$\alpha$ and $\Pp$ as in \eqref{eq:cal-baseline-witness}.
The selected structural class implies $G\le t$ on $D_{\mathbf{x}}$, so
\[
 R_B(\alpha)\cup D_{\mathbf{x}}\subseteq R_G(\alpha).
\]
The witness therefore contradicts validity of $G$.
This proves leastness. Every valid merger below $M$ is also below $B$
and hence equals $M$, proving admissibility. Any other valid merger
below $B$ is dominated by $M$ and cannot be admissible.\qedhere
\end{proof}

An exact-size law and a local obstruction witness play different roles.
For $A=R_F(\alpha)$ and $\Pp\in\mathcal P_n$,
\begin{equation*}
 \Pp(A\cup D_{\mathbf{x}})>\alpha
 \quad\Longleftrightarrow\quad
 \Pp(D_{\mathbf{x}}\setminus A)>\alpha-\Pp(A).
\end{equation*}
A law satisfying $\Pp(A)=\alpha$ therefore blocks the cap only when it
also charges the newly added part. A witness need not have exact size
before the cap: the added probability need only exceed the remaining
allowance. The common exact-size laws in
\cref{prop:cal-gaffke-dominators} consequently do not by themselves
establish admissibility.

In the proof of \cref{thm:Kad-optimality}, the witnesses are constructed
against $R_{K_2}(t)$ with $D_{\mathbf{x}}=\{\mathbf{x}\}$, for every $0\le t<\Kad(\mathbf{x})$.
Thus \cref{prop:cal-least-rule} applies with $B=K_2$ and $M=\Kad$ even
in the unrestricted Borel class. For the higher-dimensional completion,
admissibility gives \eqref{eq:cal-local-witness} for the completed rule's
own rejection regions and $D_{\mathbf{x}}=U_{\mathbf{x}}$. It does not supply
\eqref{eq:cal-baseline-witness} against $K_n$, and hence does not imply
that the completion is a least or unique admissible dominator of $K_n$.

\subsection{Proof of \texorpdfstring{\cref{prop:ad-randomized-validity}}{Proposition \getrefnumber{prop:ad-randomized-validity}}}
\label{app:proof-ad-randomized-validity}

\begin{proof}
The equivalence in \eqref{eq:ad-fractional-capacity} follows from Fubini's
theorem.  If $Q$ dominates $F$, then $F(\mathbf{x})\le\alpha$ implies
$Q(\mathbf{x};u)\le\alpha$ for almost every $u$, which gives
\[
\phi^Q_\alpha(\mathbf{x})\ge\one\{F(\mathbf{x})\le\alpha\}.
\]

Now suppose that $Q$ strictly dominates $F$.  Then, for some $\mathbf{x}$,
\[
\mathbb{L}\{u:Q(\mathbf{x};u)<F(\mathbf{x})\}>0.
\]
Since
\[
\{u:Q(\mathbf{x};u)<F(\mathbf{x})\}
=
\bigcup_{\substack{r\in\mathbb Q\\0\le r<F(\mathbf{x})}}
\{u:Q(\mathbf{x};u)\le r\},
\]
there exists some $r<F(\mathbf{x})$ such that
\[
\phi^Q_r(\mathbf{x})>0=\one\{F(\mathbf{x})\le r\}.
\]
Hence a strict randomized dominator would yield, at some level, a valid
fractional test that is pointwise no smaller than the indicator of
$R_F(r)$ and strictly larger somewhere.  Pointwise maximality at every level
therefore rules out every strict randomized dominator.
\end{proof}

\subsection{Proof of \texorpdfstring{\cref{thm:ad-KPi-validity}}{Theorem \getrefnumber{thm:ad-KPi-validity}}}
\label{app:proof-ad-KPi-validity}

\begin{proof}
The pointwise comparison follows from
\eqref{eq:ad-K-product-orthant}.  To prove validity, use
\cref{lem:ad-mixture} independently for each marginal and condition on the
selected elementary components.

If at least one selected component is a point mass $\mathbb{D}_a$ with $a<1$,
then every conditional realization lies outside the upper orthant.  The rule
therefore equals $K_n$, and conditional validity follows from
\cref{thm:VT}.

It remains to consider a component system in which every coordinate is either
$\mathbb{D}_1$ or a mean-one law $\mathbb{D}_{a_i,b_i}$ with
$0\le a_i<1<b_i$.  Among the finitely many conditional support points, the
upper orthant contains exactly one point $\mathbf{x}^\top$: every nondegenerate
coordinate takes $b_i$, while every degenerate coordinate equals one.  Put
\[
q=\Pp\{\mathbf{E}=\mathbf{x}^\top\}
=
\prod_{\substack{1\le i\le n\\E_i\sim \mathbb{D}_{a_i,b_i}}}
\frac{1-a_i}{b_i-a_i}
\]
and
\[
k=K_n(\mathbf{x}^\top)
=
\prod_{\substack{1\le i\le n\\E_i\sim \mathbb{D}_{a_i,b_i}}}\frac1{b_i}.
\]
For each nondegenerate coordinate,
\[
\frac{1-a_i}{b_i-a_i}\le\frac1{b_i},
\]
because this is equivalent to $a_i(b_i-1)\ge0$.  Hence $q\le k$.

Fix $\alpha<k$.  Every non-top support point is coordinatewise below
$\mathbf{x}^\top$, so its Gaffke value is at least $k$; it is therefore not rejected.
At the top point, the randomized value is $Uk$, and hence the conditional
rejection probability is
\[
q\,\Pp(Uk\le\alpha)
=q\frac{\alpha}{k}
\le\alpha.
\]
If $\alpha\ge k$, the top point is always rejected, and all other decisions
are identical to those made by $K_n$.  The entire conditional rejection event
therefore agrees with $\{K_n\le\alpha\}$ and has probability at most
$\alpha$.  Averaging over the latent component choices proves validity for
arbitrary independent e-variables.
\end{proof}

\subsection{Proof of \texorpdfstring{\cref{thm:CI}}{Theorem \getrefnumber{thm:CI}}}
\label{app:proof-CI}

\begin{proof}
Fix $\mathbf x\in[0,1]^n$.  The Dirichlet law has a density on the simplex.
Consequently, a nonconstant linear functional of $\mathbf D$ is atomless, and
every nonempty open subinterval of its range has positive probability.  Thus
$S_{\mathbf x}$ has a continuous CDF that is strictly increasing on its range
unless $\mathbf x=\mathbf0$, in which case $S_{\mathbf x}=0$ almost surely.
Similarly, $T_{\mathbf x}$ has a continuous CDF that is strictly increasing on
its range unless $\mathbf x=\mathbf1$, in which case $T_{\mathbf x}=1$ almost
surely.  Using \eqref{eq:pCDF} and the endpoint conventions for $Q_0$ and
$Q_1$, in all cases
\[
\{\theta\in(0,1):p_{\le}(\theta;\mathbf x)>\alpha_{\mathrm L}\}
=(L_{\alpha_{\mathrm L}}(\mathbf x),1)
\]
and
\[
\{\theta\in(0,1):p_{\ge}(\theta;\mathbf x)>\alpha_{\mathrm U}\}
=(0,U_{\alpha_{\mathrm U}}(\mathbf x)).
\]
Hence \eqref{eq:strict-inversion} has exactly the interval form in
\eqref{eq:CIset}, with the parameter-space boundary point $0$ or $1$ included
only when the sample is respectively all zero or all one.

Suppose first that $0<\mu<1$.  At the true mean,
$p_{\le}(\mu;\mathbf X)=K_n(\mathbf X/\mu)$ is a valid p-value by
\cref{thm:VT}, because $\E[X_i/\mu]=1$.  Likewise,
$p_{\ge}(\mu;\mathbf X)=K_n((\mathbf1-\mathbf X)/(1-\mu))$ is a valid
p-value.  Therefore
\begin{align*}
\Pp\{\mu\notin\calC_{1-\alpha}(\mathbf X)\}
&\le
\Pp\{p_{\le}(\mu;\mathbf X)\le\alpha_{\mathrm L}\}
+\Pp\{p_{\ge}(\mu;\mathbf X)\le\alpha_{\mathrm U}\}\\
&\le\alpha_{\mathrm L}+\alpha_{\mathrm U}.
\end{align*}
If $\mu=0$, then $\mathbf X=\mathbf0$ almost surely and the first line of
\eqref{eq:CIset} contains $0$.  If $\mu=1$, then
$\mathbf X=\mathbf1$ almost surely and the second line contains $1$.  This
proves \eqref{eq:coverage}.

Finally, $D_0>0$ almost surely, so
$S_{\mathbf x}<T_{\mathbf x}$ almost surely.  Also,
$\alpha_{\mathrm L}<1-\alpha_{\mathrm U}$ because
$\alpha_{\mathrm L}+\alpha_{\mathrm U}\le\alpha<1$.  The elementary strict
quantile ordering
\[
Y<Z\ \text{a.s. and }q<r
\quad\Longrightarrow\quad
Q_q(Y)<Q_r(Z)
\]
therefore gives
$L_{\alpha_{\mathrm L}}(\mathbf x)
<U_{\alpha_{\mathrm U}}(\mathbf x)$.  For completeness, if instead
$t=Q_r(Z)\le Q_q(Y)$, then
$\Pp(Z\le t)\le\Pp(Y<t)\le q<r$, contradicting
$\Pp(Z\le Q_r(Z))\ge r$.
\end{proof}

\subsection{Proof of \texorpdfstring{\cref{prop:sm025-interval}}{Proposition \getrefnumber{prop:sm025-interval}}}
\label{app:proof-sm025-interval}

\begin{proof}
For $0<\theta<1$,
\[
 M_n(\mathbf{x}/\theta)\le q
 \quad\Longleftrightarrow\quad
 K_n(\mathbf{x}/\theta)\le q\ \text{ or }\ \theta\le d_n(\mathbf{x}).
\]
This gives the lower endpoint, and reflection gives the upper endpoint.
Validity of $M_n$ and the same union-bound argument as in
\cref{thm:CI} give coverage at least $1-2q=0.95$. Domination gives
the inclusion, with the boundary convention unchanged.

For nonemptiness, put $d=d_n(\mathbf{x})\le1/56$. If $d>0$, at least $n-1$
observations satisfy $x_i\ge d$, so at most one coordinate of
$(\one-\mathbf{x})/(1-d)$ exceeds one. That coordinate is at most $1/(1-d)$.
Monotonicity and the neutral-coordinate identity therefore give
\[
 K_n\!\left(\frac{\one-\mathbf{x}}{1-d}\right)
 \ge K_n\!\left(\frac1{1-d},1,\ldots,1\right)
 =1-d\ge\frac{55}{56}>q.
\]
Thus $U_q(\mathbf{x})>d$. The same inequality holds when $d=0$, since
$U_q(\mathbf{x})>0$. Reflection gives $L_q(\mathbf{x})<1-d_n(\one-\mathbf{x})$.
Together with $L_q(\mathbf{x})<U_q(\mathbf{x})$ from \cref{thm:CI} and
\[
 d_n(\mathbf{x})+d_n(\one-\mathbf{x})\le\frac1{28}<1,
\]
these are the four strict comparisons needed for $L_M(\mathbf{x})<U_M(\mathbf{x})$.
\end{proof}

\subsection{Proof of \texorpdfstring{\cref{thm:sm025-length}}{Theorem \getrefnumber{thm:sm025-length}}}
\label{app:proof-sm025-length}

\begin{proof}
Take
\[
 \mathbf{x}^0=\frac1{57}(a_n,1,\ldots,1,56),\qquad \theta_0=\frac1{57}.
\]
Then $d_n(\mathbf{x}^0)=\theta_0$, whereas \eqref{eq:sm025-witness} gives
$K_n(\mathbf{x}^0/\theta_0)>q$. The original lower-tail inversion therefore
retains $\theta_0$, and
\[
 L_q(\mathbf{x}^0)<\theta_0=L_M(\mathbf{x}^0).
\]
Since both intervals are nonempty and $U_M\le U_q$, their length
difference at $\mathbf{x}^0$ is at least $\theta_0-L_q(\mathbf{x}^0)>0$.

The functions $L_q$ and $U_q$ are continuous; in fact, they are
$1$-Lipschitz in the sup norm. To see this, couple their defining
Dirichlet averages with the same Dirichlet vector. The resulting
random variables differ by at most $\|\mathbf{x}-\mathbf{y}\|_\infty$, so their
quantiles differ by at most the same amount. Order statistics are
continuous, hence so are $d_n,L_M,U_M$. The positive length
difference at the interior point $\mathbf{x}^0$ consequently persists, with
a fixed positive lower bound, on a sufficiently small open neighborhood
$V_n$. Finally, the length difference is nonnegative for every sample,
so its expectation is at least $\varepsilon_n\Pp(\mathbf{X}\in V_n)>0$.
\end{proof}

\subsection{Proof of \texorpdfstring{\cref{thm:width}}{Theorem \getrefnumber{thm:width}}}
\label{app:proof-width}
\label{app:proof-dirCLT}

We first record the quantile limit used for the conditional Dirichlet averages.

\begin{lemma}
\label{lem:dirCLT}
For each $m$, let $a_{1,m},\ldots,a_{m,m}$ be deterministic and uniformly
bounded.  Define
\[
\bar a_m=\frac1m\sum_{j=1}^m a_{j,m},
\quad
s_m^2=\frac1m\sum_{j=1}^m(a_{j,m}-\bar a_m)^2,
\]
and suppose $s_m\to s>0$.  If
$\mathbf D^{(m)}\sim\Dir(1,\ldots,1)$ has $m$ coordinates and
$W_m=\sum_{j=1}^m a_{j,m}D_j^{(m)}$, then, for every fixed $q\in(0,1)$,
\begin{equation}\label{eq:quantileCLT}
Q_q(W_m)
=
\bar a_m+\frac{s_m}{\sqrt{m+1}}\{z_q+o(1)\}.
\end{equation}
\end{lemma}

\begin{proof}
Let $E_1,\ldots,E_m$ be iid unit exponentials and
$D_j^{(m)}=E_j/\sum_{\ell=1}^mE_\ell$.  Since
$\sum_j(a_{j,m}-\bar a_m)=0$,
\[
W_m-\bar a_m
=
\frac{\sum_{j=1}^m(a_{j,m}-\bar a_m)(E_j-1)}
     {\sum_{j=1}^mE_j}.
\]
Set
\[
b_{j,m}=\frac{a_{j,m}-\bar a_m}{s_m\sqrt m}.
\]
Then $\sum_jb_{j,m}^2=1$ and $\max_j|b_{j,m}|\to0$ by boundedness and
$s_m\to s>0$.  The Lindeberg central limit theorem gives
\[
\sum_{j=1}^m b_{j,m}(E_j-1)\Rightarrow N(0,1),
\]
while $m^{-1}\sum_jE_j\to1$ in probability.  Indeed,
\[
\frac{\sqrt{m+1}(W_m-\bar a_m)}{s_m}
=
\sqrt{\frac{m+1}{m}}\,
\frac{m}{\sum_{j=1}^mE_j}
\sum_{j=1}^m b_{j,m}(E_j-1),
\]
and therefore Slutsky's theorem gives
\[
\frac{\sqrt{m+1}(W_m-\bar a_m)}{s_m}
\Rightarrow N(0,1).
\]
The limiting CDF is continuous and strictly increasing, so convergence of each
fixed quantile follows, proving \eqref{eq:quantileCLT}.
\end{proof}

\begin{proof}[Proof of \cref{thm:width}]
For the lower endpoint, the knots are $(0,X_1,\ldots,X_n)$.  Their empirical
mean and variance are
\begin{align}
\bar a_{\mathrm L,n}
&=\frac{n\bar X_n}{n+1},
\label{eq:meanL}\\
s_{\mathrm L,n}^2
&=\frac{n}{n+1}\widehat\sigma_n^2
+\frac{n}{(n+1)^2}\bar X_n^2.
\label{eq:varL}
\end{align}
For the upper endpoint, the knots are $(1,X_1,\ldots,X_n)$, so
\begin{align}
\bar a_{\mathrm U,n}
&=\frac{1+n\bar X_n}{n+1},
\label{eq:meanU}\\
s_{\mathrm U,n}^2
&=\frac{n}{n+1}\widehat\sigma_n^2
+\frac{n}{(n+1)^2}(1-\bar X_n)^2.
\label{eq:varU}
\end{align}

On an event of probability one,
$\bar X_n\to\mu$ and $\widehat\sigma_n^2\to\sigma^2$.  Hence, by
\eqref{eq:varL}--\eqref{eq:varU},
$s_{\mathrm L,n}\to\sigma$ and $s_{\mathrm U,n}\to\sigma$.  Apply
\cref{lem:dirCLT} pathwise to the lower and upper knot arrays.  Since each
array has $n+1$ knots,
\begin{align*}
L_{\alpha_{\mathrm L}}
&=
\frac{n\bar X_n}{n+1}
+\frac{s_{\mathrm L,n}}{\sqrt{n+2}}z_{\alpha_{\mathrm L}}
+o(n^{-1/2}),\\
U_{\alpha_{\mathrm U}}
&=
\frac{1+n\bar X_n}{n+1}
+\frac{s_{\mathrm U,n}}{\sqrt{n+2}}z_{1-\alpha_{\mathrm U}}
+o(n^{-1/2}).
\end{align*}
The differences between the displayed centers and $\bar X_n$ are $O(n^{-1})$,
and $s_{\mathrm L,n}-\widehat\sigma_n=o(1)$ and
$s_{\mathrm U,n}-\widehat\sigma_n=o(1)$.  This proves
\eqref{eq:Lasymp}--\eqref{eq:Uasymp}, and subtraction gives
\eqref{eq:widthlimit}.

Finally, the studentized central limit theorem and the endpoint expansions give
\[
\Pp\{L_{\alpha_{\mathrm L}}<\mu< U_{\alpha_{\mathrm U}}\}
\to
\Pp\{z_{\alpha_{\mathrm U}}< Z<  z_{1-\alpha_{\mathrm L}}\}
=1-\alpha_{\mathrm L}-\alpha_{\mathrm U},
\]
where $Z\sim N(0,1)$.
\end{proof}

\subsection{Proof of \texorpdfstring{\cref{thm:ext-md}}{Theorem \getrefnumber{thm:ext-md}}}
\label{app:proof-ext-md}

\begin{proof}
We first give a deterministic-array argument for the conditional Dirichlet
tails. For $N$ knots $a_{1,N},\ldots,a_{N,N}\in[0,1]$, let
$\bar a_N=N^{-1}\sum_j a_{j,N}$ and
$s_N^2=N^{-1}\sum_j(a_{j,N}-\bar a_N)^2$, and suppose
$\liminf_N s_N^2>0$. With independent unit exponential variables $E_j$,
put $W_N=\sum_j a_{j,N}E_j/\sum_jE_j$.
For any positive sequence $t_N\to0$ with $Nt_N^2\to\infty$, we prove
\begin{equation}\label{eq:ext-md-tail}
 \log\Pp(W_N\le\bar a_N-t_N)
 =-\frac{Nt_N^2}{2s_N^2}\{1+o(1)\}.
\end{equation}

Suppress the subscript $N$ and write $b_j=a_j-\bar a+t$ and
$Y=\sum_j b_jE_j$. Then $\E Y=Nt$ and
$\{W_N\le\bar a-t\}=\{Y\le0\}$.
Uniformly for $|\lambda|=O(t)$, the centered log moment-generating function is
\[
 \psi_N(\lambda)=\log\E e^{\lambda(Y-Nt)}
 =\sum_j\{-\log(1-\lambda b_j)-\lambda b_j\}
 =\frac{N\lambda^2}{2}(s^2+t^2)+O(N|\lambda|^3).
\]
For $\lambda=-t/s^2$, Chernoff's inequality gives
\[
 \log\Pp(Y\le0)\le\psi_N(\lambda)+\lambda Nt
 =-\frac{Nt^2}{2s^2}\{1+o(1)\}.
\]

For the reverse bound, choose $\eta_N\downarrow0$ such that
$\eta_N\sqrt N t_N\to\infty$. Under the exponential tilt
\[
 \frac{\dd\Pp_\lambda}{\dd\Pp}
 =\exp\{\lambda(Y-Nt)-\psi_N(\lambda)\},
\]
choose $\lambda_N<0$ to make $\E_{\lambda_N}Y=-\eta_NNt_N$.
The expansion of $\psi_N'$ and the positive lower bound on $s_N^2$
give such a choice with
$\lambda_N=-(1+\eta_N)t_N/s_N^2+O(t_N^2)$.
Also $\operatorname{Var}_{\lambda_N}(Y)=\psi_N''(\lambda_N)=O(N)$.
Thus Chebyshev's inequality implies
\[
 \Pp_{\lambda_N}\{-2\eta_NNt_N\le Y\le0\}\longrightarrow1.
\]
On this event, uniformly,
\[
 -\lambda_N(Y-Nt_N)+\psi_N(\lambda_N)
 =-\frac{Nt_N^2}{2s_N^2}\{1+o(1)\}.
\]
Undoing the tilt proves the reverse inequality in
\eqref{eq:ext-md-tail}. Replacing the knots by $1-a_{j,N}$ gives
the matching upper-tail estimate.

Now work on the probability-one event where
$\bar X_n\to\mu$ and $\widehat\sigma_n^2\to\sigma^2>0$.
Apply \eqref{eq:ext-md-tail} to the $N=n+1$ lower knots
$(0,X_1,\ldots,X_n)$. For each fixed $\delta\in(0,1)$, use
\[
 t_N^{\pm}=(1\pm\delta)\sqrt{\frac{2s_{\mathrm L,n}^2h_n}{n+1}}.
\]
The two conditional tail probabilities lie on opposite sides of
$e^{-h_n}$ for all sufficiently large $n$. They therefore bracket the
lower quantile. Letting $\delta\downarrow0$ gives its expansion around
$\bar a_{\mathrm L,n}$. The upper knots
$(1,X_1,\ldots,X_n)$ give the upper endpoint in the same way.
By \eqref{eq:meanL}--\eqref{eq:varU}, the knot means differ from
$\bar X_n$ by $O(n^{-1})$ and the knot variances differ from
$\widehat\sigma_n^2$ by $O(n^{-1})$. These discrepancies are negligible
relative to $\sqrt{h_n/n}$, which proves the endpoint expansions.

We now verify the moderate-deviation optimality conditions in
\cref{sec:moderate-deviations}. Let $\widetilde I_n$ satisfy
\eqref{eq:ext-md-coverage}. In particular, each one-sided error
probability is at most $2e^{-h_n}$ under every iid law on $[0,1]$.
Fix $\varepsilon\in(0,1)$
and choose $d\in(1-\varepsilon,1)$. Set
\[
 r_n=\sigma\sqrt{\frac{2h_n}{n}},
 \qquad
 \theta_n=\frac{d r_n}{\sigma^2}.
\]
For all sufficiently large $n$, define a probability law $\Pp_n$
on $[0,1]$ by
\[
 \frac{\dd\Pp_n}{\dd\Pp}(x)=1-\theta_n(x-\mu).
\]
The density is positive because $\theta_n\to0$, and its
$\Pp$-integral is one. Its mean is
\[
 \E_{\Pp_n}X=\mu-\theta_n\sigma^2=\mu-d r_n.
\]
Under $\Pp^n$, let
\[
 \Lambda_n
 =\log\frac{\dd\Pp^n}{\dd\Pp_n^n}
 =-\sum_{i=1}^n\log\{1-\theta_n(X_i-\mu)\}.
\]
Since $X_i\in[0,1]$, Taylor expansion gives
\[
 \E_{\Pp^n}\Lambda_n
 =\frac{n\theta_n^2\sigma^2}{2}+O(n\theta_n^3)
 =d^2h_n+o(h_n),
 \qquad
 \operatorname{Var}_{\Pp^n}(\Lambda_n)
 =O(n\theta_n^2)=O(h_n).
\]
Hence $\Lambda_n/h_n\to d^2$ in probability.

Define
\[
 A_n=\{\widetilde L_n>\mu-(1-\varepsilon)r_n\}.
\]
Because $d>1-\varepsilon$,
\[
 A_n\subseteq\{\widetilde L_n>\E_{\Pp_n}X\},
 \qquad
 \Pp_n^n(A_n)\le2e^{-h_n}.
\]
Choose $b\in(d^2,1)$. Changing measure yields
\begin{align*}
 \Pp^n(A_n)
 &\le \Pp^n\{\Lambda_n>b h_n\}
      +e^{b h_n}\Pp_n^n(A_n)\\
 &\le o(1)+2e^{-(1-b)h_n}
 \longrightarrow0.
\end{align*}
Applying the same argument to $1-X_i$ gives
\[
 \Pp^n\{\widetilde U_n<\mu+(1-\varepsilon)r_n\}\longrightarrow0.
\]
These are the endpoint bounds in \eqref{eq:ext-md-endpoint-bound}.
Outside the union of the two events,
$\widetilde U_n-\widetilde L_n\ge2(1-\varepsilon)r_n$,
so the width bound follows by the union bound.
Finally, the Gaffke expansions and
$\widehat\sigma_n\to\sigma$ almost surely give
\[
 \frac{U_{e^{-h_n}}-L_{e^{-h_n}}}{2r_n}
 \longrightarrow1
 \qquad\text{almost surely}.
\]
Moreover,
\[
 \frac{\E[(\bar X_n-\mu)^2]}{r_n^2}=\frac1{2h_n}\longrightarrow0.
\]
Thus $\bar X_n-\mu=o_{\Pp}(r_n)$, and the endpoint expansions imply
\eqref{eq:ext-md-attainment} with $L_n=L_{q_n}(\mathbf X)$ and
$U_n=U_{q_n}(\mathbf X)$. The closed Gaffke interval belongs to the
comparison class by \cref{thm:CI}: taking the closure preserves its
endpoints and width and cannot reduce coverage. This proves
moderate-deviation optimality.
\end{proof}

\subsection{Proof of \texorpdfstring{\cref{thm:ext-ld}}{Theorem \getrefnumber{thm:ext-ld}}}
\label{app:proof-ext-ld}

The proof uses the bounded-support KL-inf duality of
\citet{HondaTakemura2010}, in the two-sided form
\begin{equation}\label{eq:ext-duality}
 d(\Pp,m)=\sup_{-1/(1-m)\le\lambda\le1/m}
       \int\log\{1+\lambda(x-m)\}\,\dd\Pp(x).
\end{equation}
We include a proof of this identity and the properties of $d$ used in
\cref{sec:large-deviations}.

\begin{proof}
For a law $\Pp$ on $[0,1]$ and $m\in(0,1)$, write
$\mu_{\Pp}=\int x\,\dd\Pp(x)$ and
\[
 f_{\lambda,m}(x)=1+\lambda(x-m),\qquad
 \Lambda_m=[-1/(1-m),1/m].
\]
We use $\log0=-\infty$ and first prove \eqref{eq:ext-duality}.
To prove the lower bound, let $\Pp'$ have mean $m$. Since
$\int f_{\lambda,m}\,\dd\Pp'=1$, the relative-entropy inequality applied to the
probability law $f_{\lambda,m}\,\dd\Pp'$ gives
\[
 \operatorname{KL}(\Pp\|\Pp')\ge\int\log f_{\lambda,m}\,\dd\Pp.
\]
For completeness, if $q$ is the density with respect to $\Pp$ of the
absolutely continuous part of $\Pp'$, Jensen's inequality gives
$\int\log(f_{\lambda,m}q)\,\dd\Pp\le\log\int f_{\lambda,m}q\,\dd\Pp\le0$.
Since $\operatorname{KL}(\Pp\|\Pp')=-\int\log q\,\dd\Pp$, this proves the inequality;
cases with infinite relative entropy or right-hand side $-\infty$ are
immediate.

For the reverse bound, first suppose $m<\mu_{\Pp}$. Negative $\lambda$ cannot
improve on zero, since $\log(1+z)\le z$. The objective is upper
semicontinuous on $[0,1/m]$, has positive derivative at zero, and is
bounded above. It therefore has a maximizer $\lambda_*>0$.
If $\lambda_*<1/m$, its derivative is zero, so
\[
 \int\frac{x-m}{f_{\lambda_*,m}(x)}\,\dd\Pp(x)=0,
 \qquad
 \int\frac{1}{f_{\lambda_*,m}(x)}\,\dd\Pp(x)=1.
\]
Thus $\dd\Pp'=\dd\Pp/f_{\lambda_*,m}$ is a probability law of mean $m$ and attains
equality. If $\lambda_*=1/m$, finiteness of the objective implies
$\Pp\{0\}=0$. The left derivative condition gives
$m\int x^{-1}\,\dd\Pp(x)\le1$; this follows also by taking the limit of the
interior derivatives, whose negative parts converge monotonically.
Consequently,
\[
 \dd\Pp'(x)=\frac{m}{x}\,\dd\Pp(x)
       +\left(1-m\int\frac{1}{x}\,\dd\Pp(x)\right)\dd\mathbb{D}_0(x)
\]
is a probability law of mean $m$ and again attains equality.
Reflection proves the case $m>\mu_{\Pp}$, and $\Pp'=\Pp$ handles $m=\mu_{\Pp}$.
This proves \eqref{eq:ext-duality}, including point-mass laws.

The function $d(\Pp,\cdot)$ is finite and convex on $(0,1)$. Indeed, mixtures
of $\Pp$ with $\mathbb{D}_0$ or $\mathbb{D}_1$ show finiteness, and convexity follows
by mixing feasible laws and using convexity of
$\operatorname{KL}(\Pp\|\cdot)$. It is therefore continuous there.
When $0<\mu_{\Pp}<1$, it vanishes only at $\mu_{\Pp}$: the derivative at
$\lambda=0$ in \eqref{eq:ext-duality} is $\mu_{\Pp}-m$.
Convexity now makes it strictly decreasing below $\mu_{\Pp}$ and strictly
increasing above $\mu_{\Pp}$. It diverges at both endpoints. For example,
choose $\delta>0$ with $\Pp\{X\ge\delta\}>0$ and substitute
$\lambda=1/(2m)$ in \eqref{eq:ext-duality} to obtain divergence as
$m\downarrow0$; reflection handles $m\uparrow1$.
Thus the roots defining $l_c(\Pp)$ and $u_c(\Pp)$ exist uniquely for every
$c>0$.

Write $\widehat{\Pp}_n=n^{-1}\sum_{i=1}^n\mathbb{D}_{X_i}$ for the empirical law.
We next prove convergence of the empirical information profiles. On the
event of probability one on which $\widehat{\Pp}_n$ converges weakly to $\Pp$,
fix $m\in(0,1)$. For every $\lambda$ in the interior of $\Lambda_m$,
$\log f_{\lambda,m}$ is bounded and continuous, so its empirical integral
converges to its $\Pp$-integral. Concavity of the objective and its value zero
at $\lambda=0$ allow a boundary maximizer to be approached from the
interior. Hence
\[
 \liminf_n d(\widehat{\Pp}_n,m)\ge d(\Pp,m).
\]
For the upper bound, take a subsequence attaining the limsup and choose
maximizing $\lambda_n\in\Lambda_m$. Passing to a further subsequence gives
$\lambda_n\to\lambda_*$. For every $M>0$,
the function
\[
 (\lambda,x)\longmapsto
 \max\{\log f_{\lambda,m}(x),-M\}
\]
is continuous on the compact set $\Lambda_m\times[0,1]$. Therefore
\[
 \limsup_n d(\widehat{\Pp}_n,m)
 \le\int\max\{\log f_{\lambda_*,m}(x),-M\}\,\dd\Pp(x).
\]
Letting $M\to\infty$ proves the reverse inequality. These arguments hold
for every $m$ on the same weak-convergence event. Since the profiles are
finite convex functions, their secant-slope bounds upgrade pointwise
convergence to uniform convergence on every compact subinterval of
$(0,1)$.

Define the one-sided profiles by
\[
 d_-(\Pp,m)=\inf_{\Pp':\,\int x\,\dd\Pp'\le m}\operatorname{KL}(\Pp\|\Pp'),
 \qquad
 d_+(\Pp,m)=\inf_{\Pp':\,\int x\,\dd\Pp'\ge m}\operatorname{KL}(\Pp\|\Pp').
\]
By the monotonicity just established, $d_-(\Pp,m)$ equals $d(\Pp,m)$ below
$\mu_{\Pp}$ and zero at or above $\mu_{\Pp}$; the reflected statement holds for
$d_+$; the same descriptions follow directly from
\eqref{eq:ext-duality} for $\Pp=\mathbb{D}_0$ or $\Pp=\mathbb{D}_1$.
Equivalently, their dual domains in \eqref{eq:ext-duality} are
$[0,1/m]$ and $[-1/(1-m),0]$, respectively. To relate these profiles
to the optimized-betting statistic, define, for $0<m<1$,
\[
 p_{\mathrm{bet},\le}(m;\mathbf X)
 :=p_{\mathrm{bet}}(\mathbf X/m),\qquad
 p_{\mathrm{bet},\ge}(m;\mathbf X)
 :=p_{\mathrm{bet}}((\one-\mathbf X)/(1-m)).
\]
Writing $t\in[0,1]$ for the betting fraction, the substitutions
$\lambda=t/m$ and $\lambda=-t/(1-m)$ in the two optimized products
give the exact identities
\begin{equation}\label{eq:ext-betting-exact}
 \begin{aligned}
 p_{\mathrm{bet},\le}(m;\mathbf X)
 &=e^{-n d_-(\widehat{\Pp}_n,m)},\\
 p_{\mathrm{bet},\ge}(m;\mathbf X)
 &=e^{-n d_+(\widehat{\Pp}_n,m)}.
 \end{aligned}
\end{equation}
Let $(L_{\mathrm{bet},n},U_{\mathrm{bet},n})$ and
$(L_{\mathrm{SP},n},U_{\mathrm{SP},n})$ denote the optimized-betting and
SymPol inversion endpoints at the same tail levels, with the convention
of \eqref{eq:strict-inversion}.
By \eqref{eq:ext-betting-exact}, convergence of the empirical information
profiles, and strict monotonicity of $d(\Pp,\cdot)$ on each side of $\mu$,
bracketing each limiting endpoint by two fixed points gives
\[
 L_{\mathrm{bet},n}\longrightarrow l_{c_{\mathrm L}}(\Pp),\qquad
 U_{\mathrm{bet},n}\longrightarrow u_{c_{\mathrm U}}(\Pp)
 \quad\text{almost surely}.
\]
The pointwise comparison in \cref{sec:domination} gives
\[
 L_{\mathrm{bet},n}\le L_{\mathrm{SP},n}
 \le L_{\alpha_{\mathrm L,n}}(\mathbf X),\qquad
 U_{\alpha_{\mathrm U,n}}(\mathbf X)
 \le U_{\mathrm{SP},n}\le U_{\mathrm{bet},n}.
\]
Consequently, the Gaffke endpoints satisfy
\begin{equation}\label{eq:ext-outer-limits}
 \liminf_n L_{\alpha_{\mathrm L,n}}\ge l_{c_{\mathrm L}}(\Pp),
 \qquad
 \limsup_n U_{\alpha_{\mathrm U,n}}\le u_{c_{\mathrm U}}(\Pp)
 \quad\text{almost surely}.
\end{equation}

Finally, consider any lower bound $\widetilde L_n$ satisfying the first
inequality in \eqref{eq:ext-ld-coverage}. Fix
$m\in(l_{c_{\mathrm L}}(\Pp),\mu_{\Pp})$ and choose a law $\Pp'$ of mean $m$ with
$D=\operatorname{KL}(\Pp\|\Pp')<c_{\mathrm L}$. Such a law exists by the
definition of $d(\Pp,m)$. For
$A_n=\{\widetilde L_n>m\}$, let
$Z_n=\log(\dd\Pp^n/\dd(\Pp')^n)$. Finite relative entropy ensures that the
log-likelihood ratio is integrable under $\Pp$, so $Z_n/n\to D$ almost
surely. Choose $\eta>0$ such that $D+\eta<c_{\mathrm L}$. Then
\begin{align*}
 \Pp^n(A_n)
 &\le e^{n(D+\eta)}(\Pp')^n(A_n)
       +\Pp^n\{Z_n>n(D+\eta)\}\\
 &\le e^{n(D+\eta)}\alpha_{\mathrm L,n}+o(1)
 \longrightarrow0.
\end{align*}
Letting $m\downarrow l_{c_{\mathrm L}}(\Pp)$ proves, for every
$\varepsilon>0$,
\[
 \Pp^n\{\widetilde L_n>l_{c_{\mathrm L}}(\Pp)+\varepsilon\}\to0.
\]
Reflection proves the upper-endpoint bound under the second inequality
in \eqref{eq:ext-ld-coverage}, establishing
\eqref{eq:ext-ld-endpoint-bound}. The Gaffke endpoints satisfy
\eqref{eq:ext-ld-coverage} by the one-sided versions of \cref{thm:CI}.
Applying these bounds and combining them with \eqref{eq:ext-outer-limits} proves their convergence
in probability and large-deviation optimality as defined in
\cref{sec:large-deviations}. The endpoint ordering above
then implies that the SymPol endpoints converge in probability to the
same respective limits. Applying the two competitor endpoint bounds
with $\varepsilon/2$ and taking a union bound proves the width lower bound;
subtracting the Gaffke endpoint limits gives its limiting width.

For the equal-tail comparison under joint coverage $1-2q_n$, each
one-sided error is at most $2q_n$. The same change-of-measure argument
applies with $2q_n$ in place of the corresponding tail level, since
$\log(1/(2q_n))/n\to c$. This proves the claimed endpoint and width
bounds without a prescribed tail allocation.\qedhere
\end{proof}

\subsection{Extreme-confidence limits at fixed sample size}
\label{app:proof-ext-fixed}

We fix the sample size and let the tail levels tend to zero.
These limits concern the observed vector and do not require an iid sampling
assumption. For $\mathbf x\in[0,1]^n$, write
\[
 k_+=\#\{i:x_i>0\},\qquad k_-=\#\{i:x_i<1\},\qquad
 A_+(\mathbf x)=\prod_{x_i>0}x_i,\qquad
 A_-(\mathbf x)=\prod_{x_i<1}(1-x_i),
\]
with an empty product equal to one.

\begin{proposition}\label{prop:ext-fixed}
For every fixed $\mathbf x$, as $a\downarrow0$,
\begin{align}
 L_a(\mathbf x)
 &=\left\{\frac{aA_+(\mathbf x)}{\binom n{k_+}}\right\}^{1/k_+}
    \{1+O_{\mathbf x}(a^{1/k_+})\},\quad k_+>0,
 \label{eq:ext-lower}\\
 1-U_a(\mathbf x)
 &=\left\{\frac{aA_-(\mathbf x)}{\binom n{k_-}}\right\}^{1/k_-}
    \{1+O_{\mathbf x}(a^{1/k_-})\},\quad k_->0.
 \label{eq:ext-upper}
\end{align}
The constants implicit in $O_{\mathbf x}$ may depend on $\mathbf x$;
no uniformity over $\mathbf x$ is asserted.
If $k_+=0$, then $L_a=0$; if $k_-=0$, then $U_a=1$.
If all $x_i>0$, put $r_a=(a\prod_{i=1}^n x_i)^{1/n}$. Whenever
$r_a<\min_i x_i$,
\begin{equation}\label{eq:ext-interior}
 L_a(\mathbf x)=r_a.
\end{equation}  
For each $a\in(0,1)$, let $\ell_a:[0,1]^n\to[0,1]$ be a Borel function satisfying
$\Pp\{\ell_a(\mathbf X)>m\}\le a$ for every $m\in[0,1]$ and every
collection of independent $[0,1]$-valued observations with common mean $m$.
Under the condition in \eqref{eq:ext-interior}, $\ell_a(\mathbf x)\le r_a$.
More generally, for every fixed vector with $k_+>0$,
\begin{equation}\label{eq:ext-face-bound}
 \limsup_{a\downarrow0}\frac{\ell_a(\mathbf x)}{a^{1/k_+}}
 \le A_+(\mathbf x)^{1/k_+}.
\end{equation}
The reflected assertions hold for upper confidence bounds. Thus the interior
endpoints are pointwise optimal, and the boundary powers are sharp, within
this deterministic validity class.
\end{proposition}

\begin{proof}
For the interior assertions, suppose all $x_i>0$.
For $0<t<\min_i x_i$, \eqref{eq:product-rule} gives
$\Pp_{\mathbf D}(S_{\mathbf x}\le t)=t^n/\prod_{i=1}^n x_i$,
which proves \eqref{eq:ext-interior}.
If $\ell_a(\mathbf x)>r_a$, choose
$r_a<m<\min\{\ell_a(\mathbf x),\min_i x_i\}$ and take independent
observations with $\Pp(X_i=x_i)=m/x_i$ and $\Pp(X_i=0)=1-m/x_i$.
They all have mean $m$, while the vector $\mathbf x$ has probability
$m^n/\prod_{i=1}^n x_i>a$, contradicting validity.
Reflection proves the upper interior assertion.

For the boundary expansions and sharp-power bound, put $k=k_+>0$, and
relabel the positive coordinates as $x_1,\ldots,x_k$. Aggregating the
$n-k$ zero knots and the additional zero knot in $S_{\mathbf x}$ gives
Dirichlet parameters $(1,\ldots,1,n-k+1)$. Thus the marginal density of
the positive weights is
\[
 \frac{n!}{(n-k)!}\left(1-\sum_{i=1}^k d_i\right)^{n-k}
 \one\left\{d_i\ge0,\ \sum_{i=1}^k d_i\le1\right\}.
\]
For $0<t<\min_{i\le k}x_i$, the substitution $d_i=ty_i/x_i$ yields
\begin{align}
 \Pp_{\mathbf D}(S_{\mathbf x}\le t)
 &=\frac{n!t^k}{(n-k)!A_+(\mathbf x)}
 \int_{\substack{y_i\ge0\\\sum_i y_i\le1}}
 \left(1-t\sum_{i=1}^k\frac{y_i}{x_i}\right)^{n-k}
 \,\dd\mathbf y \notag\\
 &=\frac{\binom nk}{A_+(\mathbf x)}t^k\{1+O_{\mathbf x}(t)\}.
 \label{eq:ext-smallball}
\end{align}
The simplex has volume $1/k!$. Inverting \eqref{eq:ext-smallball} gives
\[
 L_a(\mathbf x)=
 \left\{\frac{aA_+(\mathbf x)}{\binom nk}\right\}^{1/k}
 \{1+O_{\mathbf x}(a^{1/k})\}.
\]
When $k=n$, the integrand is identically one, so the formula is exact
whenever its value is smaller than $\min_i x_i$. If $k=0$, then
$S_{\mathbf x}=0$ and $L_a(\mathbf x)=0$.
The upper-endpoint assertions follow by applying the same argument to
$1-T_{\mathbf x}=\sum_i(1-x_i)D_i$.

For the sharp-power assertion, let $\ell_a$ be a deterministic lower
bound valid at level $a$ for every common mean $m\in[0,1]$ and every
collection of independent observations on $[0,1]$ with that mean.
For $k>0$ and $0<m<\min_{i\le k}x_i$, put mass $m/x_i$ at $x_i$ and
the remaining mass at zero in each positive coordinate. In each zero
coordinate, use a Bernoulli$(m)$ law. These independent laws all have
mean $m$, and the probability of the particular ordered vector
$\mathbf x$ is
\begin{equation}
 \frac{m^k(1-m)^{n-k}}{A_+(\mathbf x)}.
 \label{eq:ext-face-atom}
\end{equation}
Fix $c>A_+(\mathbf x)^{1/k}$ and set $m=ca^{1/k}$. For all sufficiently
small $a$, \eqref{eq:ext-face-atom} is greater than $a$. Therefore
$\ell_a(\mathbf x)>ca^{1/k}$ would violate validity. Letting
$c\downarrow A_+(\mathbf x)^{1/k}$ proves the asserted limsup bound.
If $k=0$, validity under the all-zero law forces
$\ell_a(\mathbf0)=0$ for $a<1$. Reflection gives the corresponding upper
bound, including its value one at the all-one sample.
\end{proof}

The constants in \eqref{eq:ext-lower}--\eqref{eq:ext-upper} agree with
those of SymPol, by \eqref{eq:LSP}--\eqref{eq:USP}. On $(0,1)^n$,
the two intervals eventually have exactly the same endpoints as both tail
levels decrease to zero. The optimality statement compares separately
calibrated deterministic bounds valid under independence with a common mean;
it does not assert optimal leading constants on boundary faces or among
procedures required to be valid only under iid sampling.
The threshold for interior optimality depends on $\mathbf x$, so these
conclusions are compatible with the improvements in
\cref{sec:CI-inadmissibility} at fixed tail levels.

The following example shows why iid validity alone does not imply the
pointwise sharp-constant conclusion of \cref{prop:ext-fixed}.
Fix $n=2$ and $0<r<s<1$. For
$0<a<1$, define
\begin{equation}
 \ell_a(x_1,x_2)=\sqrt{ar(2s-r)}\,
 \one\{\min\{x_1,x_2\}\ge r,\ \max\{x_1,x_2\}\ge s\}.
 \label{eq:ext-iid-bound}
\end{equation}
This deterministic, symmetric, coordinatewise nondecreasing lower bound
is valid at level $a$ under iid sampling from every law on $[0,1]$.
Indeed, write $\mu=\E[X_1]$, $p=\Pp(X_1\ge s)$, and
$q=\Pp(r\le X_1<s)$. Then $sp+rq\le\mu$, and independence gives
\[
 \Pp\{\min\{X_1,X_2\}\ge r,\ \max\{X_1,X_2\}\ge s\}=p^2+2pq.
\]
The identity
\[
 (sp+rq)^2-r(2s-r)(p^2+2pq)=\{(s-r)p-rq\}^2
\]
therefore bounds this probability by $\mu^2/\{r(2s-r)\}$. If
$\mu<\sqrt{ar(2s-r)}$, the failure probability of
\eqref{eq:ext-iid-bound} is at most $a$; otherwise it is zero.

At the vector $(r,s)$, the Gaffke lower endpoint equals $\sqrt{ars}$
for all sufficiently small $a$. The bound in \eqref{eq:ext-iid-bound}
is larger by the constant factor
\[
 \sqrt{\frac{2s-r}{s}}>1.
\]
This does not give a uniform improvement over Gaffke's. It shows that the
pointwise sharp-constant conclusion of \cref{prop:ext-fixed} does not
extend to procedures required to be valid only in the iid model.

\subsection{Numerical effect of the explicit interval refinement}
\label{app:sm025-simulation}

We quantify the width reduction from the merger $M_n$ in
\cref{thm:sm025-seed}, using the endpoints in
\cref{prop:sm025-interval}. We compare its $95\%$ equal-tailed interval
$I_M$ with the Gaffke interval $I_G$ on the same samples. For each
$n\in\{3,5,10\}$, we generate $10{,}000$ independent samples from each
of three continuous laws and a deliberately constructed three-point law
$\Pp_n$ with
\[
 \Pp_n\{a_n/57\}=\Pp_n\{56/57\}=\frac1n,
 \qquad \Pp_n\{1/57\}=\frac{n-2}{n}.
\]
The latter depends on $n$ and assigns positive probability to the witness
sample in the proof of \cref{thm:sm025-length}; it is included to exhibit
nonzero gains, rather than as a representative sampling model.

With $q=1/40$, write $W_G=U_q-L_q$ and
$\Delta=(U_q-L_q)-(U_M-L_M)\ge0$. The endpoints are computed by the
B-spline method in \cref{sec:computation}, with absolute root tolerance
$10^{-13}$. We use NumPy's PCG64 generator with base seed $20261018$
and separate streams indexed by the model and $n$.
\Cref{tab:sm025-simulation} reports the mean Gaffke width, the frequency
of $\Delta>10^{-10}$, and the paired mean reduction with its Monte Carlo
standard error. For $\Pp_n$, enumeration of all multinomial count vectors
also gives mean reductions of $1.61\times10^{-5}$,
$2.51\times10^{-6}$ and $4.11\times10^{-7}$, respectively, consistent
with the simulated values with Monte Carlo errors.

\begin{table}[!htbp]
\centering
\small
\caption{Width gains from the explicit refinement $M_n$ at the
$95\%$ equal-tailed level, with 10,000 replications per row.
``Improved'' is the percentage of samples with $\Delta>10^{-10}$.
The final column reports $\overline{\Delta}/\overline{W_G}$,
the reduction in mean width.
Monte Carlo standard errors are in parentheses on the same
$10^6$ scale as the mean gain.
Zeros indicate no observed gain above the numerical threshold.}
\label{tab:sm025-simulation}
\begin{tabular}{lrrrrr}
\toprule
Sampling law & $n$ & $\overline{W_G}$
& \shortstack{Improved \\ samples (\%)}
& \shortstack{$10^6\overline{\Delta}$\\(MCSE)}
& \shortstack{Reduction in\\mean width} \\
\midrule
$\mathrm{Unif}(0,1)$
 & 3  & 0.7509 & 0.00 & 0 (0) & 0 \\
 & 5  & 0.6036 & 0.00 & 0 (0) & 0 \\
 & 10 & 0.4240 & 0.00 & 0 (0) & 0 \\
\addlinespace
$\mathrm{Beta}(1/2,1/2)$
 & 3  & 0.7780 & 0.00 & 0 (0) & 0 \\
 & 5  & 0.6455 & 0.00 & 0 (0) & 0 \\
 & 10 & 0.4757 & 0.00 & 0 (0) & 0 \\
\addlinespace
$\mathrm{Beta}(2,5)$
 & 3  & 0.7211 & 0.00 & 0 (0) & 0 \\
 & 5  & 0.5479 & 0.00 & 0 (0) & 0 \\
 & 10 & 0.3490 & 0.00 & 0 (0) & 0 \\
\addlinespace
Constructed $\Pp_n$
 & 3  & 0.8241 & 23.02 & 16.704 (0.305) & $2.03 \times 10^{-5}$ \\
 & 5  & 0.6618 & 16.81 &  2.440 (0.054) & $3.69 \times 10^{-6}$  \\
 & 10 & 0.4180 & 14.80 &  0.403 (0.010) & $9.64\times 10^{-7}$  \\
\bottomrule
\end{tabular}
\end{table}

\FloatBarrier

No shortening above the numerical threshold was observed under the
continuous laws. This does not contradict the positive expected reduction
under uniform sampling in \cref{thm:sm025-length}: the improving samples
can be too rare to appear in this experiment. Under $\Pp_n$, shortening
occurs regularly but is small in magnitude; the largest relative reduction
in mean width is about $0.0020\%$, at $n=3$. Thus these results illustrate
how strict finite-sample inadmissibility can coexist with small practical
gains for this explicit refinement. They do not assess the admissible
completions or other valid improvements.

\end{document}